\documentclass[11pt, letterpaper]{amsart}

\usepackage{amsmath,amsfonts,amssymb,amsthm,tikz,hyperref,diagbox,bbm}
\usetikzlibrary{patterns}
\usepackage{ytableau}
\ytableausetup{smalltableaux, centertableaux}
\usepackage[square,sort,comma,numbers]{natbib}
\usepackage{graphicx, subcaption}
\usepackage{stmaryrd} 
\usepackage{todonotes}

\colorlet{myblue}{cyan!50}
\colorlet{myorange}{orange!50}

\usepackage{genyoungtabtikz}
\YFrench
\Yboxdim{0.5cm}
\newcommand{\Yred}{\Yfillcolour{red}}
\newcommand{\Ycyan}{\Yfillcolour{myblue}}

\newtheorem{thm}{Theorem}[section]
\newtheorem{lem}[thm]{Lemma}
\newtheorem{cor}[thm]{Corollary}
\newtheorem{prop}[thm]{Proposition}

\theoremstyle{definition}
\newtheorem{defn}[thm]{Definition}

\newtheorem{ex}[thm]{Example}
\newtheorem{rem}[thm]{Remark}

\newcommand{\ds}{\displaystyle}
\newcommand{\bb}[1]{\mathbb{#1}}

\newcommand{\mb}[1]{\mathbf{#1}}
\newcommand{\x}{\mathbf{x}}
\newcommand{\y}{\mathbf{y}}

\newcommand{\wh}[1]{\widehat{#1}}

\newcommand{\mc}[1]{\mathcal{#1}}

\newcommand{\ov}[1]{\overline{#1}}

\newcommand{\tb}[1]{\textbf{#1}}

\newcommand{\U}{\mathcal{U}}
\newcommand{\D}{\mathcal{D}}
\DeclareMathOperator{\nw}{nw}

\newcommand{\id}{{\rm id}}

\newcommand{\Ret}{\mathbf{R}}
\newcommand{\Set}{\mathbf{S}}

\DeclareMathOperator{\shift}{shift}

\newcommand{\SYT}{{\rm SYT}}
\newcommand{\SSYT}{{\rm SSYT}}
\newcommand{\la}{\lambda}
\newcommand{\vp}{\varphi}
\newcommand{\dRSK}{qt{\rm RSK}^*}

\newcommand{\RS}{{\rm RS}}
\newcommand{\RSK}{{\rm RSK}}
\newcommand{\qrst}{q{\rm RS}t}

\newcommand{\col}{{\rm col}}
\newcommand{\row}{{\rm row}}
\newcommand{\hl}[1]{h_{#1}^\ell}
\newcommand{\ha}[1]{h_{#1}^a}

\newcommand{\lacell}{
\begin{tikzpicture}[scale=0.4]
\draw[thick, magenta, pattern=horizontal lines, pattern color=magenta] (0,0) rectangle (1,1);
\end{tikzpicture}
}
\newcommand{\mucell}{
\begin{tikzpicture}[scale=0.4]
\draw[thick, orange, fill=myorange] (0,0) rectangle (1,1);
	\node at (.5,.5) {$-$};
\end{tikzpicture}
}
\newcommand{\nucell}{
\begin{tikzpicture}[scale=0.4]
\draw[thick, blue, fill=myblue] (0,0) rectangle (1,1);
	\node at (.5,.5) {$+$};
\end{tikzpicture}
}
\newcommand{\rhocell}{
\begin{tikzpicture}[scale=0.4]
\draw[thick, green!70!black, pattern=vertical lines, pattern color=green!70!black] (0,0) rectangle (1,1);
\end{tikzpicture}
}

\newcommand{\deff}[1]{\emph{\color{cyan!80!blue!80!black} #1}} 

\newcommand{\circled}[1]{
\begin{tikzpicture}[scale=0.4]
\draw[thick, green!70!black, pattern=vertical lines, pattern color=green!70!black] (0,0) rectangle (1,1);
\end{tikzpicture}
}

\makeatletter
\newcommand\mathcircled[1]{
  \mathpalette\@mathcircled{#1}
}
\newcommand\@mathcircled[2]{%
  \tikz[baseline=(math.base)] \node[draw,circle,inner sep=1pt] (math) {$\m@th#1#2$};%
}
\makeatother

\title[$\dRSK$: A probabilistic dual RSK for Macdonald polynomials]{$\dRSK$: A probabilistic dual RSK correspondence for Macdonald polynomials}
\author[Frieden]{Gabriel Frieden}
\address[G.~Frieden]{McGill University, Montr\'eal, Canada}
\email{gabriel.frieden@mcgill.ca}
\author[Schreier-Aigner]{Florian Schreier-Aigner}
\address[F.~Schreier-Aigner]{University of Vienna, Austria}
\email{florian.schreier-aigner@univie.ac.at}

\begin{document}

\begin{abstract}
We introduce a probabilistic generalization of the dual\linebreak Robinson--Schensted--Knuth correspondence, called $\dRSK$, depending on two parameters $q$ and $t$. This correspondence extends the $\qrst$ correspondence, recently introduced by the authors, and allows the first tableaux-theoretic proof of the dual Cauchy identity for Macdonald polynomials. By specializing $q$ and $t$, one recovers the row and column insertion version of the classical dual RSK correspondence as well as of $q$- and $t$-deformations thereof which are connected to $q$-Whittaker and Hall--Littlewood polynomials. When restricting to Jack polynomials and $\{0,1\}$-matrices corresponding to words, we prove that the insertion tableaux obtained by $\dRSK$ are invariant under swapping letters in the input word.
 Our approach is based on Fomin's growth diagrams and the notion of probabilistic bijections. 
\end{abstract}

\maketitle

\tableofcontents
\section{Introduction}

The Robinson--Schensted--Knuth (RSK) correspondence is a bijection between matrices of nonnegative integers with finite support and pairs of semistandard Young tableaux of the same shape. It was introduced by Knuth \cite{Knuth70} and generalizes the Robinson-Schensted (RS) correspondence which was introduced by Robinson \cite{Robinson38} for permutations and independently by Schensted \cite{Schensted61} for words\footnote{We refer by ``RS'' to the special case of RSK for permutations and by ``RS for words'' to the special case of RSK for words.}. Both correspondences have many applications in combinatorics, representation theory, probability theory and algebraic geometry. In particular RSK gives a bijective proof of the Cauchy identity
\begin{equation}
\label{eq:cauchy}
\sum_\la s_\la(\x)s_\la(\y) = \prod_{i,j \geq 1} \frac{1}{1-x_iy_j},
\end{equation}
where the sum is over all partitions $\la$ and $s_\la(\mb{z})$ denotes the Schur function in the variables $\mb{z}=(z_1,z_2,\ldots)$. A closely related bijection is the dual RSK correspondence ($\RSK^*$) introduced by Knuth \cite{Knuth70}  which yields a bijective proof of the dual Cauchy identity
\begin{equation}
\label{eq:dual cauchy}
\sum_{\la}s_\la(\x) s_{\la^\prime}(\y) = \prod_{i,j \geq 1} (1+x_iy_j),
\end{equation}
where the sum is again over all partitions $\la$.

All of the above mentioned correspondences have been extended into various directions throughout the last decades. Among others several randomized generalizations of RS, RSK and $\RSK^*$ were introduced in
\cite{BorodinPetrov16, BufetovMatveev18, BufetovPetrov15, MatveevPetrov17, OConnellPei13, Pei14, Pei17}. These generalizations associate to each permutation or nonnegative integer matrix respectively a distribution on pairs of (semi)standard Young tableaux depending on a parameter $q$ or $t$ and thereby giving a proof of the (dual) Cauchy identity for $q$-Whittaker or Hall--Littlewood symmetric functions. Similar to the classical RSK algorithm, these randomized generalizations have many applications to probabilistic models, compare for example with \cite{BorodinPetrov16,BufetovMatveev18,BufetovPetrov15,MatveevPetrov17}. 

In a previous paper \cite{AignerFrieden22}, the authors introduced a randomized generalization for RS called $\qrst$ depending on two parameter $q$ and $t$. This generalization was designed to prove the squarefree part of the Cauchy identity for the Macdonald symmetric functions $P_\la(\x;q,t)$ and $Q_\la(\x;q,t)$. Analogously to Macdonald symmetric functions, which specialize to $q$-Whittaker, Hall--Littlewood and Schur symmetric functions, the $\qrst$ correspondence specializes to the corresponding randomized variations of RS for $q=0$,$t=0$, $q\rightarrow \infty$ or $t\rightarrow \infty$, and to the row- or column-insertion of RS itself for $q=t=0$ or $q=t\rightarrow \infty$ respectively.

The aim of this paper is to give a unifying generalization of both $\qrst$ and $\RSK^*$, called $\dRSK$, and thereby giving the first tableaux theoretic proof of the dual Cauchy identity for Macdonald polynomials
\begin{equation}
\label{eq:dual cauchy mac}
\sum_\la P_\la(\x;q,t) P_{\la^\prime}(\y;t,q) = \prod_{i,j}(1+x_iy_j).
\end{equation}
Although $\dRSK$ specializes to the dual row-insertion and dual column-insertion it is not a simple generalization of both algorithms but a ``superposition''. In particular the insertion of a sequence of integers $i_1 < \ldots < i_r$ is done in ascending order for dual column-insertion, in descending order for dual row-insertion, and so-to-speak simultaneously.

%

\subsection{Bijectivizing an algebraic proof}

A very concise proof of the classical (dual) Cauchy identity for Schur symmetric functions is based on the commutation relation of the up and (dual) down operators on Young's lattice. The framework of Fomin's growth diagrams \cite{Fomin95, Fomin86} allows one to translate a bijective proof of the commutation relation easily into a bijective proof of the (dual) Cauchy identity.
By defining the Macdonald polynomials $P_\la(\x;q,t)$ through their monomial expansion
\begin{equation}
\label{eq:macdonald monomial}
P_\la (\x; q,t) = \sum_{T \in \SSYT(\la)} \psi_T(q,t) \tb{x}^T,
\end{equation}
where $\psi_T(q,t)$ is a rational function on $q$ and $t$,
the (dual) Cauchy identity is again an immediate consequence of the commutation relation
\begin{equation}
\label{eq:dual commutation relation}
D^*_y(q,t) U_x(q,t) = (1+xy) U_x(q,t) D^*_y(q,t),
\end{equation}
where $ U_x(q,t)$ and $D^*_y(q,t)$ are the $(q,t)$-weighted extensions of the up and dual down operators on Young's lattice.
By restricting the above to the coefficient of $xy$, we obtain the commutation relation
\[
D_{q,t}U_{q,t}= \frac{1-t}{1-q}\id + U_{q,t}D_{q,t} \qquad \Leftrightarrow \qquad D_{q,t}U_{q,t}-U_{q,t}D_{q,t}= \frac{1-t}{1-q}\id,
\]
which was the fundamental commutation relation in the authors previous paper \cite{AignerFrieden22}. By further setting $q=t=0$ the above becomes the equation 
\[
 DU-UD=\id,
\]
which is the defining property for Young's lattice to be a differential poset, see \cite{Stanley88}.

  It is clear that no bijective proof of the commutation relation \eqref{eq:dual commutation relation} which preserves the $(q,t)$-weight can exist since it would imply a weight-preserving bijective proof of the dual Cauchy identity for Macdonald polynomials. However, as we see in Table~\ref{tab:explicit bijection} this is not possible. Our approach is nevertheless to prove \eqref{eq:dual commutation relation} ``as bijectively as possible'' by using the notion of \emph{probabilistic bijections} introduced in \cite{BufetovMatveev18, BufetovPetrov19}.
Let $X,Y$ be sets together with weight functions  $\omega$ and $\ov{\omega}$ respectively which satisfy
\begin{equation}
\label{eq:weighted gen fct}
\sum_{x \in X}\omega(x) = \sum_{y \in Y} \ov{\omega}(y).
\end{equation}
Instead of one weight preserving bijection between $X$ and $Y$, a probabilistic bijection consists of two families of ``probability distributions''\footnote{Note that the term ``probability distribution'' is used in a more general setting allowing them to have values in $\mathbb{Q}(q,t)$ instead of real values in $[0,1]$; nevertheless the sum of the entries of the ``probability distribution'' has to be $1$.}. For each $x \in X$ we have a ``forward'' probability distribution $\mc{P}(x \rightarrow y)$, which can be thought of as the probability of mapping $x$ to $y$, and for each $y \in Y$ a ``backward'' probability distribution $\ov{\mc{P}}(x \leftarrow y)$, the probability of mapping $y$ back to $x$. Probability distributions satisfying the equation
\begin{equation}
\label{eq:comp condition}
\omega(x) \mc{P}(x \rightarrow y) = \ov{\mc{P}}(x \leftarrow y) \ov{\omega}(y),
\end{equation}
are called \emph{probabilistic bijections}. It is immediate, however very significant, that the above compatibility condition between the probabilistic distributions and the weights implies \eqref{eq:weighted gen fct}.

Our combinatorial proof of the commutation relation \eqref{eq:dual commutation relation} boils down to finding, for partitions $\la,\rho$, a probabilistic bijection between two sets $\U^k(\la,\rho)$, $\D^k(\la,\rho)\cup \D^{k-1}(\la,\rho)$ of partitions together with weights coming from the definition of the $(q,t)$-up and $(q,t)$-dual down operator.
The main guiding principle for finding the explicit formulas for both the forward and backward probabilities 
$\mc{P}(x \rightarrow y)$ and $\ov{\mc{P}}(x \leftarrow y)$ in this setting was to use \eqref{eq:comp condition} as a defining relation and to distribute the factors in $\frac{\omega(x)}{\omega(y)}$ evenly in a minimal way such that both $\mc{P}$ and $\ov{\mc{P}}$ are homogeneous rational functions in certain parameters. The actual definition of the probabilities however is based on an intuitive geometric approach using the \emph{Quebecois notation} for Young diagrams, a variation of both the French and English notation, 
for which it takes a decent amount of work to  prove \eqref{eq:comp condition}.
On the other hand, our definition of the probabilities is not only very concise, but also naturally led us to rediscovering an extension of Lagrange interpolation for symmetric polynomials by Chen and Louck \cite{ChenLouck96} which immediately implies that the forward and backward probabilities are indeed probability distributions.

Using Fomin's growth diagrams allows us to lift our probabilistic bijection of the commutation relation to a probabilistic bijection of the dual Cauchy identity. The next table shows the forward probabilities for matrices and pairs of tableaux corresponding to the coefficient of $x_1^2x_2 y_1y_2y_3$ in \eqref{eq:dual cauchy mac}.  We can observe easily in this example that $\mc{P}(A \rightarrow P,Q)$ is a probability distribution since the sum over the probabilities in each row yields $1$.
\begin{table}[h]
\[
\begin{array}{l|ccc}
& \young(11,2)\; , \; \young(12,3) & \young(11,2)\; , \; \young(13,2) & \young(112)\; , \; \young(123)
 \\[12pt]
 \hline \\[-6pt]
\begin{pmatrix}
1&1&0 \\ 0&0&1
\end{pmatrix} &  t \dfrac{1-q^2}{1-q^2t} & 0 & \dfrac{1-t}{1-q^2 t}
 \\[18pt]
\begin{pmatrix}
1&0&1 \\ 0&1&0
\end{pmatrix} & \dfrac{(1-q)(1-t)}{(1-q t)(1-q^2t)} & t\dfrac{1-q}{1-q t} & q \dfrac{1-t}{1-q^2t}
 \\[18pt]
\begin{pmatrix}
0&1&1 \\ 1&0&0
\end{pmatrix} & q \dfrac{(1-q)(1-t)}{(1-q t)(1-q^2 t)} & \dfrac{1-q}{1-qt} & q^2 \dfrac{1-t}{1-q^2t} \\[12pt]
\hline \\[-6pt]
& \dfrac{(1-q^2)(1-q t^2)}{(1-q t)(1-q^2t)} & \dfrac{(1-q)(1-t^2)}{(1-t)(1-qt)} & \dfrac{(1-t)(1-q^3)}{(1-q)(1-q^2t)}
\end{array}
\]
\caption{\label{tab:explicit bijection}
The forward probabilities $\mc{P}(A \rightarrow P,Q)$ where all matrices have weight $x_1^2x_2 y_1y_2y_3$ and the weight of the pair of tableaux are $x_1^2x_2 y_1y_2y_3$ times the according entry in the last row. 
}
\end{table}

\subsection{Properties of the $\dRSK$ correspondence}
Our randomized \linebreak $\dRSK$ correspondence can be specialized in two different ways: one can specialize the parameters $q,t$ or restrict the correspondence to a smaller family of matrices.

For $q,t \in [0,1)$ or $q,t \in (1,\infty)$ the forward and backward probabilities of $\dRSK$ become actual probabilities, i.e., they have values in $[0,1]$.
By setting $t=0$ or $t\rightarrow\infty$ we obtain the \emph{$q$-Whittaker dual row insertion} or \emph{$q$-Whittaker dual column insertion} respectively for $q$-Whittaker polynomials which were first defined by Matveev and Petrov \cite[\S5.1, \S5.4]{MatveevPetrov17}. By a symmetry property of $\dRSK$ these two specializations are connected to setting $q=0$ or $q\rightarrow\infty$, i.e., $t$-deformations of the row and column version of $\RSK^*$  for Hall-Littlewood polynomials. Finally for $q=t=0$ or $q=t\rightarrow \infty$ we obtain the row or column insertion version of $\RSK^*$; this can be checked easily in the example presented in Table~\ref{tab:explicit bijection}. 

Another interesting specialization of $\dRSK$ comes from setting $q=t^\alpha$ and taking the limit $t \rightarrow 1$; this corresponds to specializing Macdonald polynomials to Jack polynomials. By restricting to matrices with at most one nonzero entry in each column, the insertion tableau obtained by the $\dRSK$ correspondence is invariant with respect to interchanging columns of the input matrix. In Table~\ref{tab:explicit bijection Jack} this can be observed by the fact that the sum of the probabilities in the first two columns as well as the probability in last column are equal for all three input matrices.\medskip
\begin{table}[h]
\[
\begin{array}{l|ccc}
& \young(11,2)\; , \; \young(12,3) \quad & \quad  \young(11,2)\; , \; \young(13,2)\quad  & \quad  \young(112)\; , \; \young(123)
 \\[12pt]
 \hline \\[-6pt]
\begin{pmatrix}
1&1&0 \\ 0&0&1
\end{pmatrix} &  \dfrac{2\alpha}{2\alpha+1} & 0 & \dfrac{1}{2\alpha+1}
 \\[18pt]
\begin{pmatrix}
1&0&1 \\ 0&1&0
\end{pmatrix} & \dfrac{\alpha}{(\alpha+1)(2\alpha+1)} & \dfrac{\alpha}{\alpha+1} &  \dfrac{1}{2\alpha+1}
 \\[18pt]
\begin{pmatrix}
0&1&1 \\ 1&0&0
\end{pmatrix} & \dfrac{\alpha}{(\alpha+1)(2\alpha+1)} & \dfrac{\alpha}{\alpha+1} & \dfrac{1}{2\alpha+1} 
\end{array}
\]
\caption{\label{tab:explicit bijection Jack} The forward probabilities $\mc{P}(A \rightarrow P,Q)$ of Table~\ref{tab:explicit bijection} in the Jack limit.
}
\end{table}

By restricting the input of $\dRSK$ to $\{0,1\}$-matrices with at most one entry equal to $1$ in each column we obtain a $(q,t)$-deformation of RS for words. By further restricting to permutation matrices we obtain $\qrst$. These specializations are summarized in the next table.

\[
\begin{array}{lp{2.7cm}l|p{2.5cm}}
\text{Corresp.} & \text{Domain} & \text{Image} & \text{Macdonald case}\\
\hline
\text{RS} & $\pi \in S_n$ & \ds \bigcup\limits_{\la \vdash n} \SYT_{\la} \times \SYT_{\la} & \text{\cite[Def 4.6.5]{AignerFrieden22}} \\[18pt]
\text{RS for words} & $w \in \mathbb{Z}_{>0}^*$ & \ds \bigcup\limits_{\la} \SSYT_{\la} \times \SYT_{\la} 
& \text{Definition~\ref{def:qrst insertion}} \\[18pt]
\RSK^* & $\{0,1\}$\text{-matrices} & \ds \bigcup\limits_{\la} \SSYT_{\la} \times \SSYT^*_{\la}
& \text{ Theorem~\ref{thm:dRSK is prob bij},} \text{Definition~\ref{def:qtRSK insertion}} \\[18pt]
\RSK & nonnegative integer matrices & \ds \bigcup\limits_{\la} \SSYT_{\la} \times \SSYT_{\la} & \text{open}
\end{array}
\]

Similarly to the classical dual RSK, the $\dRSK$ correspondence also yields a tableaux theoretic proof of the skew dual Cauchy identity for Macdonald polynomials. The dual Pieri rule for Macdonald polynomials
\begin{equation}
\label{eq:Pieri rule}
P_\mu(\x;q,t) e_r(\x)=\sum_{\substack{\la \succ^\prime \mu, \\ |\la/\mu|=r}} \varphi^*_{\la/\mu} P_\la(\x;q,t).
\end{equation}
can be obtained as the coefficient of $y_1^r$ of the skew dual Cauchy identity \eqref{eq:skew dual cauchy} in the special case
\begin{multline*}
P_\mu(\x;q,t) \prod_{i}(1+x_i y_1) = \sum_{\la} P_\la(\x;q,t) P_{\la^\prime/\mu^\prime}(y_1;t,q)\\
=\sum_{\la \succ^\prime \mu}  P_\la(\x;q,t) \varphi^*_{\la/\mu} y_1^{|\la/\mu|},
\end{multline*}
where the last equality is obtained by using the definition of $P_{\la^\prime/\mu^\prime}(y_1;t,q)$ and the fact that $P_{\la^\prime/\mu^\prime}(y_1;t,q) =0$ unless $\la/\mu$ is a vertical strip. Hence it is immediate that $\dRSK$ yields a tableaux theoretic proof of the dual Pieri rule for Macdonald polynomials. In particular, for each semistandard Young tableau $T$ of shape $\mu$ and sequence $(a_i)$ with $r$ entries equal to $1$ and all other entries $0$ we obtain by $\dRSK$  a probability distribution on semistandard Young tableaux $\wh{T}$ of shape $\la \succ^\prime \mu$ by filling the following growth diagram according to our probabilistic local growth rules.

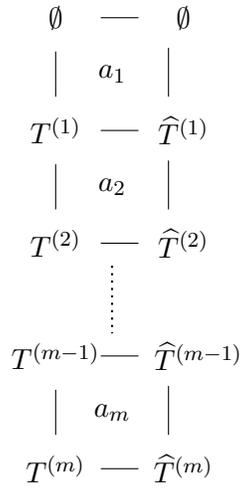
\begin{center}
\begin{tikzpicture}
\node at (0,0) {$\emptyset$};
\node at (0,-1.5) {$T^{(1)}$};
\node at (0,-2*1.5) {$T^{(2)}$};
\node at (0,-3*1.5) {$T^{(m-1)}$};
\node at (0,-4*1.5) {$T^{(m)}$};

\node at (1.7,0) {$\emptyset$};
\node at (1.7,-1.5) {$\wh T^{(1)}$};
\node at (1.7,-2*1.5) {$\wh T^{(2)}$};
\node at (1.9,-3*1.5) {$\wh T^{(m-1)}$};
\node at (1.7,-4*1.5) {$\wh T^{(m)}$};

\foreach \x in {0,...,4}{
		\draw (0.6,-1.5*\x) -- (1.1,-1.5*\x);
	}
	
\draw (0,-0.45) -- (0,.45-1.5*1);
\draw (0,-0.45-1.5*1) -- (0,.45-1.5*2);
\draw (0,-0.45-1.5*3) -- (0,.45-1.5*4);
\draw (1.5,-0.4) -- (1.5,.45-1.5*1);
\draw (1.5,-0.4-1.5*1) -- (1.5,.45-1.5*2);
\draw (1.5,-0.4-1.5*3) -- (1.5,.45-1.5*4);

\draw [dotted, thick] (1.5*0.5,-2*1.5-.3) -- (1.5*0.5,-1.5*3+.3);
\node at (1.5*0.5,-1.5*0.5) {$a_{1}$};
\node at (1.5*0.5,-1.5*1.5) {$a_{2}$};
\node at (1.5*0.5,-1.5*3.5) {$a_{m}$};
\end{tikzpicture}
\end{center}

\subsection{Future work}
One immediate question which we want to study in the future is whether there is an analogue extension of $\qrst$ to a $(q,t)$-variation of RSK. There already exists a $q$-deformation of row insertion and a $t$-variation of column insertion by \cite{MatveevPetrov17, BufetovMatveev18}, which a hoped for $(q,t)$-variation of RSK would specialize to. Another direction is to provide a combinatorial proof for the fact that Macdonald functions, when defined by \eqref{eq:macdonald monomial}, are symmetric with respect to the $\x$-variables. In the case of Schur functions such a proof is obtained by either using the Bender--Knuth involution or the Lascoux--Sch\"utzenberger symmetric group action on semistandard Young tableaux. Again one would hope for a probabilistic variation generalizing one or both of these bijections.

For two weighted sets with equal sum of weights the forward and backward probabilities have to satisfy simple algebraic equations. For a generic situation this implies that there exist many probabilistic bijections between these two sets. The probabilities presented in this paper were obtained by regarding small examples,  choosing the ``reasonable'' solution for these equations in each example, and then guessing a general formula. It was crucial that for all regarded examples there existed exactly one ``reasonable'' solution, where reasonable meant that the solution factorised into a minimal number of irreducible factors (this property is closely related to the notion of height functions in Diophantine geometry).
An open problem is to identify an additional property, e.g. the above described one, for probabilistic bijections and prove that the probabilistic bijections presented in this paper are unique with respect to this property.

\subsection{Outline of the paper}

In \S \ref{sec:dual cauchy}, we introduce basic notions and show several approaches for proving the dual Cauchy identity \eqref{eq:dual cauchy}. First we present the column version of dual RSK, continue with an algebraic proof using up and dual down operators on Young's lattice, and finally show how Fomin's growth diagrams can be used to ``bijectivize'' such an algebraic proof leading to a family of insertion algorithms generalizing dual RSK.

In \S \ref{sec:Mac}, we review basic facts about Macdonald polynomials and define $(q,t)$-up and dual down operators.
We introduce in \S \ref{sec:prob bij} the notion of a probabilistic bijection and present in \S \ref{sec:probabilities definition} the forward and backward probabilities which yield a proof of the commutation relation of the $(q,t)$-up and dual down operators. In \S \ref{sec:recovering dual RSK} we give a second formulation of our probabilities which allows us to show that they actually generalize the classical dual RSK. In particular we obtain a lattice path description for the ``maximal'' monomial in $q$ and $t$ of the probabilities.
In \S \ref{sec:dual qtRSK} we use probabilistic growth rules to define the $\dRSK$ correspondence.

In \S \ref{sec:weights proof}, we prove that the forward and backward probabilities actually form a probabilistic bijection. The fact that these probabilities form probability distributions follows by a version of Lagrange interpolation presented in \S \ref{sec:interpolation}, while the compatibility condition is shown by introducing and analyzing a third formulation of the forward and backward probabilities  in \S \ref{sec:prob via hook-lengths}-\ref{sec:weights formula}.

In \S \ref{sec:degenerations}, we consider various degenerations of $\dRSK$ by specializing $q$ or $t$ respectively; in particular we prove an interchanging property of the correspondence in the Jack limit. In \S \ref{sec:words} we restrict $\dRSK$ to $\{0,1\}$-matrices with at most one $1$ entry per column and obtain a $(q,t)$-version of RS for words.

\subsection*{Acknowledgments}
GF was partially supported by the Canada Research Chairs program. FSA acknowledges the financial support from the Austrian Science  Fund (FWF) \href{https://dx.doi.org/10.55776/J4387}{doi: 10.55776/J4387}, \href{https://dx.doi.org/10.55776/P34931}{doi: 10.55776/P34931} and \href{https://dx.doi.org/10.55776/P36863}{doi: 10.55776/P36863}.

.

\section{Dual Cauchy identity, up and down operators, and growth diagrams}
\label{sec:dual cauchy}

\subsection{Schur functions and the dual Cauchy identity}
\label{sec: insertion algorithms}

A \deff{partition} $\lambda=(\lambda_1,\ldots,\lambda_k)$ is a weakly decreasing sequence of nonnegative integers. We say that $\lambda$ is a partition of the integer $|\lambda|=\lambda_1+\cdots + \lambda_k$ which is denoted by $\lambda \vdash |\lambda|$. We identify the partition $\la$ with its \deff{Young diagram} which is, following french convention, a collection of left-justified boxes consisting of $\lambda_i$ boxes in the $i$-th row from bottom. Starting in \S \ref{sec:probabilities} we denote Young diagrams using \deff{Quebecois convention}, in which  the boxes are right-justified instead of left-justified. The boxes, or \deff{cells}, of a Young diagram are indexed by positive Cartesian coordinates, i.e., the $x$-th cell in the $y$-th row from bottom is $c=(x,y)$. We denote by $\lambda^\prime$ the \deff{conjugate} of $\la$, which is obtained by reflecting the Young diagram of $\la$ in the diagonal $x=y$.

We define for a cell $c=(x,y) \in \la$ its \deff{arm-length} $a_\lambda(c)$ and its \deff{leg-length} $\ell_\lambda(c)$ by
\[
a_\lambda(c)= \lambda_y-x, \qquad \qquad
\ell_\lambda(c)=\lambda_x^\prime-y.
\]
The \deff{hook-length} of $c$ is defined as $h_\lambda(c)=a_\lambda(c)+\ell_\lambda(c)+1$. The cell $c$ as in Figure \ref{fig: Young diagram} has arm-length $a_\la(c) = 5$, leg-length $\ell_\la(c) = 3$, and hook-length $h_\la(c) = 9$.

\begin{figure}[h]
\begin{center}
\begin{tikzpicture}
\Yboxdim{16 pt}
\Ylinecolour{lightgray}
\tyng(0,0,7,6,3,2,1,1)
\Ylinecolour{black}
\tgyoung(0pt,0pt,:;)
\node at (24pt,8pt) {$c$};
\node at (72pt,8pt) {$a_\lambda(c)$};
\draw[->] (86pt,8pt) -- (106pt,8pt);
\draw[->] (58pt,8pt) -- (38pt,8pt);
\node at (24pt,40pt) {$\ell_\lambda(c)$};
\draw[<-] (24pt,20pt) -- (24pt,30pt);
\draw[->] (24pt,50pt) -- (24pt,60pt);
\end{tikzpicture}
\end{center}
\caption{\label{fig: Young diagram} The Young diagram of the partition $\la = (7,6,3,2,1,1)$ for which the cell $c=(2,1)$ is marked.}
\end{figure}
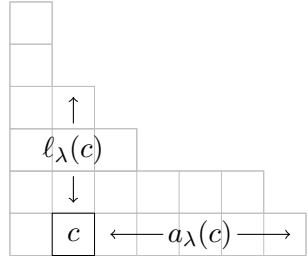

Let $\lambda$ be a partition. A filling of the cells of $\la$ with positive integers is called a
\begin{itemize}
\item \deff{semistandard Young tableau} (SSYT) of shape $\la$, if the rows are weakly increasing and columns are strictly increasing,
\item \deff{dual semistandard Young tableau} of shape $\la$, if the rows are strictly increasing and columns are weakly increasing,
\item \deff{partial standard Young tableau} of shape $\la$, if the rows and columns are strictly increasing and each entry appears at most once.
\end{itemize}
It is immediate, that semistandard Young tableaux of shape $\la$ are in bijection to dual semistandard Young tableaux of shape $\la^\prime$. We denote by $\SSYT(\lambda)$ the set of all semistandard Young tableaux of shape $\lambda$, by $\SSYT^*(\lambda)$ the set of all dual semistandard Young tableaux of shape $\la$ and by $\SYT^{\leq n}(\la)$ the set of all partial standard Young tableaux of shape $\la$ and entries at most $n$. The \deff{content} of an SSYT $T$ is the sequence $(\mu_1,\mu_2,\ldots)$, where $\mu_i$ is the number of entries in $T$ equal to $i$. The content of a dual semistandard Young tableau is defined analogously.

\begin{ex}
The following is a semistandard Young tableau of shape $(6,5,3)$ and content $(2,2,2,3,2,3)$.
\begin{center}
\young(112356,24446,356)
\end{center}
\end{ex}

Let $\mb{x} = (x_1, x_2, \ldots)$ be an infinite sequence of indeterminates. For a (dual) semistandard Young diagram $T$ with content $\mu$, we set $\mb{x}^T= \prod_{i\geq 1}x_i^{\mu_i}$.  The \deff{Schur function} $s_\lambda(\mb{x})$ associated to the partition $\lambda$ is defined by
\[
s_\lambda(\mathbf{x}) = \sum_{T \in \SSYT(\lambda)} \mathbf{x}^T.
\]
The next two theorems are fundamental results in the theory of symmetric functions.

\begin{thm}[Cauchy identity]
\label{thm: Schur Cauchy}
For two sequences of indeterminates $\mb{x} = (x_1, x_2, \ldots)$ and $\mb{y} = (y_1, y_2, \ldots)$, we have
\[
\prod_{i,j \geq 1}\frac{1}{1-x_iy_j} = \sum_\lambda s_\lambda(\mathbf{x})s_\lambda(\mathbf{y}),
\]
where the sum is over all partitions $\lambda$.
\end{thm}

The Cauchy identity can be rewritten in the form
\begin{equation}
\label{eq_Cauchy_bijective}
\sum_{A=(a_{ij})} \prod_{i,j \geq 1}(x_iy_j)^{a_{ij}} = \sum_{(P,Q)} \mb{x}^P \mb{y}^Q,
\end{equation}
where the sum on the left-hand side runs over all nonnegative integer matrices with finite support (i.e., only finitely many nonzero entries), and the sum on the right-hand side runs over all ordered pairs of SSYTs of the same shape.

\begin{thm}[Dual Cauchy identity]
\label{thm: Schur Dual Cauchy}
For two sequences of indeterminates $\mb{x} = (x_1, x_2, \ldots)$ and $\mb{y} = (y_1, y_2, \ldots)$, we have
\[
\prod_{i,j \geq 1}(1+x_iy_j) = \sum_\lambda s_\lambda(\mathbf{x})s_{\lambda^\prime}(\mathbf{y}),
\]
where the sum is over all partitions $\lambda$.
\end{thm}

Analogously to above, the dual Cauchy identity can be rewritten as
\begin{equation}
\label{eq: Dual Cauchy bijective}
\sum_{A=(a_{ij})} \prod_{i,j \geq 1}(x_iy_j)^{a_{ij}} = \sum_{(P,Q)} \mb{x}^P \mb{y}^Q,
\end{equation}
where the sum on the left is over all finitely supported $\{0,1\}$-matrices, and the sum on the right is over all SSYTs $P$ and dual SSYTs $Q$ of the same shape.\bigskip

The Cauchy and dual Cauchy identities can be proven combinatorially by the \deff{Robinson--Schensted--Knuth (RSK) correspondence} and the \deff{dual RSK correspondence (RSK${}^*$)}, respectively. Each of these bijections has a \deff{row insertion version} and a \deff{column insertion version}. In the remainder of this section, we recall the column insertion version of dual RSK. 

The basic building block of the correspondence is the \deff{insertion} of a number $k$ into a column of a semistandard Young tableau: if $k$ is larger than all entries in the regarded column, add $k$ to the top of the column. Otherwise, let $z$ be the bottom-most entry in the column which is larger than or equal to $k$, and replace $z$ with $k$. We say that $z$ is \deff{bumped} out of the column.
Next, define the \deff{column insertion} of a number $k$ into an SSYT $P$, denoted $k \rightarrow P$, by the following process.
\begin{itemize}
\item Insert $k$ into the first (left-most) column of $P$. If no entry is bumped, the process terminates.
\item If $z$ is bumped out of column $i$, insert $z$ into column $i+1$. Continue in this manner until no bumping occurs.
\end{itemize}
Note that the shape of $k \rightarrow P$ differs from the shape of $P$ by the addition of a single cell. \bigskip

Let $A$ be a finitely supported $\{0,1\}$-matrix, and denote by $\omega_A$ the \deff{biword} obtained as follows. We read the columns of $A$ from top to bottom, starting with the first column, and append $\begin{pmatrix} j \\ i \end{pmatrix}$ to $\omega_A$ if the entry in row $i$ and column $j$ is $1$. Given the biword
\[
\omega_A = \begin{pmatrix}
j_1 & j_2 & \cdots & j_N \\
i_1 & i_2 & \cdots & i_N
\end{pmatrix},
\]
we define $P$ as the SSYT obtained by the sequence of column insertions $i_N \rightarrow i_{N-1} \rightarrow \cdots \rightarrow i_1 \rightarrow \emptyset$. We define $Q$ as the dual SSYT which records the growth of $P$ by placing the number $j_k$ in the cell that was added during the insertion of $i_k$. It is straightforward to see that this is indeed a dual SSYT. The tableaux $P$ and $Q$ are called the \deff{insertion tableau} and \deff{recording tableau} of $A$, respectively. One can show that these two tableaux provide enough information to reverse the insertion procedure and recover $A$. The map $A \mapsto (P,Q)$ is therefore a bijection between finitely supported $\{0,1\}$-matrices and pairs of tableaux of the same shape, where the first tableau is semistandard and the second is dual semistandard. We refer to this bijection as the \deff{column insertion version of the dual RSK}.

\begin{ex}
\label{ex: dual RSK insertion}
Let $A$ be the matrix
\[
A=\begin{pmatrix}
1 & 0 & 0 & 0 & 1 \\ 
1 & 1 & 0 & 1 & 0 \\ 
0 & 1 & 1 & 0 & 0
\end{pmatrix}.
\]
The corresponding biword is 
\[
\omega_A=\begin{pmatrix}
1 & 1 & 2 & 2 & 3 & 4 & 5 \\
1 & 2 & 2 & 3 & 3 & 2 & 1
\end{pmatrix}.
\]
The process of constructing $P$ and $Q$ by the successive insertions
\[
1 \rightarrow 2 \rightarrow 3 \rightarrow 3 \rightarrow 2 \rightarrow 2 \rightarrow 1 \rightarrow \emptyset 
\]
is shown next.
\[
\begin{array}{cccccccccccccccc}
 \emptyset & \young(1) & \young(1,2) & \young(12,2) & \young(12,2,3) & \young(12,23,3) & \young(122,23,3) & \young(1122,23,3) & = P \\ \\
 \emptyset & \young(1) & \young(1,1) & \young(12,1) & \young(12,1,2) & \young(12,13,2) & \young(124,13,2) & \young(1245,13,2) & = Q
\end{array}
\]
\end{ex}

Note that the \deff{row insertion version of dual RSK} is defined analogously with the difference that we use row insertion and that we construct the biword by reading the columns from bottom to top, starting with the first column.

\subsection{Up and down operators}
\label{sec: up and down operators}

In this section we explain an algebraic approach for proving the (dual) Cauchy identity which is based on linear operators on Young's lattice. For more insights on these operators compare for example to \cite{Fomin95} and \cite[§2.2-3.1]{vanLeeuwen05}.\bigskip

Let $\la, \mu$ be partitions. We write $\mu \subseteq \la$ if the Young diagram of $\mu$ is contained in the Young diagram of $\la$. Denote by $\la \cup \mu$ (resp., $\la \cap \mu$) the partition obtained by the union (resp., intersection) of the Young diagrams of $\la$ and $\mu$. \deff{Young's lattice} is the partial order on partitions defined by the inclusion relation $\subseteq$; its meet and join are given by $\cap$ and $\cup$, respectively. 
For $\mu \subseteq \la$ we denote by $\la/\mu$ the \deff{skew diagram} which consists of the cells in $\la$ which are not included in $\mu$. We say that $\la/\mu$ is a \deff{horizontal strip} (resp., vertical strip) if no two cells of $\la / \mu$ are in the same column (resp., row), where we use the notation $\mu \prec \la$ (resp., $\mu \prec^\prime \la$). 
Denote by $\bb{Y}$ the set of all partitions and by $\bb{Q}\bb{Y}$ the vector space generated by partitions, i.e.,
\[
\bb{Q}\bb{Y}:=\bigoplus_{\la \in \bb{Y}}\bb{Q}\la.
\]
We define $\bb{Y}\llbracket x_1,\ldots,x_m\rrbracket$ as the $\bb{Q}$-vector space of formal power series in the variables $x_1,\ldots,x_m$ with coefficients in $\bb{Q}\bb{Y}$, i.e., the coefficient of each monomial $x^iy^j$ is a finite formal sum over partitions. We define the \deff{up operator} $U_x$, the \deff{down operator} $D_y$ and \deff{dual down operator} $D_y^*$ as $\bb{Q}$-linear maps on $\bb{Y}\llbracket x,y\rrbracket$ via
\[
U_x \lambda = \sum_{\nu \succ \la} x^{|\nu / \la|} \nu, \qquad \qquad
D_y \lambda = \sum_{\mu \prec \la} y^{|\la / \mu|} \mu, \qquad \qquad
D_y^* \lambda = \sum_{\mu \prec^\prime \la} y^{|\la / \mu|} \mu.
\]

\begin{thm}
\label{thm: Schur commutation relation}
The up and (dual) down operator satisfy the commutation relations
\begin{align}
\label{eq: Schur commutation relation}
D_y U_x = \frac{1}{1-xy} U_x D_y ,\\
D_y^* U_x = (1+xy) U_x D_y^*.
\label{eq: Schur dual commutation relation}
\end{align}
\end{thm}

Before explaining the proof of the second commutation relation, we show how it implies the dual Cauchy identity for Schur polynomials; the Cauchy identity follows analogously. Denote by $\langle \cdot,\cdot\rangle$ the inner product on $\bb{Q}\bb{Y}$ defined by $\langle \lambda, \mu \rangle = \delta_{\lambda,\mu}$, for all $\la,\mu \in \bb{Y}$ which we extend linearly to a $\bb{Q}\llbracket x_1,\ldots,x_m,y_1,\ldots,y_n\rrbracket$-bilinear function from $\bb{Y}\llbracket x_1,\ldots,x_m,y_1,\ldots,y_n\rrbracket \times \bb{Q}\bb{Y}$ to $\bb{Q}\llbracket x_1,\ldots,x_m,y_1,\ldots,y_n\rrbracket$. On the one hand, we have 
\[
\left\langle  D^*_{y_n} \cdots D^*_{y_1}  U_{x_m} \cdots U_{x_1} \emptyset, \emptyset \right\rangle = \sum_{\la} \langle  U_{x_m} \cdots U_{x_1}\emptyset, \la \rangle \; \langle D^*_{y_n} \cdots D^*_{y_1}  \la, \emptyset \rangle.
\]
A semistandard Young tableau $P$ of shape $\la$ can be seen as a chain
\[
 \emptyset = P^{(0)} \prec P^{(1)} \prec \cdots \prec P^{(m)} = \lambda,
\]
in Young's lattice, where $P^{(i)}$ denotes the shape of the subtableau of $P$ consisting of entries less than or equal to $i$. Further we have $\mb{x}^P = \linebreak \prod_{i=1}^m x_i^{|P^{(i)} / P^{(i-1)}|}$. This implies that $\langle  U_{x_m} \cdots U_{x_1}\emptyset, \la \rangle = s_\la(x_1,\ldots,x_m)$. Analogously, a dual semistandard Young tableau $Q$ can be seen as a chain 
\[
 \emptyset = Q^{(0)} \prec^\prime Q^{(1)} \prec^\prime \cdots \prec^\prime Q^{(n)} = \lambda.
\]

On the other hand, we can repeatedly use the commutation relation to move a $D_{y_j}^*$ past all $U_{x_i}$'s and obtain
\begin{multline*}
D^*_{y_j}U_{x_m} \cdots U_{x_1}  = (1+x_m y_j) U_{x_m} D^*_{y_j}U_{x_{m-1}} \cdots U_{x_1} \\= \prod_{i=1}^m (1+x_iy_j) \cdot  U_{x_m} \cdots U_{x_1} D^*_{y_j}.
\end{multline*}
Since $D^*_{y}\emptyset = \emptyset$, we obtain by induction
\begin{multline*}
\left\langle  D^*_{y_n} \cdots D^*_{y_1}  U_{x_m} \cdots U_{x_1} \emptyset, \emptyset \right\rangle =
\prod\limits_{\substack{ 1 \leq i \leq m \\ 1 \leq j \leq n}}(1+x_i y_j) \left\langle U_{x_m} \cdots U_{x_1} \emptyset, \emptyset \right\rangle \\ =
\prod\limits_{\substack{ 1 \leq i \leq m \\ 1 \leq j \leq n}}(1+x_i y_j) ,
\end{multline*}
which proves  Theorem \ref{thm: Schur Dual Cauchy} when restricting to the variables $x_1,\ldots,x_m$ and $y_1,\ldots,y_n$. \bigskip

Note that both commutation relations in Theorem~\ref{thm: Schur commutation relation} are special cases of the skew version of the (dual) Cauchy identity for Schur polynomials, compare to \cite[Ch. 1, Eq. (5.8 a)]{Mac}.
 In the remainder of the section we present a combinatorial proof of the commutation relation \eqref{eq: Schur dual commutation relation} for the up and dual down operator.

Define the sets $\U^k(\la,\rho):=\{ \nu | \nu \succ^\prime \la, \nu \succ \rho, |\nu / (\la \cup \rho)|=k\}$ and $\D^k(\la,\rho) := \{ \mu | \mu \prec \la, \mu \prec^\prime \rho, |(\la \cap \rho)/\mu|=k\}$. In order to prove \eqref{eq: Schur dual commutation relation} it suffices to show
\begin{equation}
\label{eq: up-down on sets}
|\U^k(\la,\rho)| = |\D^k(\la,\rho)| + |\D^{k-1}(\la,\rho)|.
\end{equation}
for all partitions $\la,\rho$ and nonnegative integers $k$.  \bigskip

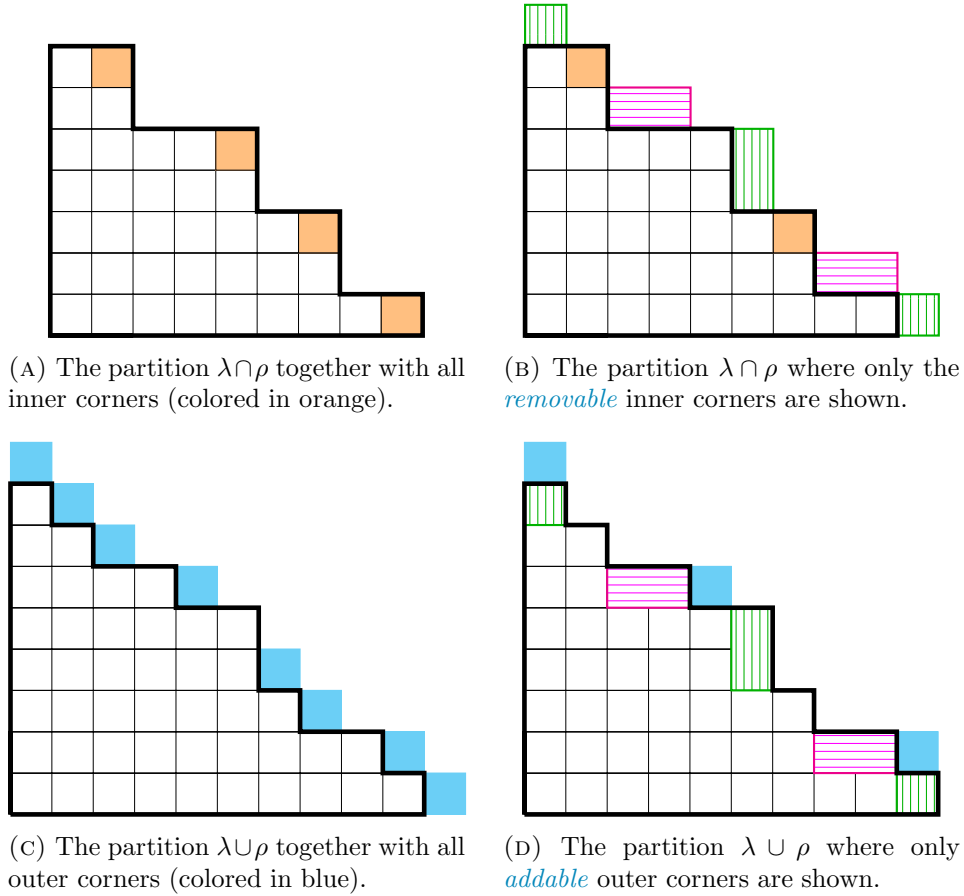
\begin{figure}[h]
\begin{center}
\begin{subfigure}[t]{.48\textwidth}
\centering
\begin{tikzpicture}[scale=1.1]
\tyng(0,0,9,7,7,5,5,2,2)
\draw[fill, color=myorange] (4.5,.5) rectangle (4,0);
\draw (4.5,.5) rectangle (4,0);
\draw[fill, color=myorange] (3.5,1.5) rectangle (3,1);
\draw (3.5,1.5) rectangle (3,1);
\draw[fill, color=myorange] (2.5,2.5) rectangle (2,2);
\draw (2.5,2.5) rectangle (2,2);
\draw[fill, color=myorange] (.5,3.5) rectangle (1,3);
\draw (.5,3.5) rectangle (1,3);
\draw[line width=1.75pt] (0,0) -- (4.5,0) -- (4.5,.5) -- (3.5,.5) -- (3.5,1.5) -- (2.5,1.5) -- (2.5,2.5) -- (1,2.5) -- (1,3.5) -- (0,3.5) -- (0,0) -- (1,0);
\end{tikzpicture}
\caption{The partition $\la \cap \rho$ together with all inner corners (colored in orange).}
\end{subfigure} \hfill
\begin{subfigure}[t]{0.48\textwidth}
\centering
\begin{tikzpicture}[scale=1.1]
\tyng(0,0,9,7,7,5,5,2,2)
\draw[thick, magenta, pattern=horizontal lines, pattern color=magenta] (3.5,.5) rectangle (4.5,1);
\draw[thick, magenta, pattern=horizontal lines, pattern color=magenta] (1,2.5) rectangle (2,3);
\draw[thick, green!70!black, pattern=vertical lines, pattern color=green!70!black] (4.5,.5) rectangle (5,0);
\draw[thick, green!70!black, pattern=vertical lines, pattern color=green!70!black] (3,1.5) rectangle (2.5,2.5);
\draw[thick, green!70!black, pattern=vertical lines, pattern color=green!70!black] (.5,3.5) rectangle (0,4);
\draw[fill, color=myorange] (3.5,1.5) rectangle (3,1);
\draw[fill, color=myorange] (.5,3.5) rectangle (1,3);
\draw (3.5,1.5) rectangle (3,1);
\draw (.5,3.5) rectangle (1,3);
\draw[line width=1.75pt] (0,0) -- (4.5,0) -- (4.5,.5) -- (3.5,.5) -- (3.5,1.5) -- (2.5,1.5) -- (2.5,2.5) -- (1,2.5) -- (1,3.5) -- (0,3.5) -- (0,0) -- (1,0);
\end{tikzpicture}
\caption{The partition $\la \cap \rho$ where only the \deff{removable} inner corners are shown. \bigskip}
\end{subfigure}
\begin{subfigure}[t]{.48\textwidth}
\centering
\begin{tikzpicture}[scale=1.1]
\tyng(0,0,10,9,7,6,6,4,2,1)
\draw[fill, color=myblue] (5.5,.5) rectangle (5,0);
\draw[fill, color=myblue] (4.5,.5) rectangle (5,1);
\draw[fill, color=myblue] (3.5,1.5) rectangle (4,1);
\draw[fill, color=myblue] (3.5,1.5) rectangle (3,2);
\draw[fill, color=myblue] (2.5,2.5) rectangle (2,3);
\draw[fill, color=myblue] (1.5,3.5) rectangle (1,3);
\draw[fill, color=myblue] (.5,3.5) rectangle (1,4);
\draw[fill, color=myblue] (.5,4.5) rectangle (0,4);
\draw[line width=1.75pt] (0,0) -- (5,0) -- (5,.5) -- (4.5,.5) -- (4.5,1) -- (3.5,1) -- (3.5,1.5) -- (3,1.5) -- (3,2.5) -- (2,2.5) -- (2,3) -- (1,3) -- (1,3.5) -- (.5,3.5) -- (.5,4) -- (0,4) -- (0,0) -- (0,1);
\end{tikzpicture}
\caption{The partition $\la \cup \rho$ together with all outer corners (colored in blue).}
\end{subfigure}
\hfill
\begin{subfigure}[t]{.48\textwidth}
\centering
\begin{tikzpicture}[scale=1.1]
\tyng(0,0,9,7,7,5,5,2,2)
\draw[thick, magenta, pattern=horizontal lines, pattern color=magenta] (3.5,.5) rectangle (4.5,1);
\draw[thick, magenta, pattern=horizontal lines, pattern color=magenta] (1,2.5) rectangle (2,3);
\draw[thick, green!70!black, pattern=vertical lines, pattern color=green!70!black] (4.5,.5) rectangle (5,0);
\draw[thick, green!70!black, pattern=vertical lines, pattern color=green!70!black] (3,1.5) rectangle (2.5,2.5);
\draw[thick, green!70!black, pattern=vertical lines, pattern color=green!70!black] (.5,3.5) rectangle (0,4);
\draw[fill, color=myblue] (4.5,.5) rectangle (5,1);
\draw[fill, color=myblue] (2.5,2.5) rectangle (2,3);
\draw[fill, color=myblue] (.5,4.5) rectangle (0,4);
\draw[line width=1.75pt] (0,0) -- (5,0) -- (5,.5) -- (4.5,.5) -- (4.5,1) -- (3.5,1) -- (3.5,1.5) -- (3,1.5) -- (3,2.5) -- (2,2.5) -- (2,3) -- (1,3) -- (1,3.5) -- (.5,3.5) -- (.5,4) -- (0,4) -- (0,0) -- (0,1);
\end{tikzpicture}
\caption{The partition $\la \cup \rho$ where only \deff{addable} outer corners are shown.}
\end{subfigure}
\end{center}
\caption{\label{fig: inner outer corner} The (removable) inner corners of $\la \cap \rho$ and the (addable) outer corners of $\la \cup \rho$ for the partitions $\la=(9,9,7,5,5,4,2)$ and $\rho=(10,7,7,6,6,2,2,1)$. We shaded the cells of $\la/(\la \cap \rho)$ horizontally in magenta and the cells of $\rho /(\la \cap \rho)$ vertically in green.}
\end{figure}

An \deff{inner corner} of a partition $\lambda$ is a cell $c \in \la$ such that $\la/\mu = \{c\}$ for a partition $\mu \prec \la$. An \deff{outer corner} of $\la$ is a cell $c \notin \la$ such that $\nu/\la =\{c\}$ for a partition $\nu$ with $\la \prec \nu$. We call an inner corner $c$ of $\la \cap \rho$ \deff{removable} with respect to $(\la,\rho)$ if $\la/\mu$ is a horizontal strip and $\rho/\mu$ is a vertical strip, where $(\la \cap \rho)/\mu=\{c\}$.
Analogously we call an outer corner $c$ of $\la \cup \rho$ \deff{addable} with respect to $(\la,\rho)$ if $\nu/\la$ is a vertical strip and $\nu/\rho$ is a horizontal strip, where $\nu/(\la \cup \rho)=\{c\}$.
For an example see Figure \ref{fig: inner outer corner}. For both removable and addable corners we omit referring to $(\la,\rho)$ whenever the partitions $\la,\rho$ are clear from context. \bigskip

Each partition $\nu$ in $\U^k(\la,\rho)$ corresponds to a $k$-subset of the addable outer corners of $\la \cup \rho$ and each partition $\mu$ in $\D^k(\la,\rho)$ corresponds to a $k$-subset of the removable inner corners of $\la \cap \rho$. It is not difficult to prove that there is one more addable outer corner of $(\la \cup \rho)$ than removable inner corner of $(\la \cap \rho)$. This implies immediately the assertion on the sizes of the sets in \eqref{eq: up-down on sets}. As we will see in the next subsection, it turns out to be fruitful to prove \eqref{eq: up-down on sets} explicitly bijective.
For each $\la,\rho$ and $k$, we choose a bijection 
\[
F_{\la,\rho,k}: \D^{k-1}(\la,\rho) \cup \D^k(\la,\rho) \rightarrow \U^k(\la,\rho).
\]
Two of the many possible bijections $F_{\la,\rho,k}$ are very natural: the \deff{dual row insertion bijection}  $F_{\la,\rho,k}^\row$ and the \deff{dual column insertion bijection} $F_{\la,\rho,k}^\col$. For $k=1$ the dual row (resp., column) insertion bijection maps a removable inner corner to the next addable outer corner in a row above (resp., column to the right) and sends the empty set to the lowest (resp., left-most) addable outer corner. Figure \ref{fig: row and col insertion map} illustrates this case. For $k> 1$ both maps are defined recursively by
\begin{equation}
\label{eq:recursion for F bijection}
F^\bullet_{\la,\rho,k} (X) = \begin{cases}
\bigcup\limits_{x \in X} F^\bullet_{\la,\rho,1}(\{x\}) \qquad &|X|=k, \\
F^\bullet_{\la,\rho,1}(\emptyset) \cup \bigcup\limits_{x \in X} F^\bullet_{\la,\rho,1}(\{x\})  &|X|=k-1,
\end{cases}
\end{equation}
where $F^\bullet_{\la,\rho,k}$ stands for $F^\row_{\la,\rho,k}$ or $F^\col_{\la,\rho,k}$ respectively.

\begin{figure}
\begin{center}
\begin{tikzpicture}[scale=1.1]
\tyng(0,0,10,9,7,6,6,4,2,1)
\draw[fill, color=myblue] (4.5,.5) rectangle (5,1);
\draw[fill, color=myblue] (2.5,2.5) rectangle (2,3);
\draw[fill, color=myblue] (.5,4.5) rectangle (0,4);
\draw[fill, color=myorange] (3.5,1.5) rectangle (3,1);
\draw[fill, color=myorange] (.5,3.5) rectangle (1,3);
\draw (3.5,1.5) rectangle (3,1);
\draw (.5,3.5) rectangle (1,3);
\draw[line width=1.75pt] (0,0) -- (5,0) -- (5,.5) -- (4.5,.5) -- (4.5,1) -- (3.5,1) -- (3.5,1.5) -- (3,1.5) -- (3,2.5) -- (2,2.5) -- (2,3) -- (1,3) -- (1,3.5) -- (.5,3.5) -- (.5,4) -- (0,4) -- (0,0) -- (0,1);
\node at (2,-.5) {$F^\row_{\la,\rho,1}$};
\draw [->,line width=1.3pt] (3.25,1.25) to [out=90,in=0] (2.25,2.75);
\draw [->,line width=1.3pt] (.75,3.25) to [out=90,in=0] (.25,4.25);

\begin{scope}[xshift=6cm]
\tyng(0,0,10,9,7,6,6,4,2,1)
\draw[fill, color=myblue] (4.5,.5) rectangle (5,1);
\draw[fill, color=myblue] (2.5,2.5) rectangle (2,3);
\draw[fill, color=myblue] (.5,4.5) rectangle (0,4);
\draw[fill, color=myorange] (3.5,1.5) rectangle (3,1);
\draw[fill, color=myorange] (.5,3.5) rectangle (1,3);
\draw (3.5,1.5) rectangle (3,1);
\draw (.5,3.5) rectangle (1,3);
\draw[line width=1.75pt] (0,0) -- (5,0) -- (5,.5) -- (4.5,.5) -- (4.5,1) -- (3.5,1) -- (3.5,1.5) -- (3,1.5) -- (3,2.5) -- (2,2.5) -- (2,3) -- (1,3) -- (1,3.5) -- (.5,3.5) -- (.5,4) -- (0,4) -- (0,0) -- (0,1);
\draw [->,line width=1.3pt] (3.25,1.25) to [out=0,in=90] (4.75,.75);
\draw [->,line width=1.3pt] (.75,3.25) to [out=0,in=90] (2.25,2.75);
\node at (2,-.5) {$F^\col_{\la,\rho,1}$};
\end{scope}
\end{tikzpicture}
\end{center}
\caption{\label{fig: row and col insertion map} The two maps $F^\row_{\la,\rho,1}$ (left) and $F^\col_{\la,\rho,1}$ (right) for $\la=(9,9,7,5,5,4,2) $, $\rho=(10,7,7,6,6,2,2,1)$. The removable inner corners (colored in orange) and the addable outer corners (colored in blue) are obtained in Figure~\ref{fig: inner outer corner}.}
\end{figure}
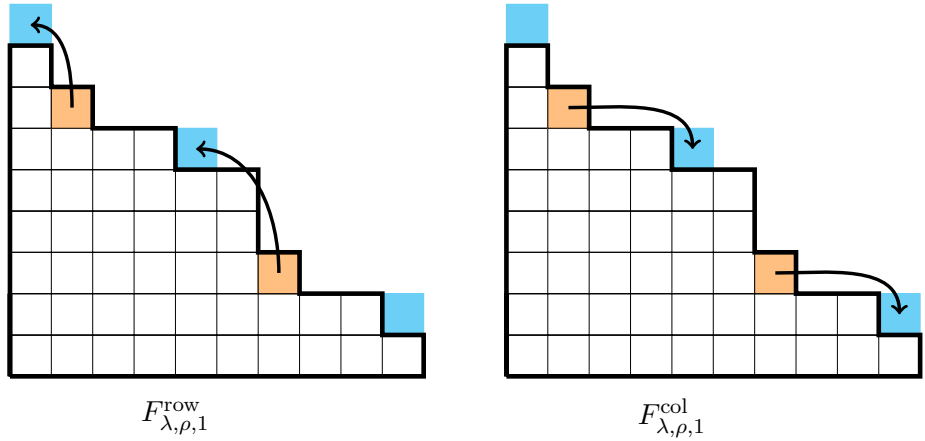

\subsection{Growth diagrams}
\label{sec: growth diagrams}

In this section we review Fomin's growth diagrams \cite{Fomin86, Fomin_DGG_2} which provide a mechanism for turning a bijective proof of the commutation relations in Theorem~\ref{thm: Schur commutation relation} into a bijective proof of the (dual) Cauchy identity. We explain how each bijective proof arising in this way can be interpreted as an insertion algorithm, with the row and column versions of the dual Robinson--Schensted correspondence as special cases. We only consider growth diagrams for $\{0,1\}$-matrices; for the generalization to nonnegative integer matrices and the RSK and Burge correspondences, we refer the reader to \cite[\S 2.2, \S3.1-3.2]{vanLeeuwen05}.

Let $A$ be an $m \times n$ $\{0,1\}$-matrix. We view $A$ as an $m \times n$ grid of squares, and consider labelings of the vertices of this grid with partitions. We index the squares of the grid by $(i,j) \in [m] \times [n]$, and the vertices by $(i,j) \in [0,m] \times [0,n]$; these indices are interpreted as matrix coordinates, rather than Cartesian coordinates.

\begin{defn}
Fix an $m \times n$ $\{0,1\}$-matrix $A$. A \deff{dual growth associated with $A$} is an assignment $\Lambda = (\Lambda_{ij})$ of partitions to the vertices $(i,j) \in [0,m] \times [0,n]$, such that
\begin{itemize}
\item if $i < m$, $\Lambda_{i+1,j}/\Lambda_{i,j}$ is a horizontal strip,
\item if $j < n$, $\Lambda_{i,j+1}/\Lambda_{i,j}$ is a vertical strip,
\item $|\Lambda_{ij}|$ is equal to the number of $1$'s in $A$ to the northwest of the vertex $(i,j)$, that is,
\[
|\Lambda_{i,j}| = \sum_{1 \leq k \leq i \atop 1 \leq \ell \leq j} A_{k,\ell}.
\]
\end{itemize}
\end{defn}

It is clear that the matrix $A$ is determined by the partitions $\Lambda_{i,j}$, so we will often omit reference to $A$ and simply refer to $\Lambda$ as a \deff{dual growth}. This definition can be traced back to the notion of ``two-dimensional growth'' introduced by Fomin \cite{Fomin86}.

\begin{ex}
\label{ex: dual growths}
There are three dual growths associated with the matrix
$
A = \begin{pmatrix}
0 & 1 & 0 \\
1 & 0 & 1 \\
1 & 1 & 1
\end{pmatrix}
$:

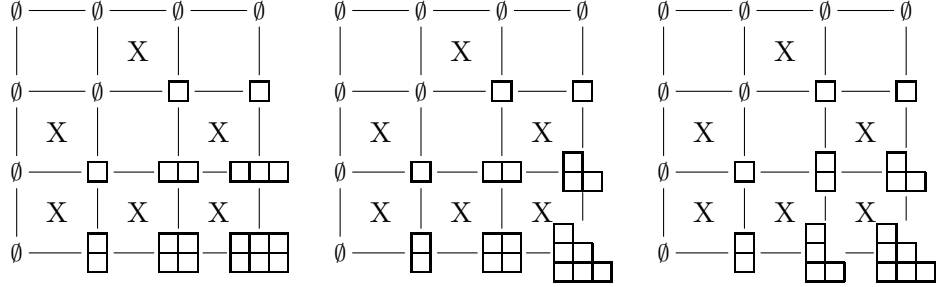
\begin{center}
\begin{tikzpicture}[scale=1.075]

	\node at (1.5,-0.5) {X};
	\node at (0.5,-1.5) {X};
	\node at (2.5,-1.5) {X};
	\node at (0.5,-2.5) {X};
	\node at (1.5,-2.5) {X};
	\node at (2.5,-2.5) {X};
	
	\draw (0.15,0) -- (0.85,0);
	\draw (1.15,0) -- (1.85,0);
	\draw (2.15,0) -- (2.85,0);
	\draw (0.15,-1) -- (0.85,-1);
	\draw (1.15,-1) -- (1.8,-1);
	\draw (2.2,-1) -- (2.8,-1);
	\draw (0.15,-2) -- (0.8,-2);
	\draw (1.2,-2) -- (1.7,-2);
	\draw (2.3,-2) -- (2.6,-2);
	\draw (0.15,-3) -- (0.8,-3);
	\draw (1.2,-3) -- (1.7,-3);
	\draw (2.3,-3) -- (2.6,-3);
	
	\draw (0,-0.2) -- (0,-0.8);
	\draw (0,-1.2) -- (0,-1.8);
	\draw (0,-2.2) -- (0,-2.8);
	\draw (1,-0.2) -- (1,-0.8);
	\draw (1,-1.2) -- (1,-1.8);
	\draw (1,-2.2) -- (1,-2.7);
	\draw (2,-0.2) -- (2,-0.8);
	\draw (2,-1.2) -- (2,-1.8);
	\draw (2,-2.2) -- (2,-2.7);
	\draw (3,-0.2) -- (3,-0.8);
	\draw (3,-1.2) -- (3,-1.8);
	\draw (3,-2.2) -- (3,-2.7);

	\footnotesize
	\foreach \x in {0,...,3}
		\node at (\x,0) {$\emptyset$};
	
	\node at (0,-1) {$\emptyset$};	
	\node at (1,-1) {$\emptyset$};
	\node at (2,-1) {$\ytableaushort{\empty}$};
	\node at (3,-1) {$\ytableaushort{\empty}$};
	
	\node at (0,-2) {$\emptyset$};	
	\node at (1,-2) {$\ytableaushort{\empty}$};
	\node at (2,-2) {$\ytableaushort{\empty\empty}$};	
	\node at (3,-2) {$\ytableaushort{\empty\empty\empty}$};	
	
	\node at (0,-3) {$\emptyset$};	
	\node at (1,-3) {$\ytableaushort{\empty,\empty}$};
	\node at (2,-3) {$\ytableaushort{\empty\empty,\empty\empty}$};
	\node at (3,-3) {$\ytableaushort{\empty\empty\empty,\empty\empty\empty}$};

	\begin{scope}[xshift=4cm]
	\normalsize
	\node at (1.5,-0.5) {X};
	\node at (0.5,-1.5) {X};
	\node at (2.5,-1.5) {X};
	\node at (0.5,-2.5) {X};
	\node at (1.5,-2.5) {X};
	\node at (2.5,-2.5) {X};
	
	\draw (0.15,0) -- (0.85,0);
	\draw (1.15,0) -- (1.85,0);
	\draw (2.15,0) -- (2.85,0);
	\draw (0.15,-1) -- (0.85,-1);
	\draw (1.15,-1) -- (1.8,-1);
	\draw (2.2,-1) -- (2.8,-1);
	\draw (0.15,-2) -- (0.8,-2);
	\draw (1.2,-2) -- (1.7,-2);
	\draw (2.3,-2) -- (2.7,-2);
	\draw (0.15,-3) -- (0.8,-3);
	\draw (1.2,-3) -- (1.7,-3);
	\draw (2.3,-3) -- (2.6,-3);
	
	\draw (0,-0.2) -- (0,-0.8);
	\draw (0,-1.2) -- (0,-1.8);
	\draw (0,-2.2) -- (0,-2.8);
	\draw (1,-0.2) -- (1,-0.8);
	\draw (1,-1.2) -- (1,-1.8);
	\draw (1,-2.2) -- (1,-2.7);
	\draw (2,-0.2) -- (2,-0.8);
	\draw (2,-1.2) -- (2,-1.8);
	\draw (2,-2.2) -- (2,-2.7);
	\draw (3,-0.2) -- (3,-0.8);
	\draw (3,-1.2) -- (3,-1.7);
	\draw (3,-2.25) -- (3,-2.6);

	\footnotesize
	\foreach \x in {0,...,3}
		\node at (\x,0) {$\emptyset$};
	
	\node at (0,-1) {$\emptyset$};	
	\node at (1,-1) {$\emptyset$};
	\node at (2,-1) {$\ytableaushort{\empty}$};
	\node at (3,-1) {$\ytableaushort{\empty}$};
	
	\node at (0,-2) {$\emptyset$};	
	\node at (1,-2) {$\ytableaushort{\empty}$};
	\node at (2,-2) {$\ytableaushort{\empty\empty}$};	
	\node at (3,-2) {$\ytableaushort{\empty,\empty\empty}$};	
	
	\node at (0,-3) {$\emptyset$};	
	\node at (1,-3) {$\ytableaushort{\empty,\empty}$};
	\node at (2,-3) {$\ytableaushort{\empty\empty,\empty\empty}$};
	\node at (3,-3) {$\ytableaushort{\empty,\empty\empty,\empty\empty\empty}$};
	\end{scope}

	\begin{scope}[xshift=8cm]
	\normalsize
	\node at (1.5,-0.5) {X};
	\node at (0.5,-1.5) {X};
	\node at (2.5,-1.5) {X};
	\node at (0.5,-2.5) {X};
	\node at (1.5,-2.5) {X};
	\node at (2.5,-2.5) {X};
	
	\draw (0.15,0) -- (0.85,0);
	\draw (1.15,0) -- (1.85,0);
	\draw (2.15,0) -- (2.85,0);
	\draw (0.15,-1) -- (0.85,-1);
	\draw (1.15,-1) -- (1.8,-1);
	\draw (2.2,-1) -- (2.8,-1);
	\draw (0.15,-2) -- (0.8,-2);
	\draw (1.2,-2) -- (1.8,-2);
	\draw (2.2,-2) -- (2.7,-2);
	\draw (0.15,-3) -- (0.8,-3);
	\draw (1.2,-3) -- (1.7,-3);
	\draw (2.25,-3) -- (2.6,-3);
	
	\draw (0,-0.2) -- (0,-0.8);
	\draw (0,-1.2) -- (0,-1.8);
	\draw (0,-2.2) -- (0,-2.8);
	\draw (1,-0.2) -- (1,-0.8);
	\draw (1,-1.2) -- (1,-1.8);
	\draw (1,-2.2) -- (1,-2.7);
	\draw (2,-0.2) -- (2,-0.8);
	\draw (2,-1.2) -- (2,-1.7);
	\draw (2,-2.3) -- (2,-2.6);
	\draw (3,-0.2) -- (3,-0.8);
	\draw (3,-1.2) -- (3,-1.7);
	\draw (3,-2.3) -- (3,-2.6);

	\footnotesize
	\foreach \x in {0,...,3}
		\node at (\x,0) {$\emptyset$};
	
	\node at (0,-1) {$\emptyset$};	
	\node at (1,-1) {$\emptyset$};
	\node at (2,-1) {$\ytableaushort{\empty}$};
	\node at (3,-1) {$\ytableaushort{\empty}$};
	
	\node at (0,-2) {$\emptyset$};	
	\node at (1,-2) {$\ytableaushort{\empty}$};
	\node at (2,-2) {$\ytableaushort{\empty,\empty}$};	
	\node at (3,-2) {$\ytableaushort{\empty,\empty\empty}$};	
	
	\node at (0,-3) {$\emptyset$};	
	\node at (1,-3) {$\ytableaushort{\empty,\empty}$};
	\node at (2,-3) {$\ytableaushort{\empty,\empty,\empty\empty}$};
	\node at (3,-3) {$\ytableaushort{\empty,\empty\empty,\empty\empty\empty}$};
	\end{scope}

\end{tikzpicture}
\end{center}

\noindent Above we represent the partitions $\Lambda_{i,j}$ by their Young diagrams. Also, for purposes of readability, we represent the 1's in $A$ with $X$'s and omit the 0's.
\end{ex}

Given a dual growth $\Lambda$, let $P(\Lambda)$ be the semistandard tableau of shape $\la = \Lambda_{m,n}$ determined by the last column
\[
\emptyset = \Lambda_{0,n} \prec \Lambda_{1,n} \prec \cdots \prec \Lambda_{m,n} = \la,
\]
and let $Q(\Lambda)$ be the dual semistandard tableau of shape $\la$ determined by the last row
\[
\emptyset = \Lambda_{m,0} \prec' \Lambda_{m,1} \prec' \cdots \prec' \Lambda_{m,n} = \la.
\]
For example, the second dual growth in Example \ref{ex: dual growths} has
\[
P(\Lambda) = \young(123,23,3) \,, \qquad\qquad Q(\Lambda) = \young(123,12,3) \,.
\]
If $\Lambda$ is associated with $A$ and $P = P(\Lambda), Q = Q(\Lambda)$, we write
\[
\Lambda : A \rightarrow (P,Q),
\]
and say that ``$\Lambda$ is a dual growth from $A$ to $(P,Q)$.'' In general, there may be multiple dual growths from $A$ to $(P,Q)$.

By definition, each square in a dual growth has the form
\begin{equation}
\label{eq: dual growth square}
%
\begin{tikzpicture}[baseline=-3.5ex]
\draw (-0.2,-.1) --node[left]{\rotatebox{-90}{$\prec$}} (-0.2,-.9);
\draw (1.2,-.1) --node[right]{\rotatebox{-90}{$\prec$}} (1.2,-.9);
\draw (0.1,0.2) --node[above]{$\prec'$} (.9,0.2);
\draw (.1,-1.2) --node[below]{$\prec'$} (.9,-1.2);
\node at (-0.2,0.2) {$\mu$};
\node at (1.2,0.2) {$\rho$};
\node at (-0.2,-1.15) {$\la$};
\node at (1.2,-1.2) {$\nu$};
\node at (.5,-.5) {$a$};
\end{tikzpicture}
\quad \text{with} \quad a \in \{0,1\}, \quad \nu \in \mc{U}^k(\la,\rho), \quad \mu \in \mc{D}^{k-a}(\la,\rho),
\end{equation}
where $k = |\nu/(\la \cup \rho)| = |(\la \cap \rho)/\mu| + a$.

\begin{defn}
A set of \deff{local dual growth rules} $F_\bullet$ is a choice of bijections
\[
F_{\lambda,\rho,k}: \D^{k-1}(\la,\rho) \cup \D^k(\la,\rho) \rightarrow \U^k(\la,\rho),
\]
for all ordered pairs of partitions $(\la,\rho)$, and all $k \geq 0$. A dual growth $\Lambda$ is an \deff{$F_\bullet$-dual growth diagram} if each square satisfies $\nu = F_{\la,\rho,|(\la \cap \rho)/\mu| + a}(\mu)$, where $a,\mu,\la,\rho,\nu$ are as in \eqref{eq: dual growth square}.
\end{defn}

The key property of local dual growth rules is that they make the process of constructing dual growths deterministic. More precisely, if $\Lambda$ is an $F_\bullet$-dual growth diagram, then for each square, the partition $\nu$ is determined by the partitions $\mu,\la,\rho$ and the value of $a$. Conversely, the pair $(a,\mu)$ is determined by the partitions $\la,\rho,\nu$ 
\begin{equation}
\label{eq: mu a}
\mu = F_{\la,\rho,|\nu/(\la \cup \rho)|}^{-1}(\nu), \qquad\quad a = |\nu| + |\mu| - |\la| - |\rho|.
\end{equation}
This implies that the set of $F_\bullet$-dual growth diagrams is in bijection with both $\{0,1\}$-matrices and pairs $(P,Q)$ of tableaux of the same shape, where $P$ is semistandard and $Q$ is dual semistandard. Given a $\{0,1\}$-matrix $A$, one constructs the unique $F_\bullet$-dual growth diagram associated with $A$ by filling the top row and left column of the grid with the empty partition, and using the bijections $F_{\la,\rho,k}$, along with the positions of 1's in the matrix $A$, to recursively fill in the rest of the grid. Similarly, given a semistandard tableau $P$ and a dual semistandard tableau $Q$ of the same shape, one fills the right and bottom edges of the grid with the chains in Young's lattice corresponding to $P$ and $Q$, respectively, and then recursively fills in the rest of the grid (and the entries of the matrix) using \eqref{eq: mu a}. Thus, each set of local dual growth rules induces a different bijective proof of the dual Cauchy identity.

\begin{rem}
In effect, the framework of dual growth diagrams transforms a bijective proof of the commutation relation \eqref{eq: Schur dual commutation relation} into a bijective proof of the dual Cauchy identity by ``bijectivizing'' the algebraic argument given in \S \ref{sec: up and down operators}.
\end{rem}

We denote by $\RSK^*_{F_\bullet}$ the bijection $A \mapsto (P,Q)$ induced by the local dual growth rules $F_\bullet$. We now explain how to translate $\RSK^*_{F_\bullet}$ into an insertion algorithm. Let $A$ be a $\{0,1\}$-matrix, and $\Lambda$ the $F_\bullet$-dual growth diagram associated with $A$. Suppose that the $j$th column of $\Lambda$ looks like

\begin{center}
\begin{tikzpicture}
\node at (0,0) {$T^{(0)}$};
\node at (0,-1.5) {$T^{(1)}$};
\node at (0,-2*1.5) {$T^{(m-1)}$};
\node at (0,-3*1.5) {$T^{(m)}$};

\node at (1.7,0) {$\wh T^{(0)}$};
\node at (1.7,-1.5) {$\wh T^{(1)}$};
\node at (1.9,-2*1.5) {$\wh T^{(m-1)}$};
\node at (1.7,-3*1.5) {$\wh T^{(m)}$};

\foreach \x in {0,...,3}{
		\draw (0.6,-1.5*\x) -- (1.1,-1.5*\x);
	}
	
\draw (0,-0.45) -- (0,.45-1.5*1);
\draw (0,-0.45-1.5*2) -- (0,.45-1.5*3);
\draw (1.5,-0.4) -- (1.5,.45-1.5*1);
\draw (1.5,-0.4-1.5*2) -- (1.5,.45-1.5*3);

\draw [dotted, thick] (1.5*0.5,-1.5-.3) -- (1.5*0.5,-1.5*2+.3);
\node at (1.5*0.5,-1.5*0.5) {$A_{1,j}$};
\node at (1.5*0.5,-1.5*2.5) {$A_{m,j}$};
\end{tikzpicture}
\end{center}
for semistandard tableaux $T$ and $\wh{T}$. Let $\{i_1 < i_2 < \cdots < i_r\}$ be the set of rows $i$ for which $A_{i,j} = 1$.

\begin{lem}
\label{lem: growth insertion}
In the situation described above, the tableau $\wh{T}$ is obtained as follows:
\begin{itemize}
\item Initialize the \deff{insertion queue} as the multiset\footnote{Note that at the beginning the insertion queue is actually a set, however it may contain elements with multiplicity at a later point of the algorithm.} $\{i_1,\ldots,i_r\}$.
\item Let $i$ be the smallest integer of the insertion queue. Denote by $C$ the set of cells of $F_{\la,\rho,k}(\mu)/(\la\cup\rho)$ where $\la=T^{(i)}$, $\rho=\wh{T}^{(i-1)}$, $\mu=T^{(i-1)}$ and $k=|(\la\cap\rho)/\mu|+A_{i,j}$. Place all $i$'s of the insertion queue into the cells of $C$, delete all $i$'s from the insertion queue and add all entries which have been replaced (bumped) in the current step to the insertion queue.
\item Repeat the previous step until the insertion queue is empty.
\end{itemize}
\end{lem}

\begin{proof}
For the proof we regard the following variation of the above algorithm: in the $i$-th insertion step we apply the above procedure if $i$ is in the insertion queue and do nothing otherwise.

By definition $T^{(i)}$ is the subtableau of $T$ with entries at most $i$ and $T^{(i)}/T^{(i-1)}$ is the set of cells with entry $i$ in $T$. Hence $(T^{(i)}\cap \wh{T}^{(i-1)}) /T^{(i-1)}$ is the set of cells where an entry $i$ was bumped during the first $(i-1)$ insertion steps. Since each bumped entry is added to the insertion queue and the insertion queue contains at the beginning all $i$ with $A_{i,j}=1$, the multiplicity of an entry $i$ in the insertion queue after the $(i-1)$-st insertion step is equal to $k=|(T^{(i)}\cap \wh{T}^{(i-1)}) /T^{(i-1)}|+A_{i,j}$. By definition of the local $F_\bullet$-dual growth rules, we have $\wh{T}^{(i)}=T^{(i)}\cup \wh{T}^{(i-1)}$ if $k=0$, i.e., we do nothing in this step. If $i$ is in the insertion queue, we have $k>0$ and by definition of the local $F_\bullet$- dual growth rules $\wh{T}^{(i)}=F_{T^{(i)},\wh{T}^{(i-1)},k}(T^{(i-1)})$. This implies $k=|T^{(i)}/(T^{(i)}\cup \wh{T}^{(i-1)})|$, i.e, we fill each cell of $C$ by one of the $i$ entries of the insertion queue. This proves the claim.
\end{proof}

If $T$ is a semistandard Young tableau and $i_1<\cdots<i_r$ are integers, we define the \deff{$F_\bullet$-insertion} of $i_1,\ldots,i_r$ into $T$ by the algorithm of Lemma \ref{lem: growth insertion}. For a given $\{0,1\}$-matrix $A$ denote by $i_{1}^{(j)}<\cdots<i_{r_j}^{(j)}$ the rows for which $A$ has a $1$ entry in the $j$-th column. It follows from the preceding discussion that if $\RSK^*_{F_\bullet}(A) = (P,Q)$, then $P$ can be obtained by the successive $F_\bullet$-insertion of $i_{1}^{(j)},\ldots,i_{r_j}^{(j)}$, starting with $j=1$, into the empty tableau, and $Q$ records the growth of $P$, just as for usual dual RSK insertion.

We call the $F_\bullet$-insertion \deff{traceable} if we can keep track of where a bumped entry in the above algorithm is inserted again and instead of ``simultaneously'' inserting $i_1,\ldots,i_r$ we can apply the insertion process for $i_1$, then for $i_2$ etc. sequentially. This means that in order to obtain the insertion tableau, we can replace a column with $r$ entries $1$ in the rows $i_1< \cdots < i_r$ by $r$ columns where we have a $1$ in the $i_j$-th row of the $j$-th new column. Analogously we define the $F_\bullet$-insertion \deff{reverse traceable} if the sequential insertion is in the reverse ordering, i.e. first insert $i_r$, then $i_{r-1}$ until finally $i_1$.
 It is not difficult to convince oneself by using \eqref{eq:recursion for F bijection}, that the dual column insertion $F_\bullet^\col$ is traceable while the dual row insertion $F_\bullet^\row$ is reverse traceable. Hence it is immediate that they  correspond to the usual row or column insertion version of  $\RSK^*$.



\begin{ex}
\label{ex: column insertion dual growth diagram}
The $F_\bullet^{\col}$-dual growth diagram associated with the matrix
\[
A=\begin{pmatrix}
1 & 0 & 0 & 0 & 1 \\ 
1 & 1 & 0 & 1 & 0 \\ 
0 & 1 & 1 & 0 & 0
\end{pmatrix}
\]
is shown below. As in Example \ref{ex: dual growths}, we omit the 0's in the permutation matrix, and write $X$ instead of 1.


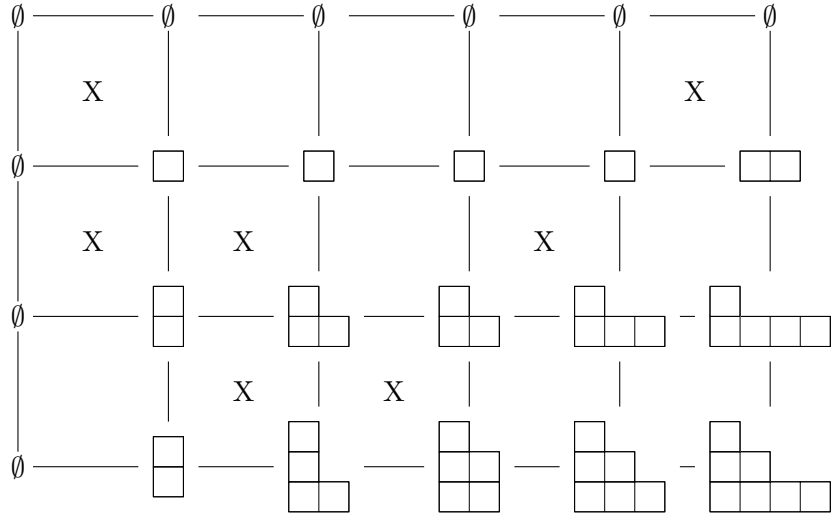
\begin{center}
\begin{tikzpicture}[scale=2]
\Yboxdim{.2cm}
\node at (0,3) {$\emptyset$};
\node at (0,2) {$\emptyset$};
\node at (0,1) {$\emptyset$};
\node at (0,0) {$\emptyset$};
\node at (1,3) {$\emptyset$};
\node at (2,3) {$\emptyset$};
\node at (3,3) {$\emptyset$};
\node at (4,3) {$\emptyset$};
\node at (5,3) {$\emptyset$};
\node at (.5,2.5) {X};
\node at (.5,1.5) {X};
\node at (1.5,1.5) {X};
\node at (1.5,.5) {X};
\node at (2.5,.5) {X};
\node at (3.5,1.5) {X};
\node at (4.5,2.5) {X};
\tyng(.9cm,1.9cm,1);
\tyng(1.9cm,1.9cm,1);
\tyng(2.9cm,1.9cm,1);
\tyng(3.9cm,1.9cm,1);
\tyng(4.8cm,1.9cm,2);
\tyng(.9cm,.8cm,1,1);
\tyng(1.8cm,.8cm,2,1);
\tyng(2.8cm,.8cm,2,1);
\tyng(3.7cm,.8cm,3,1);
\tyng(4.6cm,.8cm,4,1);
\tyng(.9cm,-.2cm,1,1);
\tyng(1.8cm,-.3cm,2,1,1);
\tyng(2.8cm,-.3cm,2,2,1);
\tyng(3.7cm,-.3cm,3,2,1);
\tyng(4.6cm,-.3cm,4,2,1);

\draw (0.1,0) -- (.8,0);
\draw (0.1,1) -- (.8,1);
\draw (0.1,2) -- (.8,2);
\draw (0.1,3) -- (.9,3);
\draw (1.2,0) -- (1.7,0);
\draw (1.2,1) -- (1.7,1);
\draw (1.2,2) -- (1.8,2);
\draw (1.2,3) -- (1.9,3);
\draw (2.3,0) -- (2.7,0);
\draw (2.3,1) -- (2.7,1);
\draw (2.2,2) -- (2.8,2);
\draw (2.2,3) -- (2.9,3);
\draw (3.3,0) -- (3.6,0);
\draw (3.3,1) -- (3.6,1);
\draw (3.2,2) -- (3.8,2);
\draw (3.2,3) -- (3.9,3);
\draw (4.4,0) -- (4.5,0);
\draw (4.4,1) -- (4.5,1);
\draw (4.2,2) -- (4.7,2);
\draw (4.2,3) -- (4.9,3);

\draw (0,0.1) -- (0,0.9);
\draw (0,1.1) -- (0,1.9);
\draw (0,2.1) -- (0,2.9);
\draw (1,0.3) -- (1,0.7);
\draw (1,1.3) -- (1,1.8);
\draw (1,2.2) -- (1,2.9);
\draw (2,0.4) -- (2,0.7);
\draw (2,1.3) -- (2,1.8);
\draw (2,2.2) -- (2,2.9);
\draw (3,0.4) -- (3,0.7);
\draw (3,1.3) -- (3,1.8);
\draw (3,2.2) -- (3,2.9);
\draw (4,0.4) -- (4,0.7);
\draw (4,1.3) -- (4,1.8);
\draw (4,2.2) -- (4,2.9);
\draw (5,0.4) -- (5,0.7);
\draw (5,1.3) -- (5,1.8);
\draw (5,2.2) -- (5,2.9);

\end{tikzpicture}
\end{center}
Reading the last column and last row of this diagram, we find
\[
P = \young(1122,23,3) \;, \qquad Q= \young(1245,13,2) \;, 
\]
in agreement with Example \ref{ex: dual RSK insertion}.
\end{ex}

Given a set of local dual growth rules $F_{\bullet}$ we define the \deff{transpose} $F_{\bullet}^\prime$ of the local dual growth rules by
\begin{equation}
\label{eq:transpose growth rule}
F^\prime_{\rho^\prime,\la^\prime,k}(\mu^\prime) = \nu^\prime \qquad :\Leftrightarrow \qquad
F_{\la,\rho,k}(\mu) = \nu.
\end{equation}
It is immediate that the transpose $\left(F^{\col}_{\bullet}\right)^\prime$ of the column insertion is the row insertion $F^{\row}_{\bullet}$ and vice-versa. 
By definition we obtain the following symmetry for the ${F_\bullet}$-$\RSK^*$
\begin{equation}
\label{eq:sym of dual RSK}
\RSK^*_{F_\bullet}(A) = (P,Q) \qquad \Leftrightarrow \qquad
\RSK^*_{F_\bullet^\prime}(A^T) = (Q^\prime,P^\prime).
\end{equation}

\section{Background on Macdonald polynomials}
\label{sec:Mac}
In this section we review certain basic properties of Macdonald polynomials, following \cite[Ch. VI]{Mac}.\\

The \deff{Macdonald symmetric functions} $P_\la(\mb{x}; q,t)$ are symmetric functions in an infinite set of variables $\mb{x} = (x_1, x_2, \ldots)$ with coefficients in the field $\bb{Q}(q,t)$ of rational functions in two additional variables $q$ and $t$. They were originally defined as the orthogonal basis obtained by applying the Gram--Schmidt orthogonalization procedure to the basis of monomial symmetric functions (ordered by dominance order) with respect to a certain inner product $\langle \cdot, \cdot \rangle_{q,t}$ that depends on $q$ and $t$. We denote by $Q_\la(\mb{x};q,t)$ the elements of the basis dual to the $P_\la(\mb{x}; q,t)$. We usually refer to the symmetric functions $P_\la$ and $Q_\la$ as \deff{Macdonald polynomials}, even though they are not polynomials over any ring. 
Macdonald polynomials generalize many families of symmetric functions. Of importance to this paper are the \deff{$q$-Whittaker functions} $W_\la(\mathbf{x};q)$, the \deff{Hall--Littlewood functions} $P_\la(\mathbf{x};t)$, the \deff{Jack polynomials} $P_\la(x;\alpha)$, and the Schur functions $s_\la(\mathbf{x})$, which are obtained from the Macdonald polynomials by
\begin{align*}
W_\la(\mathbf{x};q) &= P_\la(\mathbf{x};q,0), \\
P_\la(\mathbf{x};t) &= P_\la(\mathbf{x};0,t), \\
P_\la(x;\alpha) &= \lim_{t\rightarrow 1}P_\la(\mathbf{x};t^\alpha,t),\\
s_\la(\mathbf{x}) &= P_\la(\mathbf{x};q,q).
\end{align*}

It follows easily from the definition of the inner product $\langle \cdot, \cdot \rangle_{q,t}$ that the Macdonald polynomials satisfy a generalization of the Cauchy identity for Schur functions.

\begin{thm}[{\cite[Ch. VI (4.13)]{Mac}}]
\label{thm_Cauchy_Mac}
Let $\tb{x} = (x_1, x_2, \ldots)$ and $\tb{y} = (y_1, y_2, \ldots)$ be two sets of variables. Then
\begin{equation}
\label{eq_Cauchy_Mac}
\prod_{i,j \geq 1} \dfrac{(tx_iy_j ; q)_\infty}{(x_iy_j ; q)_\infty} = \sum_\la P_\la(\tb{x};q,t) Q_\la(\tb{y};q,t),
\end{equation}
where $(\alpha;q)_\infty = (1-\alpha)(1-\alpha q)(1-\alpha q^2) \cdots$ is the \deff{infinite $q$-Pochhammer symbol}.
\end{thm}
Using the automorphism $\omega_{q,t}$ which can be defined as
\[
\omega_{q,t}P_\la(\x;q,t) = Q_{\la^\prime}(\x;t,q),
\]
and its properties one obtains immediately the dual Cauchy identity for Macdonald polynomials.
\begin{thm}[{\cite[Ch. VI (5.4)]{Mac}}]
\label{thm_dual_Cauchy Mac}
Let $\tb{x} = (x_1, x_2, \ldots)$ and $\tb{y} = (y_1, y_2, \ldots)$ be two sets of variables. Then
\begin{align}
\label{eq_dual_Cauchy Mac}
\prod_{i,j}(1+x_iy_j) &= \sum_\la P_\la(\x ;q,t) P_{\la'}(\mb{y} ;t,q) \\
&= \sum_\la Q_\la(\x ;q,t) Q_{\la'}(\mb{y} ;t,q).
\end{align}
\end{thm}

With a good deal of effort, Macdonald was able to derive explicit formulas for the monomial expansions of $P_\la$ and $Q_\la$ as weighted sums over semistandard Young tableaux. In order to describe these expansions we need some notations. For a partition $\lambda$ and a cell $c$, define
\[
\hl{\la}(c) = 1-q^{a_\la(c)}t^{\ell_\la(c)+1}, \qquad \ha{\la}(c) = 1-q^{a_\la(c)+1}t^{\ell_\la(c)},
\]
if $c \in \la$, and $\hl{\la}(c) = \ha{\la}(c) = 1$ if $c \not \in \la$.  We refer to both of these polynomials as \deff{$(q,t)$-hook-lengths}. Further we need their ratio which is denoted by $b_\lambda(c)$
\[
b_\lambda(c) = \dfrac{\hl{\la}(c)}{\ha{\la}(c)}.
\]
For $\mu \subseteq \la$, define
\[
\psi_{\la/\mu}(q,t) = \prod_{c \in \mc{R}_{\la/\mu} - \mc{C}_{\la/\mu}} \dfrac{b_\mu(c)}{b_\la(c)}, \qquad\qquad \vp^*_{\la/\mu}(q,t) = \prod_{c \in \mc{C}_{\la/\mu} - \mc{R}_{\la/\mu}} \dfrac{b_\la(c)}{b_\mu(c)},
\]
where $\mc{R}_{\la/\mu}$ (resp., $\mc{C}_{\la/\mu}$) is the set of all cells in $\la$  which are in the same row (resp., column) as a cell of $\la/\mu$. This is the definition used by Macdonald\footnote{Contrary to Macdonald \cite[Ch. VI (6.24)]{Mac} we use the symbol $\vp^*$ instead of $\psi^\prime$.} and differs from the one in \cite{AignerFrieden22}.

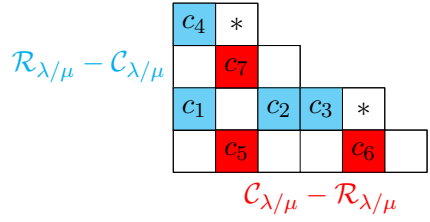
\begin{figure}
\begin{center}
\begin{tikzpicture}
	\Yboxdim{16 pt}
	\Ylinecolour{black}
	\tyng(0,0,6,5,3,2)
	\Ycyan \tgyoung(0pt,0pt,,;:;;,,;)
	\Yred \tgyoung(0pt,0pt,:;::;,,:;)
	\node at (4.5*16pt,1.5*16pt) {$*$};
	\node at (1.5*16pt,3.5*16pt) {$*$};
	\node at (3.5*16pt,-10pt) {\color{red}{$\mc{C}_{\la/\mu}-\mc{R}_{\la/\mu}$}};
	\node at (-32pt,2.5*16pt) {\color{cyan}{$\mc{R}_{\la/\mu}-\mc{C}_{\la/\mu}$}};
	\node at (.5*16pt,3.5*16pt) {$c_4$};
	\node at (3.5*16pt,1.5*16pt) {$c_3$};
	\node at (2.5*16pt,1.5*16pt) {$c_2$};
	\node at (.5*16pt,1.5*16pt) {$c_1$};
	\node at (4.5*16pt,.5*16pt) {$c_6$};
	\node at (1.5*16pt,.5*16pt) {$c_5$};
	\node at (1.5*16pt,2.5*16pt) {$c_7$};
\end{tikzpicture}
\end{center}
\caption{\label{fig: lambda mu} For $\la=(6,5,3,2)$ and $\mu=(6,4,3,1)$ the skew diagram $\la/\mu$ consists of the cells with a star entry. The cells of $\mc{R}_{\lambda/\mu}$ and $\mc{C}_{\lambda/\mu}$ are colored blue and red, respectively.}
\end{figure}
\begin{ex}
For $\la = (6,5,3,2)$ and $\mu = (6,4,3,1)$ as in Figure \ref{fig: lambda mu}, we have
\begin{multline*}
\psi_{\la/\mu}(q,t)= \dfrac{b_{\mu}(c_1)b_{\mu}(c_2)b_{\mu}(c_3)b_{\mu}(c_4)}{b_{\la}(c_1)b_{\la}(c_2)b_{\la}(c_3)b_{\la}(c_4)}\\
=\dfrac{(1-t)^2(1-q^2)^2(1-q t^2)(1-q^3t)(1-q^3t^3)(1-q^5t^2)}
{(1-q)^2(1-qt)^2(1-q^2t)(1-q^2t^2)(1-q^4t^2)(1-q^4t^3)},
\end{multline*}
\[
\vp^*_{\la/\mu}(q,t) = \dfrac{b_\la(c_5)b_\la(c_6)b_\la(c_7)}{b_\mu(c_5)b_\mu(c_6)b_\mu(c_7)}=
\dfrac{(1-q^2)^2(1-qt^2)^2(1-q^5t^2)(1-q^4t^4)}{(1-qt)^2(1-q^2t)^2(1-q^4t^3)(1-q^5t^3)}.
\]
\end{ex}

For a semistandard Young tableau $T$ and a dual semistandard Young tableau $T^*$, define the rational functions $\psi_T(q,t), \vp^*_{T^*}(q,t)$ by
\[
\psi_T(q,t) = \prod_{i \geq 1} \psi_{T^{(i)}/T^{(i-1)}}(q,t), \qquad\qquad
\vp^*_{T^*}(q,t) = \prod_{i \geq 1} \vp_{T^{*(i)}/T^{*(i-1)}}(q,t),
\]
where, as in previous sections, $T^{(i)}$ and $T^{*(i)}$ denotes the shape of the subtableau of $T$ or $T^*$ respectively consisting of entries less than or equal to $i$.

\begin{thm}[{\cite[Ch. VI ($7.13$)]{Mac}}]
\label{thm_Mac_monomial}
The Macdonald polynomial $P_\la(\x; q,t)$ has the following monomial expansions over semistandard Young tableaux of shape $\la$
\begin{equation}
\label{eq:P via SSYT}
P_\la (\x; q,t) = \sum_{T \in \SSYT(\la)} \psi_T(q,t) \tb{x}^T.
\end{equation}
\end{thm}

The above monomial expansion of $P_\la$ over SSYTs immediately implies an expansion over dual SSYTs
\begin{equation}
\label{eq:P via SSYT*}
P_{\la^\prime} (\x; t,q) = \sum_{T^* \in \SSYT^*(\la)} \vp^*_{T^*}(q,t) \tb{x}^{T^*},
\end{equation}
by using that $\vp^*_{\kappa/\rho}(q,t) = \psi_{\kappa'/\rho'}(t,q)$ holds for any $\rho \subset \kappa$. 
The Macdonald polynomials $Q_\la(\x; q,t)$ can be expressed analogously as shown in \eqref{eq:Q via SSYT}.\\

In this paper, we take the somewhat unusual perspective of viewing Theorem \ref{thm_Mac_monomial} as the \emph{definition} of the Macdonald polynomials $P_\la$. The theory of Schur functions can be developed in elegant combinatorial fashion by taking the monomial expansion over semistandard Young tableaux as the starting point (for example, this is Stanley's approach in \cite[Ch. 7]{EC2}). We believe that trying to mimic this approach in the more general Macdonald setting will lead to interesting combinatorial and probabilistic results.

\subsection{Up and down operators for Macdonald polynomials}
\label{sec_qt_up_down}

In \S \ref{sec: up and down operators} and \S \ref{sec: growth diagrams} we showed that the dual Cauchy identity for Schur polynomials can be proven combinatorially by ``bijectivizing” the algebraic proof via up and dual down operators. We now the take the first step towards an analogous proof in the Macdonald setting by providing the generalized algebraic setting.\\

 We define the \deff{$(q,t)$-up operator} and \deff{$(q,t)$-dual down operator} as
\[
U_z(q,t) \la = \sum_{\nu \succ \la} z^{|\nu/\la|} \psi_{\nu/\la}(q,t) \; \nu, \quad\quad 
D_z^*(q,t) \la = \sum_{\mu \prec^\prime \la} z^{|\la/\mu|} \vp_{\la/\mu}^*(q,t) \; \mu.
\]

\begin{thm}
\label{thm_dual_DU_UD}
The $(q,t)$-up and $(q,t)$-dual down operators satisfy the commutation relation
\begin{equation}
\label{eq_dual_DU_UD}
D^*_y(q,t) U_x(q,t) = (1+xy) U_x(q,t) D^*_y(q,t).
\end{equation}
\end{thm}

This is in fact equivalent to the skew dual Cauchy identity in the special case of $\x=(x_1)$ and $\y=(y_1)$, compare to \cite[Ch. VI, Ex 6(c)]{Mac} or \eqref{eq:skew dual cauchy}, but we will take the opposite perspective: we seek to prove this identity with a probabilistic bijection, and then use it to deduce the dual Cauchy identity (which implies the skew dual Cauchy identity by a standard argument).

The commutation relation of the up and dual down operator is equivalent to the skew version of the dual Cauchy identity in the special case of $\x=(x_1)$ and $\y=(y_1)$, compare to \cite[Ch. VI, Ex 6(c)]{Mac}.

\begin{cor}
Let $\mb{x} = (x_1, \ldots, x_m), \mb{y} = (y_1, \ldots, y_{n})$ be two finite sets of indeterminates. Then
\[
\prod_{1 \leq i \leq m \atop 1 \leq j \leq n}(1+x_iy_j) = \sum_\la P_\la(\mb{x};q,t) P_{\la'}(\mb{y};t,q),
\]
where the sum is over all partitions $\la$.
\end{cor}
\begin{proof}
It is immediate by the definition of the $(q,t)$-up and $(q,t)$-dual down operator and the monomial expansions of $P_\la$ in \eqref{eq:P via SSYT} and \eqref{eq:P via SSYT*} that
\begin{align*}
P_{\la}(\mb{x}; q,t) &= \langle U_{x_m}(q,t)\cdots U_{x_1}(q,t) \emptyset, \la \rangle, \\
P_{\la'}(\mb{y}; t,q) &= \langle D^*_{y_1}(q,t) \cdots D^*_{y_{n}}(q,t) \la, \emptyset \rangle.
\end{align*}
Now we just repeatedly apply \eqref{eq_dual_DU_UD} in the usual way.
\end{proof}

In \S \ref{sec: up and down operators} we defined the two sets $\U^k(\la,\rho):=\{ \nu | \nu \succ^\prime \la, \nu \succ \rho, |\nu / (\la \cup \rho)|=k\}$ and $\D^k(\la,\rho) := \{ \mu | \mu \prec \la, \mu \prec^\prime \rho, |(\la \cap \rho)/\mu|=k\}$. In order to prove the commutation relation \eqref{eq_dual_DU_UD} we need to show
\begin{equation}
\label{eq_dual_la_rho}
\sum_{\mu \in \D^k(\la,\rho) \cup \D^{k-1}(\la,\rho)} \psi_{\la/\mu}(q,t) \vp^*_{\rho/\mu}(q,t)
= \sum_{\nu \in \U^k(\la,\rho)} \psi_{\nu/\rho}(q,t) \vp^*_{\nu/\la}(q,t),
\end{equation}
for any partitions $\la$ and $\rho$ and integer $k \geq 0$.
For $\la=\rho$ the partitions $\nu$ and $\mu$ in the above sums have to be \deff{diagonal strips} (i.e., no two boxes in the same row or column) with respect to $\la$, i.e., $\nu /\la$ and $\la/\mu$ are both diagonal strips. When further $k=1$, this is equivalent to the single box case of the $(q,t)$-up-down commutation relation for which we have presented a probabilistic bijection in \cite{AignerFrieden22}. In general, each side of this identity has $\binom{d+1}{k}$ summands, where $d$ is the number of removable inner corners of $\la \cap \rho$.

\begin{ex}
\label{ex:weight for almost rect}
Let $\la=(h^v,1)$ and $\rho=(h+1,h^{v-1})$ be almost rectangular partitions with $h \geq 3$ and $v \geq 2$. The partitions $\mu \in \D^1(\la,\rho) \cup \D^{0}(\la,\rho)$ and $\nu \in \U^1(\la,\rho)$ are shown below (for $h=8, v=4$) together with their weights $\psi_{\la/\mu}(q,t) \vp^*_{\rho/\mu}(q,t)$ and $\psi_{\nu/\rho}(q,t) \vp^*_{\nu/\la}(q,t)$.

\begin{center}
\begin{tikzpicture}[scale=0.6]
\begin{scope}[xshift=15cm]
\draw (-3.5,7) node[left]{$\psi_{\la/\mu} \vp^*_{\rho/\mu}$};
\draw (-5,5) node{$1$};
\draw (-5,1) node{$\dfrac{(1-t^{v-1})(1-q^{h-1})}{(1-q t^{v-2})(1-q^{h-2}t)}$};
\end{scope}

\draw (4.5,7) node[left]{$\D^1(\la,\rho) \cup \D^{0}(\la,\rho)$};
\draw[step=0.5, gray, thin] (0,4) grid (4,6);
\draw (0,4) rectangle (4,6);
\draw (2,3.6) node{$h$};
\draw (-0.4,5) node{$v$};
\draw[step=0.5, gray, thin] (0,0) grid (4,1.5);
\draw[step=0.5, gray, thin] (0,1.5) grid (3.5,2);
\draw (0,0) --(4,0) -- (4,1.5)  -- (3.5,1.5) -- (3.5,2) -- (0,2) -- (0,0);
\draw (2,-0.4) node{$h$};
\draw (-0.4,1) node{$v$};

\draw[dashed] (-.5,-1.5) -- (13,-1.5);

\begin{scope}[yshift=-10cm]
\begin{scope}[xshift=-7cm]
\draw (8,7) node[right]{$\U^1(\la,\rho)$};
\draw[step=0.5, gray, thin] (7,4) grid (11,6);
\draw[step=0.5, gray, thin] (7,6) grid (7.5,6.5);
\draw[step=0.5, gray, thin] (11,4) grid (11.5,5);
\draw (7,4) -- (11.5,4) -- (11.5,4.5) -- (11.5,5)  -- (11,5) -- (11,6) -- (7.5,6) -- (7.5,6.5) -- (7,6.5) -- (7,4);
\draw (9,3.6) node{$h$};
\draw (6.6,5) node{$v$};
\draw[step=0.5, gray, thin] (7,0) grid (11,2);
\draw[step=0.5, gray, thin] (7,2) grid (8,2.5);
\draw[step=0.5, gray, thin] (11,0) grid (11.5,.5);
\draw (7,0) -- (11.5,0) -- (11.5,.5) -- (11,.5) -- (11,2) -- (8,2) -- (8,2.5) -- (7,2.5) -- (7,0);
\draw (9,-0.4) node{$h$};
\draw (6.6,1) node{$v$};
\end{scope}

\begin{scope}[xshift=-6cm]
\draw (15,7) node[right]{$\psi_{\nu/\rho} \vp^*_{\nu/\la}$};
\draw (16,5) node{$\dfrac{(1-t^{v-1})(1-q^{h} t^{v-2})}{(1-q t^{v-2})(1-q^{h-1} t^{v-1})}$};
\draw (16,1) node{$\dfrac{(1-q^{h-2}t^v)(1-q^{h-1})}{(1-q^{h-2}t)(1-q^{h-1}t^{v-1})}$};
\end{scope}
\end{scope}

\end{tikzpicture}
\end{center}

\end{ex}


\section{Probabilistic growth rules}
\label{sec:probabilities}
\subsection{Probabilistic bijections}
\label{sec:prob bij}
The following definition is due to Bufetov and Petrov \cite{BufetovPetrov19}, although they use the term
``bijectivization'' (or ``coupling'') rather than ``probabilistic bijection.'' \\

Let $X$ and $Y$ be finite sets equipped with weight functions $\omega : X \rightarrow A$, $\ov{\omega} : Y \rightarrow A$, where $A$ is an algebra. A \deff{probabilistic bijection} from $(X,\omega)$ to $(Y,\ov{\omega})$ is a pair of maps $\mc{P},\ov{\mc{P}} : X \times Y \rightarrow A$ satisfying
\begin{enumerate}
\item For each $x \in X$, $\ds \sum_{y \in Y} \mc{P}(x,y) = 1$.
\item For each $y \in Y$, $\ds \sum_{x \in X} \ov{\mc{P}}(x,y) = 1$.
\item For each $x \in X$ and $y \in Y$, $\ds \omega(x)\mc{P}(x,y) = \ov{\mc{P}}(x,y)\ov{\omega}(y)$.
\end{enumerate}

For the remainder of the paper we write $\mc{P}(x \rightarrow y)$ for $\mc{P}(x,y)$ and $\ov{\mc{P}}(x \leftarrow y)$ for $\ov{\mc{P}}(x,y)$. We think of $\mc{P}(x \rightarrow y)$ as the ``probability'' of mapping $x$ to $y$ and call it therefore the \deff{forward probability} and of $\ov{\mc{P}}(x \leftarrow y)$ as the ``probability'' of mapping $y$ ``back'' to $x$ and call it the \deff{backward probability}. We put ``probability'' in quotes because we do not require $\mc{P}(x \rightarrow y), \ov{\mc{P}}(x \leftarrow y) \in [0,1]$ (they need not even be real-valued).
 Condition (1) states that $\mc{P}$ defines a ``probability distribution'' on $Y$ for each $x$, and (2) says that $\ov{\mc{P}}$ defines a ``probability distribution'' on $X$ for each $y$.  We refer to (3) as the \deff{compatibility condition}. \bigskip


The next lemma shows, that analogously to bijections, it satisfies to find a forward probability and a backward probability, i.e., a probabilistic bijection, in order to prove that two weighted sets have equal sums of weights.
 
\begin{lem}
\label{lem:prob bij easy}
If $\mc{P},\ov{\mc{P}}$ is a probabilistic bijection between $(X,\omega)$ and $(Y,\ov{\omega})$, then
\[
\sum_{x \in X} \omega(x) = \sum_{y \in Y} \ov{\omega}(y).
\]
\end{lem}

\begin{proof}
Using properties (1), (3), and (2) successively, we compute
\[
\sum_{x \in X} \omega(x) = \sum_{x \in X} \sum_{y \in Y} \omega(x) \mc{P}(x \rightarrow y) = \sum_{y \in Y} \sum_{x \in X} \ov{\mc{P}}(x \leftarrow y) \ov{\omega}(y) = \sum_{y \in Y} \ov{\omega}(y). \qedhere
\]
\end{proof}

The existence of a probabilistic bijection $\mc{P},\ov{\mc{P}}$ also implies the more refined identities
\[
\sum_{x \in X} \omega(x) \mc{P}(x \rightarrow y) = \ov{\omega}(y), \qquad\qquad \sum_{y \in Y} \ov{\mc{P}}(x \leftarrow y) \ov{\omega}(y) = \omega(x).
\]

If $\omega(x) = \ov{\omega}(y) = 1$ for all $x \in X$ and $y \in Y$, and $f: X \rightarrow Y$ is a bijection, then we may take $\mc{P}(x \rightarrow y) = \ov{\mc{P}}(x \leftarrow y) = \delta_{y,f(x)}$. Thus, the notion of probabilistic bijection generalizes that of bijection, allowing for situations in which $X$ and $Y$ have different cardinalities, or the same cardinality but differently distributed weight functions, etc.
Further we want to point out, that there is an easy connection between the concept of probabilistic bijections and joint distributions, compare for example with \cite[Remark 4.1.4]{AignerFrieden22}.

\subsection{The probabilities}
\label{sec:probabilities definition}

For partitions $\la,\rho,\mu,\nu$ satisfying $\mu \prec \la \prec^\prime \nu$ and $\mu \prec^\prime \rho \prec \nu$ we define the weights
\begin{align*}
\omega_{\la,\rho}(\mu) = \psi_{\la/\mu}(q,t) \vp_{\rho/\mu}^*(q,t), \qquad \qquad
\ov{\omega}_{\la,\rho}(\nu) = \psi_{\nu/\rho}(q,t) \vp_{\nu/\la}^*(q,t).
\end{align*}
Then equation \eqref{eq_dual_la_rho} becomes
\begin{equation}
\label{eq: sum of weights}
\sum_{\mu \in \D^k(\la,\rho) \cup \D^{k-1}(\la,\rho)}  \omega_{\la,\rho}(\mu) =
\sum_{\nu \in \U^k(\la,\rho)} \ov{\omega}_{\la,\rho}(\nu).
\end{equation}
In order to show this equation, it ``suffices'' by Lemma \ref{lem:prob bij easy} to find a forward probability $\mc{P}_{\la,\rho}(\mu \rightarrow \nu)$ and a backward probability $\ov{\mc{P}}_{\la,\rho}(\mu \leftarrow \nu)$ for which we can prove
\begin{align*}
\sum_{\nu \in \U^k(\la,\rho)} \mc{P}_{\la,\rho}(\mu \rightarrow \nu) &= 1
 \qquad \text{ for all } \mu \in \D^k(\la,\rho) \cup \D^{k-1}(\la,\rho),\\
\sum_{\mu \in \D^k(\la,\rho) \cup \D^{k-1}(\la,\rho)} \ov{\mc{P}}_{\la,\rho}(\mu \leftarrow \nu) &= 1
 \qquad \text{ for all } \nu \in \U^k(\la,\rho),
\end{align*}
and 
\[
 \omega_{\la,\rho}(\mu) \mc{P}_{\la,\rho}(\mu \rightarrow \nu)
 = \ov{\omega}_{\la,\rho}(\nu) \ov{\mc{P}}_{\la,\rho}(\mu \leftarrow \nu),
\]
for all $\mu \in \D^k(\la,\rho) \cup \D^{k-1}(\la,\rho)$ and $\nu \in \U^k(\la,\rho)$. In case we want to emphasize the dependence of the probabilities on $q$ and $t$, we write $\mc{P}_{\la,\rho}(\mu \rightarrow \nu)[q,t]$ and $\ov{\mc{P}}_{\la,\rho}(\mu \leftarrow \nu)[q,t]$. 
Before we can define these probabilities, we need to introduce some notations.\bigskip

Denote by $d$ the number of removable inner corners of $\la \cap \rho$; see \S \ref{sec: up and down operators} for their definition. For a subset $\Ret \subseteq [d] = \{1,2,\ldots d\}$ we define the partition $\mu^{(\Ret)}$ as the partition obtained by removing from $\la \cap \rho$ the $i$-th removable inner corner of $\la \cap \rho$, counted from bottom to top, for all $i\in \Ret$. For a subset $\Set\subseteq [0,d]=\{0,1,\ldots, d\}$ we define $\nu^{(\Set)}$ as the partition obtained by adding to $\la \cup \rho$ the $i$-th addable (``supplementable'') outer corner of $\la \cup \rho$, where we count the addable outer corners again from bottom to top but starting with $0$. 

As we see in a moment, it turns out to be convenient to draw Young diagrams using \deff{Quebecois convention} in which the boxes are right-justified instead of left-justified, i.e., one obtains this new convention by reflecting diagrams in French convention vertically or reflecting diagrams in English convention both vertically and horizontally, see Figure \ref{fig: def of pts}.
We define $R_i$ (resp., $\ov{R}_i$) to be the lower right (resp., upper left) corner of the $i$-th removable inner corner of $\la \cap \rho$, $S_i$ (resp., $\ov{S}_i$) to be the lower right (resp., upper left) corner of the $i$-th addable outer corner of $\la \cup \rho$, and set $I_i = \ov{R}_i$ and $O_i=S_i$. For an example see Figure \ref{fig: def of pts}.	

For the rest of the paper we identify a point with coordinates $(x,y)$ with the monomial\footnote{By using French instead of Quebecois notation, the corresponding monomial would be $q^{-x}t^y$, and $q^{-x}t^{-y}$ for English notation.} $q^x t^y$. In order to determine the coordinates of the above defined points we assume that the cells of Young diagrams are unit squares and define the origin such that all of the above points have integer coordinates. Since the expressions we are interested in are homogeneous rational functions of degree 0 in the above defined points, these expressions are invariant under translation of the points and hence well-defined.
\medskip

For $\Ret \subseteq [d]$ and $\Set \subseteq [0,d]$, we define the probabilities
\begin{align}
\label{eq:forward using quebecois}
\mc{P}_{\la,\rho}(\mu^{(\Ret)} \rightarrow \nu^{(\Set)}) &= 
\prod_{s \in \Set} \frac{\prod\limits_{i \in [d]\setminus \Ret} (S_s-I_i)}{\prod\limits_{j \in [0,d]\setminus \Set} (S_s-O_j)}
\prod_{r \in \Ret} \frac{\prod\limits_{j \in [0,d]\setminus \Set} (R_r-O_j)}{\prod\limits_{i \in [d]\setminus \Ret} (R_r-I_i)},
\\
\label{eq:backward using quebecois}
\ov{\mc{P}}_{\la,\rho}(\mu^{(\Ret)} \leftarrow \nu^{(\Set)}) &=
\prod_{s \in \Set} \frac{\prod\limits_{i \in [d]\setminus \Ret} (\ov{S}_s-I_i)}{\prod\limits_{j \in [0,d]\setminus \Set} (\ov{S}_s-O_j)}
\prod_{r \in \Ret} \frac{\prod\limits_{j \in [0,d]\setminus \Set} (\ov{R}_r-O_j)}{\prod\limits_{i \in [d]\setminus \Ret} (\ov{R}_r-I_i)}.
\end{align}
If $\la, \rho$ are clear from the context, we abbreviate the probabilities by 
\[
p_{\Ret,\Set} = \mc{P}_{\la,\rho}(\mu^{(\Ret)} \rightarrow \nu^{(\Set)}), 
\qquad \qquad \ov{p}_{\Ret,\Set} = \ov{\mc{P}}_{\la,\rho}(\mu^{(\Ret)} \leftarrow \nu^{(\Set)}).
\]
When we want to emphasize the dependence of our probabilities on $q$ and $t$, we write $p_{\Ret,\Set}[q,t]$ or $\mc{P}_{\la,\rho}(\mu^{(\Ret)} \rightarrow \nu^{(\Set)})[q,t]$.

\begin{figure}
\begin{center}
\begin{tikzpicture}[scale=1.5, xscale=-1]
\tyng(0,0,10,9,7,6,6,4,2,1)
\draw[thick, blue, fill=myblue] (4.5,.5) rectangle (5,1);
\draw[thick, blue, fill=myblue] (2.5,2.5) rectangle (2,3);
\draw[thick, blue, fill=myblue] (.5,4.5) rectangle (0,4);
\node at (.25,4.25) {$+$};
\node at (2.25,2.75) {$+$};
\node at (4.75,.75) {$+$};
\draw[thick, orange, fill=myorange] (3.5,1.5) rectangle (3,1);
\draw[thick, orange, fill=myorange] (.5,3.5) rectangle (1,3);
\node at (3.25,1.25) {$-$};
\node at (.75,3.25) {$-$};
\draw[line width=1.75pt] (0,0) -- (5,0) -- (5,.5) -- (4.5,.5) -- (4.5,1) -- (3.5,1) -- (3.5,1.5) -- (3,1.5) -- (3,2.5) -- (2,2.5) -- (2,3) -- (1,3) -- (1,3.5) -- (.5,3.5) -- (.5,4) -- (0,4) -- (0,0) -- (0,1);
\draw[fill, blue] (0,4) circle (1.5pt);
\draw[fill, blue] (.5,4.5) circle (1.5pt);
\draw[fill, blue] (2,2.5) circle (1.5pt);
\draw[fill, blue] (2.5,3) circle (1.5pt);
\draw[fill, blue] (4.5,.5) circle (1.5pt);
\draw[fill, blue] (5,1) circle (1.5pt);
\draw[fill, red] (.5,3) circle (1.5pt);
\draw[fill, red] (1,3.5) circle (1.5pt);
\draw[fill, red] (3,1) circle (1.5pt);
\draw[fill, red] (3.5,1.5) circle (1.5pt);
\node[blue] at (4.5,.5) [below right]  {$S_0 = O_0$};
\node[blue] at (5,1) [above left] {$\ov{S}_0$};
\node[blue] at (2,2.5) [below right]  {$S_1=O_1$};
\node[blue] at (2.5,3) [above left] {$\ov{S}_1$};
\node[blue] at (0,4) [below right]  {$S_2=O_2$};
\node[blue] at (.5,4.5) [above left] {$\ov{S}_2$};
\node[red] at (3,1) [below right]  {$R_1$};
\node[red] at (3.5,1.5) [above left] {$\ov{R}_1=I_1$};
\node[red] at (.5,3) [below right]  {$R_2$};
\node[red] at (1,3.5) [above left] {$\ov{R}_2=I_2$};

\end{tikzpicture}
\caption{\label{fig: def of pts} The partition $\la \cup \rho$ together with the points $I_i,O_j,R_i,\ov{R}_i,S_j$ and $\ov{S}_j$ for $\la=(9,9,7,5,5,4,2)$ and $\rho=(10,7,7,6,6,2,2,1)$ as in Figure~\ref{fig: inner outer corner}.}
\end{center}
\end{figure}
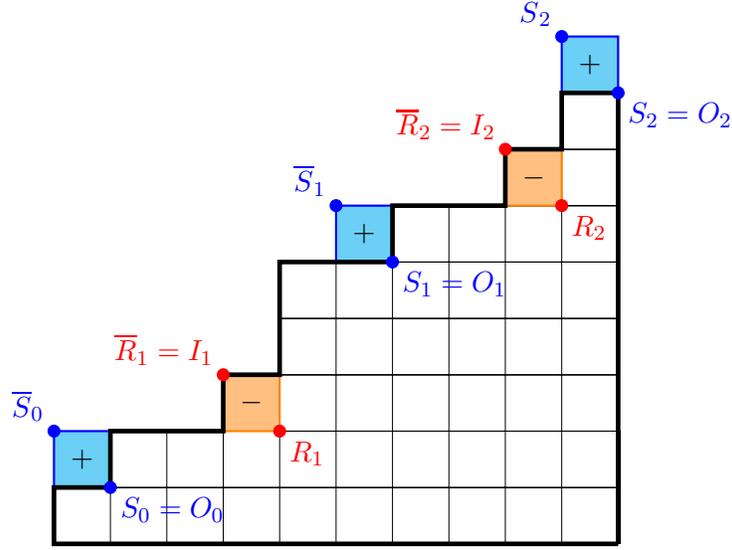

\begin{ex}
\label{ex:qt growth rule}
For $\la=\rho=(2,1)$ the points $I_i,R_i$ and $S_j=O_j$ are shown next
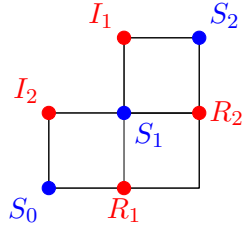
\begin{center}
\begin{tikzpicture}[scale=2, xscale=-1]
\tyng(0,0,2,1)
\draw[fill, blue] (1,0) circle (1.25pt);
\draw[fill, blue] (.5,.5) circle (1.25pt);
\draw[fill, blue] (0,1) circle (1.25pt);
\node[blue] at (1,0) [below left]  {$S_0$};
\node[blue] at (.5,.5) [below right]  {$S_1$};
\node[blue] at (0,1) [above right]  {$S_2$};
\draw[fill, red] (1,.5) circle (1.25pt);
\draw[fill, red] (0,.5) circle (1.25pt);
\draw[fill, red] (.5,0) circle (1.25pt);
\draw[fill, red] (.5,1) circle (1.25pt);
\node[red] at (.5,0) [below]  {$R_1$};
\node[red] at (0,.5) [right]  {$R_2$};
\node[red] at (.5,1) [above left]  {$I_1$};
\node[red] at (1,.5) [above left]  {$I_2$};
\end{tikzpicture}
\end{center}
For $\Ret=\{2\}$ we have the following forward probabilities.
\begin{align*}
p_{\{2\},\{0,1\}} = \mc{P}_{(2,1),(2,1)}\left((2) \rightarrow (3,2) \right)
&= \dfrac{(S_0-I_1)(S_1-I_1)}{(S_0-O_2)(S_1-O_2)} \cdot \dfrac{(R_2-O_2)}{(R_2-I_1)}\\[11pt]
&= q \dfrac{(1-t)^2(1-q)}{(1-q^2 t^2)(1-q t)(1-q^2)}, \\ \\
p_{\{2\},\{1,2\}} = \mc{P}_{(2,1),(2,1)}\left((2) \rightarrow (2,2,1) \right)
&= \dfrac{(S_1-I_1)(S_2-I_1)}{(S_1-O_0)(S_2-O_0)} \cdot \dfrac{(R_2-O_0)}{(R_2-I_1)}\\[11pt]
&= t \dfrac{(1-q)(1-q^2 t)^2}{(1-q t)(1-q^2t^2)(1-q^2)}
,\\ \\
p_{\{2\},\{0,2\}} = \mc{P}_{(2,1),(2,1)}\left((2) \rightarrow (3,1,1) \right)
&= \dfrac{(S_0-I_1)(S_2-I_1)}{(S_0-O_1)(S_2-O_1)} \cdot \dfrac{(R_2-O_1)}{(R_2-I_1)}\\[11pt]
&=\dfrac{(1-t)(1-q^2 t)(1-q)}{(1-q t)^2(1-q^2)}
.
\end{align*}
\end{ex}

For an integer $k \geq 0$ and a set $\Set$, we denote by $\binom{\Set}{k}$ the set of $k$-element subsets of $\Set$.

\begin{thm}
\label{thm:sum to 1}
Let $\la, \rho$ be partitions, $d$ the number of removable inner corners of $\la \cap \rho$, and $k \in [d+1]$. The probabilities satisfy
\begin{align}
\label{eq: sum of forward prob}
\sum_{\Set \in \binom{[0,d]}{k}} p_{\Ret,\Set} = 1 \qquad &\text{ for each } \Ret\in \binom{[d]}{k-1} \cup \binom{[d]}{k},\\
\label{eq: sum of backward prob}
\sum_{\Ret\in \binom{[d]}{k-1} \cup \binom{[d]}{k}} \ov{p}_{\Ret,\Set} = 1 \qquad &\text{ for each } \Set \in \binom{[0,d]}{k}.
\end{align}
\end{thm}

\begin{thm}
\label{thm:weights}
Let $\la, \rho$ be partitions, $d$ the number of removable inner corners of $\la \cap \rho$, and $k \in [d+1]$. For $\Ret\in \binom{[d]}{k-1} \cup \binom{[d]}{k}$ and $\Set \in \binom{[0,d]}{k}$, one has
\[
\frac{\omega_{\la,\rho}(\mu^{(\Ret)})}{\ov{\omega}_{\la,\rho}(\nu^{(\Set)})} = \frac{\ov{p}_{\Ret,\Set}}{p_{\Ret,\Set}}.
\]
\end{thm}

Theorems \ref{thm:sum to 1} and \ref{thm:weights} show that our probabilities define a probabilistic bijection. In \S \ref{sec:interpolation}, we use a generalization of Lagrange interpolation to prove Theorem \ref{thm:sum to 1}. In \S \ref{sec:prob via hook-lengths}, we derive an ``intrinsic'' formula for the probabilities that (mostly) avoids the use of coordinates, and we use this formula to prove Theorem \ref{thm:weights}.


\begin{lem}
\label{lem:inverting q and t}
The probabilities satisfy
\begin{align*}
\mc{P}_{\la,\rho}(\mu \rightarrow \nu)[q^{-1},t^{-1}] = \mc{P}_{\rho',\la'}(\mu' \rightarrow \nu')[t,q], \\
\ov{\mc{P}}_{\la,\rho}(\mu \leftarrow \nu)[q^{-1},t^{-1}] = \ov{\mc{P}}_{\rho',\la'}(\mu' \leftarrow \nu')[t,q].
\end{align*}
\end{lem}
\begin{proof}
Let us assume that the bottom right corner of the partitions $\la,\rho,\nu,\mu$ in Quebecois notation has coordinates $(0,0)$, compare to Figure~\ref{fig: def of pts}. Then we obtain the partitions  $\la^\prime,\rho^\prime,\nu^\prime,\mu^\prime$ by reflecting along the line $y=-x$, i.e., by mapping coordinates $(x,y)$   to $(-y,-x)$. Because of our identification of points with coordinates $(x,y)$ with monomials $q^x t^y$, this reflection corresponds to the variable transformation $q,t \rightarrow t^{-1},q^{-1}$. Hence we have by  \eqref{eq:forward using quebecois} 
\[
\mc{P}_{\la,\rho}(\mu \rightarrow \nu)[q,t]  = \mc{P}_{\rho',\la'}(\mu' \rightarrow \nu')[t^{-1},q^{-1}],
\]
which is equivalent to the assertion. The second statement follows analogously. 
\end{proof}

\subsection{Recovering dual RSK}
\label{sec:recovering dual RSK}

Next we show that in the limits $q,t \rightarrow 0$ and $q,t \rightarrow \infty$, our probabilistic bijection $\mc{P}_{\la,\rho}$ degenerates to the deterministic bijections $F^{\row}_{\la,\rho}$ and $F^{\col}_{\la,\rho}$, respectively. By Lemma \ref{lem:inverting q and t} and the fact that $F^{\row}_{\la,\rho}$ is the transpose of $F^{\col}_{\la,\rho}$, it is enough to consider the case $q,t \rightarrow 0$.

The points $O_j=S_j$ and $I_i=\ov{R}_i$ satisfy
\[
O_0 < I_1 < O_1 < \cdots < I_d < O_d,
\]
where $(x,y) < (x',y')$ means that $x \leq x', y \leq y'$, and at least one inequality is strict. The same inequalities hold if we replace some subset of the $I_i$ by $R_i$, or some subset of the $O_j$ by $\ov{S}_j$. We now rewrite the definitions of $p_{\Ret,\Set}$ and $\ov{p}_{\Ret,\Set}$ by pulling out the smaller point (that is, the monomial with smaller exponents) from each binomial factor. For brevity, we write $i \not \in \Ret$ and $j \not \in \Set$ instead of $i \in [d] \setminus \Ret$ and $j \in [0,d] \setminus \Set$. The result is
\begin{equation}
\label{eq:prob rewritten}
p_{\Ret,\Set} = \tau_{\Ret,\Set} \zeta_{\Ret,\Set}, \qquad\quad \ov{p}_{\Ret,\Set} = \ov{\tau}_{\Ret,\Set} \ov{\zeta}_{\Ret,\Set},
\end{equation}
where $\zeta_{\Ret,\Set}$ is defined as
\begin{equation}
\label{eq:zeta}
\dfrac{\ds \prod_{s \in \Set, i \not \in \Ret\atop s < i} \left(1-\frac{I_i}{S_s}\right)
\prod_{s \in \Set, i \not \in \Ret\atop s \geq i} \left(1-\frac{S_s}{I_i}\right)
\prod_{r \in \Ret, j \not \in \Set \atop r \leq j} \left(1-\frac{O_j}{R_r}\right)
\prod_{r \in \Ret, j \not \in \Set \atop r > j} \left(1-\frac{R_r}{O_j}\right)}
{\ds \prod_{s \in \Set, j \not \in \Set \atop s < j} \left(1-\frac{O_j}{S_s}\right)
\prod_{s \in \Set, j \not \in \Set \atop s > j} \left(1-\frac{S_s}{O_j}\right)
\prod_{r \in \Ret, i \not \in \Ret\atop r < i} \left(1-\frac{I_i}{R_r}\right)
\prod_{r \in \Ret, i \not \in \Ret\atop r > i} \left(1-\frac{R_r}{I_i}\right)},
\end{equation}
\begin{multline}
\label{eq:tau}
\tau_{\Ret,\Set} = \prod_{s \in \Set} (S_s)^{\#\{i \not \in \Ret\mid i > s\} - \#\{j \not \in \Set \mid j > s\}} \prod_{r \in \Ret} (R_r)^{\#\{j \not \in \Set \mid j \geq r\} - \#\{i \not \in \Ret\mid i > r\}} \\
\times \prod_{j \not \in \Set} (-O_j)^{\#\{r \in \Ret\mid r > j\} - \#\{s \in \Set \mid s > j\}} \prod_{i \not \in \Ret} (-I_i)^{\#\{s \in \Set \mid s \geq i\} - \#\{r \in \Ret\mid r > i\}},
\end{multline}
and $\ov{\zeta}_{\Ret,\Set}, \ov{\tau}_{\Ret,\Set}$ are obtained from $\zeta_{\Ret,\Set}, \tau_{\Ret,\Set}$ by replacing all points $R_r$ and $S_s$ with $\ov{R}_r$ and $\ov{S}_s$.
Each binomial factor in $\zeta_{\Ret,\Set}$ and $\ov{\zeta}_{\Ret,\Set}$ is of the form $1-q^xt^y$ for some $x,y \geq 0$, with $x,y$ not both 0. It follows immediately that
\begin{equation}
\label{eq:zeta at 0}
\lim_{q,t \rightarrow 0} \zeta_{\Ret,\Set} = \lim_{q,t \rightarrow 0} \ov{\zeta}_{\Ret,\Set} = 1.
\end{equation}
The analysis of the monomials $\tau_{\Ret,\Set}$ and $\ov{\tau}_{\Ret,\Set}$ is the content of the next two lemmas.

\begin{lem}
\label{lem:tau}
There are integers $x,y \geq 0$ such that
\[
\tau_{\Ret,\Set} = \ov{\tau}_{\Ret,\Set} = q^xt^y.
\]
Moreover, $x$ and $y$ are not both zero unless $\Set = \Ret$ or $\Set = \Ret\cup \{0\}$.
\end{lem}

Before proving this result, we derive some consequences.

\begin{cor} \label{cor: limits}
\begin{enumerate}
\item The $q,t \rightarrow 0$ limit of $\mc{P}_{\la,\rho}[q,t]$ is the dual row insertion bijection $F_{\la,\rho}^{\row}$.
\item The $q,t \rightarrow \infty$ limit of $\mc{P}_{\la,\rho}[q,t]$ is the dual column insertion bijection $F_{\la,\rho}^{\col}$.
\item 
\label{item: between 0,1} The rational functions $p_{\Ret,\Set}[q,t]$ and $\ov{p}_{\Ret,\Set}[q,t]$ take values in $[0,1]$ when $q,t \in [0,1)$ or $q,t \in (1,\infty)$. 
\end{enumerate}
\end{cor}

\begin{proof}
Part (1) follows immediately from Lemma \ref{lem:tau} and \eqref{eq:zeta at 0}. Part (2) follows from part (1), Lemma \ref{lem:inverting q and t} and \eqref{eq:transpose growth rule}.

For part (3), observe that by Theorem \ref{thm:sum to 1}, it suffices to show that $p_{\Ret,\Set}[q,t]$ and $\ov{p}_{\Ret,\Set}[q,t]$ are nonnegative when $q,t \in [0,1)$ or $q,t \in (1,\infty)$. By Lemma \ref{lem:tau}, $\tau_{\Ret,\Set}[q,t]$ is nonnegative whenever $q,t \geq 0$. When $q,t \in [0,1)$, each binomial factor in $\zeta_{\Ret,\Set}[q,t]$ and $\ov{\zeta}_{\Ret,\Set}[q,t]$ is strictly positive, so the result is clear. When $q,t \in (1,\infty)$, each factor in $\zeta_{\Ret,\Set}[q,t]$ and $\ov{\zeta}_{\Ret,\Set}[q,t]$ is strictly negative, and the numerators and denominators of these functions contain the same number of factors.
\end{proof}

In order to prove Lemma \ref{lem:tau}, we give a combinatorial interpretation to the exponents of each of the points $S_s, R_r, O_j, I_i$ appearing in $\tau_{\Ret,\Set}$. We regard two complementary subsequences of $D_0,U_1,D_1,\ldots,D_d,U_d$:
\begin{itemize}
\item The first subsequence includes all $U_r$ with $r \in \Ret$ and $D_s$ with $s \in \Set$.
\item The second subsequence includes all $U_i$ with $i \notin \Ret$ and $D_j$ with $j \notin \Set$.
\end{itemize}
We interpret both subsequences as lattice paths consisting of \deff{up steps} $(i,j) \rightarrow (i+1,j+1)$, and \deff{down steps} $(i,j) \rightarrow (i+1,j-1)$ by replacing each $U$ letter by an up step and each $D$ letter by a down step where we fix the endpoint of both paths to lie on the $x$-axis. We label the steps by $U_i$ or $D_j$ respectively and denote by $u_i$ the height of the starting point of $U_i$ and by $d_j$ the height of the ending point of $D_j$. See Figure~\ref{fig:exponents} for an example.
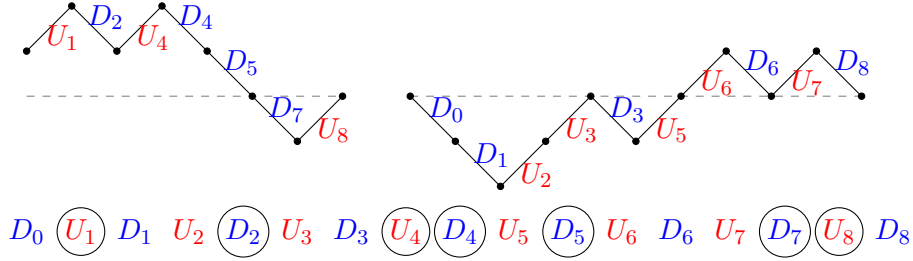
\begin{figure}
\begin{center}
\begin{tikzpicture}[scale=0.6]
\begin{scope}[scale=1.2]
\draw (0,0) node{$\textcolor{blue}{D_0}$};
\draw (1,0) node{$\mathcircled{\textcolor{red}{U_1}}$};
\draw (2,0) node{$\textcolor{blue}{D_1}$};
\draw (3,0) node{$\textcolor{red}{U_2}$};
\draw (4,0) node{$\mathcircled{\textcolor{blue}{D_2}}$};
\draw (5,0) node{$\textcolor{red}{U_3}$};
\draw (6,0) node{$\textcolor{blue}{D_3}$};
\draw (7,0) node{$\mathcircled{\textcolor{red}{U_4}}$};
\draw (8,0) node{$\mathcircled{\textcolor{blue}{D_4}}$};
\draw (9,0) node{$\textcolor{red}{U_5}$};
\draw (10,0) node{$\mathcircled{\textcolor{blue}{D_5}}$};
\draw (11,0) node{$\textcolor{red}{U_6}$};
\draw (12,0) node{$\textcolor{blue}{D_6}$};
\draw (13,0) node{$\textcolor{red}{U_7}$};
\draw (14,0) node{$\mathcircled{\textcolor{blue}{D_7}} $};
\draw (15,0) node{$\mathcircled{\textcolor{red}{U_8}}$};
\draw (16,0) node{$\textcolor{blue}{D_8}$};

\end{scope}

\begin{scope}[yshift=3cm, xshift=-1cm]
\draw[dashed,thin,gray] (1,0) -- (8,0);
\draw (1,1) -- (2,2) -- (3,1) -- (4,2) -- (5,1) -- (6,0) -- (7,-1) -- (8,0);
\filldraw (1,1) circle (2pt);
\filldraw (2,2) circle (2pt);
\filldraw (3,1) circle (2pt);
\filldraw (4,2) circle (2pt);
\filldraw (5,1) circle (2pt);
\filldraw (6,0) circle (2pt);
\filldraw (7,-1) circle (2pt);
\filldraw (8,0) circle (2pt);
\node[red] at (1.8,1.3) {$U_1$};
\node[red] at (3.8,1.3) {$U_4$};
\node[red] at (7.8,-.75) {$U_8$};
\node[blue] at (2.75,1.75) {$D_2$};
\node[blue] at (4.75,1.75) {$D_4$};
\node[blue] at (5.75,.75) {$D_5$};
\node[blue] at (6.75,-.3) {$D_7$};

\begin{scope}[xshift=9.5cm]
\draw[dashed,thin,gray] (0,0) -- (10,0);
\draw (0,0) -- (1,-1) -- (2,-2) -- (3,-1) -- (4,0) -- (5,-1) -- (6,0) --(7,1) -- (8,0) -- (9,1) -- (10,0);

\filldraw (0,0) circle (2pt);
\filldraw (1,-1) circle (2pt);
\filldraw (2,-2) circle (2pt);
\filldraw (3,-1) circle (2pt);
\filldraw (4,0) circle (2pt);
\filldraw (5,-1) circle (2pt);
\filldraw (6,0) circle (2pt);
\filldraw (7,1) circle (2pt);
\filldraw (8,0) circle (2pt);
\filldraw (9,1) circle (2pt);
\filldraw (10,0) circle (2pt);
\node[blue] at (0.8,-.3) {$D_0$};
\node[blue] at (1.8,-1.3) {$D_1$};
\node[blue] at (4.8,-.3) {$D_3$};
\node[blue] at (7.8,.75) {$D_6$};
\node[blue] at (9.8,.75) {$D_8$};
\node[red] at (2.8,-1.75) {$U_2$};
\node[red] at (3.8,-.75) {$U_3$};
\node[red] at (5.8,-.75) {$U_5$};
\node[red] at (6.8,.3) {$U_6$};
\node[red] at (8.8,.3) {$U_7$};
\end{scope}
\end{scope}
\end{tikzpicture}
\end{center}
\caption{The pair of paths used to compute the exponents in $\tau_{\Ret,\Set}$ for $\Ret= \{1,4,8\}, \Set = \{2,4,5,7\}$ and $d=8$. The height sequences are $(u_i)=(\underline{1},-2,-1,\underline{1},-1,0,0,\underline{-1})$ and $(d_i)_i=(-1,-2,\underline{1},-1,\underline{1},\underline{0},0,\underline{-1},0)$ where the heights of steps of the first path are underlined.}
\label{fig:exponents}
\end{figure}

\begin{lem}
\label{lem:exponents}
Let $u_i$ and $d_j$ be defined as above, then 
\[
\tau_{\Ret,\Set}=\prod_{s \in \Set}S_s^{d_s}\prod_{r \in \Ret}R_r^{-u_r}\prod_{j \notin \Set}(-O_j)^{d_j}\prod_{i \notin \Ret}(-I_i)^{-u_i}.
\]
\end{lem}
\begin{proof}
For $s \in \Set$, observe that the exponent of $S_s$ in $\tau_{\Ret,\Set}$ is
\begin{multline*}
\big|[s+1,d] \setminus \Ret\big| - \big|[s+1,d] \setminus \Set\big| \\
= 
\Big(\big|[s+1,d]\big|- \big|[s+1,d] \cap \Ret\big| \Big) - \Big(\big|[s+1,d]\big|-\big|[s+1,d] \cap \Set\big| \Big)\\
= \big|[s+1,d] \cap \Set\big| - \big|[s+1,d] \cap \Ret\big|,
\end{multline*}
which is the difference between the number of down steps and the number of up steps occurring to the right of the down step labeled $D_s$. Since the path ends at height $0$, this difference is equal to $d_s$. The other three cases are similar.
\end{proof}

The interpretation of $\tau_{\Ret,\Set}$ in terms of paths is useful because it allows us to pair $R_r$ and $S_s$ or $I_i$ and $O_j$ appearing with opposite exponents. Specifically, we pair each up step above the $x$-axis with the nearest down step to its right which lies at the same height, and we pair each up step below the $x$-axis to the nearest down step to its left at the same height. Since $|\Set|=|\Ret|$ or $|\Set|=|\Ret|+1$, it is easy to see that by this all down steps are paired except for one which ends at the $x$-axis. By identifying the up step with label $U_i$ with $R_i^{-u_i}$ if $i$ is in $\Ret$ and $(-I_i)^{-u_i}$ otherwise and the down step with label $D_j$ with $S_j^{d_j}$ of $j\in \Set$ and $(-O_j)^{u_j}$ otherwise, we obtain a pairing of the powers of $R_r$ and $S_s$ or $I_i$ and $O_j$ where the exponents in each pair are opposite. This implies furthermore $\tau_{\Ret,\Set} = \ov{\tau}_{\Ret,\Set}$ since $\dfrac{R_r}{S_s} = \dfrac{\ov{R}_r}{\ov{S}_s}$.

\begin{ex}
\label{ex:exponents}
Let $d = 8$ and $\Ret= \{1,4,8\}, \Set = \{2,4,5,7\}$. Using the paths in Figure \ref{fig:exponents}, we find the following pairing of the powers of $R_r$ and $S_s$ or $I_i$ and $O_j$ in $\tau_{\Ret,\Set}$
\begin{multline*}
\tau_{\Ret,\Set}=\left( \left(\dfrac{S_2}{R_1}\right)^1 \left(\dfrac{S_4}{R_4}\right)^1  S_5^0 \left(\dfrac{S_7}{R_8}\right)^{-1} \right) \\ \times
\left( \left(\dfrac{O_0}{I_3}\right)^{-1} \left(\dfrac{O_1}{I_2}\right)^{-2} \left(\dfrac{O_3}{I_5}\right)^{-1} \left(\dfrac{O_6}{I_6}\right)^0 \left(\dfrac{O_8}{I_7}\right)^0 \right),
\end{multline*}
where the terms in the first bracket corresponds to the first path and the terms in the second bracket to the second path.
\end{ex}

It is clear that each pair of up and down steps corresponds to a monomial $q^at^b$ with $a,b$ nonnegative, where both $a,b$ are $0$ only if the end point of the down step lies on the $x$-axis. By taking the product of all pairs and the left-over down step ending on the $x$-axis, the exponents of $q$ and $t$ in $\tau_{\Ret,\Set}$ are nonnegative.
Assume that the exponent of $q$ and $t$ are both $0$. The steps with labels $U_d$ and $D_d$ must be in the same path since the down step before $U_d$ would otherwise end at height $-1$ which is a contradiction. By applying the same argument recursively for $U_{d-1}, \ldots,U_1$, we see that the sequence of labels of the paths are $D_0 U_{i_1} D_{i_1} \ldots U_{i_k} D_{i_k}$ and $U_{j_1} D_{j_1} \ldots U_{j_{d-k}} D_{j_{d-k}}$ respectively. By the definition of the paths this implies either $\Set=\Ret$ or $\Set=\Ret\cup \{0\}$ which finishes the proof of Lemma~\ref{lem:tau}.

\subsection{Definition and examples of $qtRSK^*$}
\label{sec:dual qtRSK}

 In this section, we introduce the probabilistic insertion algorithm $\dRSK$ by interpreting the $\mc{P}_{\la,\rho}$ as $(q,t)$-local dual growth rules.
We freely use the definitions and notations from \S \ref{sec: growth diagrams}. Let $\Lambda$ be a dual    growth, and suppose $\square \in \Lambda$ is a square in the $m \times n$ grid of the form

\begin{center}
\begin{tikzpicture}[baseline=-3.5ex]
\draw (-0.2,-.1) --node[left]{\rotatebox{-90}{$\prec$}} (-0.2,-.9);
\draw (1.2,-.1) --node[right]{\rotatebox{-90}{$\prec$}} (1.2,-.9);
\draw (0.1,0.2) --node[above]{$\prec'$} (.9,0.2);
\draw (.1,-1.2) --node[below]{$\prec'$} (.9,-1.2);
\node at (-0.2,0.2) {$\mu$};
\node at (1.2,0.2) {$\rho$};
\node at (-0.2,-1.15) {$\la$};
\node at (1.2,-1.2) {$\nu$};
\node at (.5,-.5) {$a$};
\end{tikzpicture}
\end{center}
where $a \in \{0,1\}$, $\nu \in \mc{U}^k(\la,\rho)$ and $\mu \in \mc{D}^{k-a}(\la,\rho)$. Set
\[
\mc{P}(\Lambda) = \prod_{\square \in \Lambda} \mc{P}(\square), \qquad\qquad \ov{\mc{P}}(\Lambda) = \prod_{\square \in \Lambda} \ov{\mc{P}}(\square),
\]
where 
\[
\mc{P}(\square) = \mc{P}_{\la,\rho}(\mu \rightarrow \nu), \qquad\qquad
\ov{\mc{P}}(\square)= \ov{\mc{P}}_{\la,\rho}(\mu \leftarrow \nu).
\]

\begin{defn}
Suppose $A$ is an $m \times n$ $\{0,1\}$-matrix and $P \in \SSYT(\la)$, $Q \in \SSYT^*(\la)$ for some $\la \subseteq (m^n)$. Define
\[
\mc{P}(A \rightarrow P,Q) = \sum_{\Lambda : A \rightarrow (P,Q)} \mc{P}(\Lambda), \qquad\qquad
\ov{\mc{P}}(A \leftarrow P,Q) = \sum_{\Lambda : A \rightarrow (P,Q)} \ov{\mc{P}}(\Lambda).
\]
\end{defn}

We accomplish our goal of proving the dual Cauchy identity for Macdonald polynomials with the following theorem.

\begin{thm}
\label{thm:dRSK is prob bij}
The expressions $\mc{P}(A \rightarrow P,Q)$ and $\ov{\mc{P}}(A \leftarrow P,Q)$ define a probabilistic bijection between the weighted sets of $m \times n$ $\{0,1\}$-matrices with weight $\omega$ and $\ds \bigsqcup_{\la \subseteq (m^n)} \SSYT(\la) \times \SSYT^*(\la)$ with weight $\ov{\omega}$, where
\[
\omega(A) = \prod_{\substack{1 \leq i \leq m \\1 \leq j \leq n}}(x_iy_j)^{A_{i,j}}, \quad\quad \text{ and } \quad\quad \ov{\omega}(P,Q) = \psi_P(q,t)\vp^*_Q(q,t)\mb{x}^P\mb{y}^Q.
\]
Moreover, $\mc{P}(A \rightarrow P,Q)$ and $\ov{\mc{P}}(A \leftarrow P,Q)$ take values in $[0,1]$ when $q,t \in [0,1)$ or $q,t \in (1,\infty)$.
\end{thm}

\begin{proof}
For a fixed $\{0,1\}$-matrix  $A$, one can obtain all the dual growths associated with $A$ by starting with the empty partition along the north and west boundaries, and recursively filling in the rest of the diagram.
When the top left, bottom left, and top right vertices ($\mu,\la,\rho$, respectively) of a square have been filled in, $\nu$ is chosen according to the probability distribution $\mc{P}_{\la,\rho}(\mu \rightarrow \nu)$.
Hence the fact that the local probabilities $\mc{P}_{\la,\rho}(\mu \rightarrow \nu)$ sum to $1$ for fixed $\mu, \la,\rho$ implies that the probabilities $\mc{P}(A \rightarrow P,Q)$ sum to $1$ for fixed $A$. Since all local probabilities take values in $[0,1]$ for $q,t \in [0,1)$ or $q,t \in (1, \infty)$ by Corollary~\ref{cor: limits}(\ref{item: between 0,1}), the same is true for $\mc{P}(A \rightarrow P,Q)$.

By starting with a fixed pair of tableaux $(P,Q)$ along the right and bottom boundaries and recursively filling in the rest of the dual growth according to the local backward probabilities $\ov{\mc{P}}_{\la,\rho}(\mu \leftarrow \nu)$, one obtains that the backward probabilities $\ov{\mc{P}}(A \leftarrow P,Q)$ sum to 1 for fixed $(P,Q)$, and take values in $[0,1]$ for $q,t \in [0,1)$ or $q,t \in (1, \infty)$ as before.
Finally, we prove the compatibility relation 
 in Lemma \ref{lem_Lambda_compatibility} below.
\end{proof}

\begin{lem}
\label{lem_Lambda_compatibility}
If $\Lambda$ is a dual growth from $A$ to $(P,Q)$, then
\[
\omega(A)\mc{P}(\Lambda) = \ov{\mc{P}}(\Lambda) \ov{\omega}(P,Q).
\]
\end{lem}

\begin{proof}
Let $L$ be a lattice path in $\Lambda$ from the northeast corner $(0,n)$ to the southwest corner $(n,0)$ consisting of unit steps to the south or west. Let $\omega(L)$ be the product of the weights of the edges in $L$, where the $i$-th vertical edge $\mu \atop {| \atop \la}$ has weight $\psi_{\la/\mu}(q,t) x_i^{|\la/\mu|}$, and the $(n+1-j)$-th horizontal edge $\mu - \rho$ has weight $\vp^*_{\rho/\mu}(q,t)y_j^{|\rho/\mu|}$. Observe that the lattice path consisting of $n$ south steps followed by $n$ west steps has weight $\psi_P(q,t) \vp^*_Q(q,t) \mb{x}^P\mb{y}^Q=\ov{\omega}(P,Q)$.

Let $\nw(L)$ be the set of squares of $\Lambda$ between $L$ and the path consisting of $n$ west steps followed by $n$ south steps, and write $(i,j) \in \nw(L)$ if the square in the $i$-th row from top and $j$-th column from left is in $\nw(L)$. Define a partial order on the set of lattice paths by $L \leq L'$ if $\nw(L) \subseteq \nw(L')$. We prove by induction with respect to this partial order that
\begin{equation}
\label{eq_L_induction}
\prod_{(i,j) \in \nw(L)} (x_iy_j)^{A_{i,j}} \prod_{\square \in \nw(L)} \mc{P}(\square) = \omega(L) \prod_{\square \in \nw(L)} \ov{\mc{P}}(\square).
\end{equation}

The base case is the path consisting of $n$ west steps followed by $n$ south steps, for which both sides of \eqref{eq_L_induction} are equal to 1. For the induction step, it suffices to consider the case where $\nw(L')$ is obtained from $\nw(L)$ by adding a single square, as depicted next.

\begin{center}
\begin{tikzpicture}[scale=0.9]

\draw (0,0) -- (0,-1) -- (1,-1) -- (1,0) -- (0,0);
\draw[color=red] (0.1,0.2) -- (.9,0.2);
\draw[color=red] (-0.2,-.1) -- (-.2,-.9);
\node[color=red] at (.5,.6) {$\vp^*_{\rho/\mu}$}; 
\node[color=red] at (-.8,-.5) {$\psi_{\la/\mu}$};  
\node at (-0.2,0.2) {$\mu$};
\node at (1.2,0.2) {$\rho$};
\node at (-0.2,-1.2) {$\la$};
\node at (1.25,-1.25) {$\nu$};
\node at (.5,-.5) {$A_{i,j}$};

\begin{scope}[xshift=4cm]
\draw (0,0) -- (0,-1) -- (1,-1) -- (1,0) -- (0,0);
\draw[color=blue] (.1,-1.2) -- (.8,-1.2);
\draw[color=blue] (1.2,-.1) -- (1.2,-.9);
\node[color=blue] at (.5,-1.6) {$\vp^*_{\nu/\la}$}; 
\node[color=blue] at (1.8,-.5) {$\psi_{\nu/\rho}$};
\node at (-0.2,0.2) {$\mu$};
\node at (1.2,0.2) {$\rho$};
\node at (-0.2,-1.2) {$\la$};
\node at (1.25,-1.25) {$\nu$};
\node at (.5,-.5) {$A_{i,j}$};
\end{scope}
\draw[->] (1.7,-0.5) -- (3.3,-0.5);
\end{tikzpicture}
\end{center}
The left hand side of \eqref{eq_L_induction} changes by the factor $(x_iy_j)^{A_{i,j}} \mc{P}_{\la,\rho}(\mu \rightarrow \nu)$ while the right hand side changes by the factor 
\[
\ov{\mc{P}}_{\la,\rho}(\mu \leftarrow \nu) \frac{\psi_{\nu/\rho}\vp^*_{\nu/\la}}{\psi_{\la/\mu}\vp^*_{\rho/\mu}}x_i^{|\nu/\rho|-|\la/\mu|}y_j^{|\nu/\la|-|\rho/\mu|}.
\]
The fact that $|\nu|-|\la\cup \rho| = |\la \cap \rho|-|\mu|+A_{i,j}$ implies that the change of the exponents of $x_i$ and $y_j$ is equal on both sides. The equality of the remaining expressions is the statement of Theorem~\ref{thm:weights}.
\end{proof}

The statement of the above lemma can be reformulated as $\mc{P}(\Lambda) = \linebreak \ov{\mc{P}}(\Lambda) \psi_P(q,t) \vp^*_Q(q,t)$. It is not difficult to see, that this and Theorem~\ref{thm:dRSK is prob bij} implies the more refined identities
\begin{align*}
& \sum_{A} \mc{P}(A \rightarrow P,Q) = \psi_P(q,t) \vp^*_Q(q,t), \\
&\sum_{(P,Q)} \ov{\mc{P}}(A \leftarrow P,Q) \psi_P(q,t) \vp^*_Q(q,t) = 1.
\end{align*}

\begin{rem}
\label{rem:skew dual cauchy}
Let $\la,\rho$ be two partitions.  It is easy to see that $\dRSK$ yields a proof of the skew dual Cauchy identity
\begin{equation}
\label{eq:skew dual cauchy}
\sum_{\mu} P_{\la/\mu}(\x;q,t)P_{\rho^\prime/\mu^\prime}(\y;t,q) \prod_{\substack{1 \leq i \leq m \\1 \leq j \leq n}}= 
\sum_{\nu} P_{\nu/\rho}(\x;q,t)P_{\nu^\prime/\la^\prime}(\y;t,q).
\end{equation}
Indeed, the left hand side of \eqref{eq:skew dual cauchy} is the weighted generating function of triples $(A,P,Q)$ where $A$ is an $m \times n$ $\{0,1\}$-matrix, $P$ a skew SSYT of shape $\la/\mu$ and $Q$ a skew dual SSYT of shape $\rho/\mu$ with respect to the weight 
\[
\omega(A,P,Q)=  \prod_{\substack{1 \leq i \leq m \\1 \leq j \leq n}}(x_iy_j)^{A_{i,j}}
\psi_P(q,t) \varphi^*_Q(q,t) \x^P \y^Q.
\]
The right hand side of \eqref{eq:skew dual cauchy} is the weighted generating function of pairs $(\wh{P},\wh{Q})$ where $\wh{P}$ is a skew SSYT of shape $\nu/\rho$ and $\wh{Q}$ is a skew dual SSYT of shape $\nu/\la$ with weight
\[
\ov{\omega}(\wh{P},\wh{Q}) = \psi_{\wh{P}}(q,t)\varphi^*_{\wh{Q}}(q,t)\x^{\wh{P}}\y^{\wh{Q}}.
\]
A straight forward adaption of the proof of Theorem~\ref{thm:dRSK is prob bij} and of Lemma~\ref{lem_Lambda_compatibility} shows that the expressions
\begin{align*}
\mc{P}(A,P,Q \rightarrow \wh{P},\wh{Q})& := \mc{P}(A \rightarrow \wh{P},\wh{Q}), \\ 
\ov{\mc{P}}(A,P,Q \leftarrow \wh{P},\wh{Q})& := \ov{\mc{P}}(A \leftarrow \wh{P},\wh{Q}),
\end{align*}
yield a probabilistic bijection between the above described weighted sets.
\end{rem}

The probabilities $\mc{P}(A \rightarrow P,Q)$ and $\ov{\mc{P}}(A \leftarrow P,Q)$ satisfy the following symmetry property which generalizes \eqref{eq:sym of dual RSK}.

\begin{thm}
\label{thm:symmetry of qt prob}
Let $A$ be an $m \times n$ $\{0,1\}$-matrix, $P$ an SSYT and $Q$ a dual SSYT of the same shape. Then
\begin{equation}
\label{eq:symmetry of qt prob}
\mc{P}(A^T \rightarrow Q^\prime,T^\prime)[q,t] = \mc{P}(A \rightarrow P,Q)[t^{-1},q^{-1}],
\end{equation}
and similarly for $\ov{\mc{P}}$.
\end{thm}
\begin{proof}
Let $\Lambda$ be a dual growth from $A$ to $(P,Q)$. We define $\Lambda^\prime$ as the dual growth from $A^T$ to $(Q^\prime,P^\prime)$ which is obtained by transposing all partitions in $\Lambda$ and then taking the transpose of the array of partitions itself. It is immediate that the map $\Lambda \mapsto \Lambda^\prime$ is a bijection from the set of dual growths from $A$ to $(P,Q)$ to the set of dual growths from $A^T$ to $(Q^\prime,P^\prime)$. By Lemma~\ref{lem:inverting q and t} we obtain for the probability of a dual growth $\Lambda^\prime$
\[
\mc{P}(\Lambda^\prime)[q,t] = \prod_{\square^\prime \in \Lambda^\prime} \mc{P}(\square^\prime)[q,t] =
\prod_{\square \in \Lambda} \mc{P}(\square)[t^{-1},q^{-1}] = \mc{P}(\Lambda)[t^{-1},q^{-1}]. 
\]
The result follows by summing over all $\Lambda^\prime$.
\end{proof}

The discussion in \S \ref{sec: growth diagrams}, in particular Lemma~\ref{lem: growth insertion}, explains how the above $(q,t)$-local dual growth rules can be translated into a probabilistic insertion algorithm. Let $T$ be a semistandard Young tableau and recall that $T^{(i)}$ denotes the shape of the subtableau consisting of entries at most $i$.

\begin{defn}
\label{def:qtRSK insertion}
Let $T$ be a semistandard Young tableau and $i_1< \cdots < i_r$ be positive integers. The \deff{dual $(q,t)$-RSK insertion} of $i_1, \ldots, i_r$ into $T$, denoted
\[
(i_1,\ldots,i_r) \xrightarrow{\dRSK} T = \wh{T},
\]
is the probability distribution computed as follows:
\begin{itemize}
\item Call the multiset $\{i_1,\ldots,i_r\}$ the \deff{insertion queue}.
\item Let $i$ be the smallest integer of the insertion queue and denote by $k$ the multiplicity of $i$ in the insertion queue. For each $\nu \in \U^k\left(T^{(i)},\wh{T}^{(i-1)}\right)$, place $i$ in each cell of $\nu/\left(T^{(i)}\cup\wh{T}^{(i-1)}\right)$ with probability $\mc{P}_{T^{(i)},\wh{T}^{(i-1)}}\left(T^{(i-1)} \rightarrow \wh{T}^{(i)} \right)$. Delete all $i's$ from the insertion queue and add all entries which have been replaced (bumped) by an $i$ to the insertion queue.
\item Repeat the previous step until the insertion queue is empty.
\end{itemize}
\end{defn}

\begin{ex}
\label{ex:qrst insertion}
The insertion $(2,3) \xrightarrow{\dRSK} \young(12,3)$ produces

\begin{align*}
\young(122,33) \quad \text{ with probability }  & \mc{P}_{(2),(1)}((1)\rightarrow (3))
=\dfrac{1-qt}{1-q^2t}
, \\[10pt]
\young(123,23) \quad  \text{ with probability } & \mc{P}_{(2),(1)}((1)\rightarrow (2,1))\mc{P}_{(2,1),(2,1)}((2)\rightarrow (3,2)) \\
&=q^2 t \dfrac{(1-q)^2(1-t)^2}{(1-qt)(1-q^2)(1-q^2t)(1-q^2t^2)}
,\\
\young(12,23,3) \quad  \text{ with probability } & \mc{P}_{(2),(1)}((1)\rightarrow (2,1))\mc{P}_{(2,1),(2,1)}((2)\rightarrow (2,2,1)) \\
&=qt^2\dfrac{(1-q)^2(1-q^2t)}{(1-qt)(1-q^2)(1-q^2t^2)}
,\\
\young(123,2,3) \quad  \text{ with probability }  & \mc{P}_{(2),(1)}((1)\rightarrow (2,1))\mc{P}_{(2,1),(2,1)}((2)\rightarrow (3,1,1))\\
&=qt\dfrac{(1-q)^2(1-t)}{(1-q^2)(1-qt)^2},
\end{align*}
where some of the probabilities were already calculated in Example~\ref{ex:qt growth rule}.
\end{ex}

\section{Proof of Theorem \ref{thm:sum to 1} and Theorem \ref{thm:weights}}
\label{sec:weights proof}
The aim of this section is to prove that the probabilities $\mc{P}_{\la,\rho}$ and  $\ov{\mc{P}}_{\la,\rho}$ define probabilistic bijections.  In \S \ref{sec:interpolation} we use a generalization of Lagrange interpolation to prove Theorem~\ref{thm:sum to 1}. 
In \S \ref{sec:prob via hook-lengths} we define six rational functions in $q,t$, namely $\alpha,\beta,\gamma$ and their companions $\ov{\alpha}, \ov{\beta}, \ov{\gamma}$ and show in Proposition~\ref{prop:alpha beta gamma} an alternative definition of the probabilities using these functions. While the definition in \S \ref{sec:probabilities definition} allows a simple proof of Theorem~\ref{thm:sum to 1}, the alternative definition given in Proposition~\ref{prop:alpha beta gamma} is analogous to the one in \cite{AignerFrieden22} and allows us to prove Theorem~\ref{thm:weights}.

\subsection{Interpolation identities}
\label{sec:interpolation}

Let $\binom{A}{k}$ denote the set of $k$-element subsets of $A$ and write $[n]=\{1,\ldots,n\}$. For an element $\Set=\{s_1,\ldots,s_k\} \in \binom{A}{k}$ we denote by $\mathbf{x}_S:=(x_{s_1},\ldots,x_{s_k})$.

\begin{lem}[\cite{ChenLouck96}]
\label{lem: sym poly interpolation}
Let $K$ be a field of characteristic $0$, $f \in K[x_1,\ldots,x_n]^{S_n}$ a symmetric function of degree at most $d$ in $x_1$ and let $a_1, \ldots, a_{n+d}$ be pairwise different elements of $K$. Then we have
\begin{equation}
\label{eq: sym poly interpolation}
f(x_1,\ldots,x_n) = \sum_{\Set \in \binom{[n+d]}{n}} f(\mathbf{a}_\Set)
\prod_{i \in [n+d]\setminus \Set} \frac{\prod\limits_{j=1}^n (x_j-a_i)}{\prod\limits_{s \in \Set} (a_s-a_i)}.
\end{equation}
\end{lem}
For reasons of being self-contained, we give another proof of the above interpolation identity.
\begin{proof}
Both the left-hand and the right-hand side of \eqref{eq: sym poly interpolation} are equal for $\mathbf{x}=(a_{s_1},\ldots,a_{s_n})$ and $1 \leq s_1 < \cdots < s_n \leq n+d$. Hence it suffices to show that these positions uniquely determine $f$. We know that $f$ can be written as
\[
f(\mathbf{x}) = \sum_{\lambda \subseteq (d)^n}c_\lambda s_\lambda(\mathbf{x}).
\]
The function $f$ is therefore uniquely determined at the positions $\mb{a}_\Set$ for $\Set \in \binom{[n+d]}{n}$ if and only if the determinant $\det\left( s_\lambda(\mb{a}_\Set)\right)$, where $\Set \in \binom{[n+d]}{n}$ and $\lambda \subseteq (d)^n$ is nonzero, i.e., the matrix $(s_\lambda(\mb{a}_\Set))_{\lambda,\Set}$ is the transition matrix between the Schur polynomials and the polynomials appearing on the right hand side of \eqref{eq: sym poly interpolation}. We claim that the determinant is equal to
\begin{equation}
\label{eq: Vandermonde type determinant}
\det\limits_{\Set \in \binom{[n+d]}{n},\; \lambda \subseteq (d)^n}\Big(s_\lambda(\mb{a}_\Set)\Big)
= \pm \prod_{1 \leq i < j \leq n+d}(a_i-a_j)^{\binom{n+d-2}{n-1}}.
\end{equation}
Note that we added $\pm$ on the right-hand side of \eqref{eq: Vandermonde type determinant} since the sign of the determinant is depending on the actual order of the rows and columns. We prove \eqref{eq: Vandermonde type determinant} by using the \emph{indication of factors} method. For $a_i = a_j$ and $i \neq j$ we have $\binom{n+d-2}{n-1}$ pairs of rows which are identical, namely the rows of the form $\{i\} \cup \Set^\prime$ and $\{j \} \cup \Set^\prime$ where $\Set^\prime \subset [n+d] \setminus\{i,j\}$ is a subset of size $n-1$. Hence the right hand of \eqref{eq: Vandermonde type determinant} side is dividing the determinant. The total-degree of the determinant is equal to the sum of $|\la|$, where we sum over all partitions $\lambda$ included in the box $(d)^n$. For $\lambda=(i_1,\ldots,i_n)$ where we allow parts to be $0$, the total-degree of the determinant is 
\[
\sum_{\lambda \subseteq (d)^n}|\lambda| = \sum_{d \geq i_1 \geq \cdots \geq i_n \geq 0} (i_1+\cdots +i_n)
=\frac{n(n+1)}{2}\binom{n+d}{n+1}.
\]
The last equality can be proven by induction on $n$. This implies that the total-degree of both sides of \eqref{eq: Vandermonde type determinant} coincide. Hence it suffices to prove that the leading coefficient with respect to the lexicographical order is the same on both sides. For the left-hand side we obtain the leading term by choosing for a fixed $\lambda=(d-i_1,\ldots,d-i_n)$ the set $\Set=\{1+i_1,2+i_2,\ldots,n+i_n\}$ and then the monomial corresponding to the semistandard Young tableau with the $j$-th row filled by $j+i_j$. It is easy to see that this yields the leading coefficient and is, up to sign, equal to $1$. Since the leading coefficient of the right-hand side is $1$, this proves the assertion.
\end{proof}

For $\Ret,\Set \in \binom{[0,d]}{k}$ and two sequences $\textbf{a}=(a_0, \ldots, a_d), \textbf{b}=(b_0, \ldots, b_d)$ of pairwise different variables, define
\[
p'_{\Ret,\Set}(\textbf{a},\textbf{b}) = \prod_{i \in [0,d] \setminus \Ret} \dfrac{\prod\limits_{j \in \Set} (a_j-b_i)}{\prod\limits_{j \in \Ret} (b_j-b_i)} \prod_{i \in [0,d] \setminus \Set} \dfrac{\prod\limits_{j \in \Ret} (b_j-a_i)}{\prod\limits_{j \in \Set} (a_j-a_i)}.
\]
Note that $p'_{\Ret,\Set}(\textbf{a},\textbf{b}) = p'_{\Set,\Ret}(\textbf{b},\textbf{a})$. 

\begin{lem}
\label{lem_master_identity}
For fixed $\Ret \in \binom{[0,d]}{k}$, we have
\begin{equation}
\label{eq:sum of p prob 1}
\sum_{\Set \in \binom{[0,d]}{k}} p'_{\Ret,\Set}(\textbf{a},\textbf{b}) = 1.
\end{equation}
\end{lem}
By the symmetry of $p'_{\Ret,\Set} $ in $\Ret$ and $\Set$, we also have 
\begin{equation}
\label{eq:sum of p prob 2}
\sum_{\Ret \in \binom{[0,d]}{k}} p'_{\Ret,\Set}(\textbf{b},\textbf{a}) = 1,
\end{equation}
 for fixed $\Set$.

\begin{proof}
Define the symmetric function
\[
f(x_1, \ldots, x_k) = \prod_{j = 1}^k \; \prod_{i \in [0,d] \setminus \Ret} (x_j-b_i).
\] 
Since $f$ has degree $d+1-k$ in $x_1$, Lemma~\ref{lem: sym poly interpolation} implies
\[
f(x_1, \ldots, x_k) = \sum_{\Set \in \binom{[0,d]}{k}} f(\mb{a}_\Set) \prod_{i \in [0,d] \setminus \Set} \dfrac{\prod\limits_{j=1}^k (x_j - a_i)}{\prod\limits_{s \in \Set} (a_s - a_i)}.
\]
By evaluating $f$ at $(b_{r_1},\ldots, b_{r_k})$, where $\Ret=\{r_1,\ldots,r_k\}$, we obtain
\[
\prod_{r \in \Ret} \; \prod_{i \in [0,d] \setminus \Ret} (b_r - b_i) = \sum_{\Set \in \binom{[0,d]}{k}} \prod_{s \in \Set} \; \prod_{i \in [0,d] \setminus \Ret} (a_s - b_i) \prod_{i \in [0,d] \setminus \Set} \dfrac{\prod\limits_{r \in \Ret} (b_r - a_i)}{\prod\limits_{s \in \Set} (a_s - a_i)}.
\]
Dividing by the left-hand side, we obtain the assertion.
\end{proof}

The last ingredient for the proof of Theorem~\ref{thm:sum to 1} is the following limit which can be verified easily
\begin{equation}
\label{eq: limit of pre-prob}
\lim_{b_0 \rightarrow \infty} p'_{\Ret,\Set}(\textbf{a},\textbf{b}) = \prod_{i \in [d] \setminus \Ret} \dfrac{\prod\limits_{j \in \Set} (a_j-b_i)}{\prod\limits_{j \in \Ret\cap [d]} (b_j-b_i)} \prod_{i \in [0,d] \setminus \Set} \dfrac{\prod\limits_{j \in \Ret\cap [d]} (b_j-a_i)}{\prod\limits_{j \in \Set} (a_j-a_i)}.
\end{equation}

\begin{proof}[Proof of Theorem~\ref{thm:sum to 1}]
Lemma \ref{lem_master_identity} and \eqref{eq: limit of pre-prob} imply equation \eqref{eq: sum of forward prob} by setting $a_i=O_i$ for $0 \leq i \leq d$ and $b_i =R_i$ for $i \in \Ret$ and $b_i=I_i$ for $i \in [d]\setminus \Ret$.
Analogously equation \eqref{eq: sum of backward prob} is obtained from  \eqref{eq:sum of p prob 2} together with \eqref{eq: limit of pre-prob} by setting $a_i=\ov{S}_i$ for $i \in \Set$ and $a_i=O_i$ for $i \in [0,d]\setminus \Set$, and $b_i=I_i$ for all $1 \leq i \leq d$.
\end{proof}

\subsection{Probabilities via $(q,t)$-hook-lengths}
\label{sec:prob via hook-lengths}

By definition, we have
\begin{multline}
\label{eq:weights def}
\dfrac{\omega_{\la,\rho}(\mu)}{\ov{\omega}_{\la,\rho}(\nu)} = \dfrac{\psi_{\lambda/ \mu} \vp^*_{\rho/ \mu}}{\psi_{\nu/ \rho} \vp^*_{\nu/\lambda}} \\
= \prod_{c \in \mc{R}_{\la/\mu} - \mc{C}_{\la/\mu}} \dfrac{b_\mu(c)}{b_\la(c)} \prod_{c \in \mc{C}_{\rho/\mu} - \mc{R}_{\rho/\mu}} \dfrac{b_\rho(c)}{b_\mu(c)} \prod_{c \in \mc{R}_{\nu/\rho} - \mc{C}_{\nu/\rho}} \dfrac{b_\nu(c)}{b_\rho(c)} \prod_{c \in \mc{C}_{\nu/\la} - \mc{R}_{\nu/\la}} \dfrac{b_\la(c)}{b_\nu(c)}.
\end{multline}
In \S \ref{sec:weights formula}, we analyze the contribution of each cell $c \in \nu$ to the right-hand side of \eqref{eq:weights def} to obtain the alternative formula
\begin{multline}
\label{eq:weights formula}
\dfrac{\omega_{\la,\rho}(\mu)}{\ov{\omega}_{\la,\rho}(\nu)} 
= \prod_{c \in \mc{R}_{\la \cap \rho/\mu} - \mc{C}_{\la/\la \cap \rho}} \dfrac{b_\mu(c)}{b_\rho(c)} 
\prod_{c \in \mc{C}_{\la \cap \rho/\mu} - \mc{R}_{\rho/\la \cap \rho}} \\ \times \dfrac{b_\la(c)}{b_\mu(c)} \prod_{c \in \mc{R}_{\nu/\la \cup \rho} - \mc{C}_{\la \cup \rho/\rho}} \dfrac{b_\nu(c)}{b_\la(c)} \prod_{c \in \mc{C}_{\nu/\la \cup \rho} - \mc{R}_{\la \cup \rho/\la}} \dfrac{b_{\rho}(c)}{b_{\nu}(c)}.
\end{multline}
In Proposition \ref{prop:alpha beta gamma} we provide an alternative formula for the fraction of the probabilities
\begin{equation}
\label{eq:frac of P as alpha beta}
\dfrac{\ov{\mc{P}}_{\la,\rho}(\mu \leftarrow \nu)}{\mc{P}_{\la,\rho}(\mu \rightarrow \nu)}
=
\dfrac{\ov{\alpha}_{\la,\rho,\nu} \ov{\beta}_{\mu,\la,\rho}}{\alpha_{\la,\rho,\nu} \beta_{\mu,\la,\rho}},
\end{equation}
where the terms on the right-hand side are defined below. Furthermore Proposition~\ref{prop:alpha beta gamma} implies that the right-hand side of \eqref{eq:weights formula} is equal to the right-hand side of \eqref{eq:frac of P as alpha beta}.
Thus, Theorem \ref{thm:weights} follows from Proposition \ref{prop:alpha beta gamma} and \eqref{eq:weights formula}.\\

Recall that by Lemma \ref{lem:tau}, we have $p_{\Ret,\Set} = \tau_{\Ret,\Set} \zeta_{\Ret,\Set}$ and $\ov{p}_{\Ret,\Set} = \tau_{\Ret,\Set} \ov{\zeta}_{\Ret,\Set}$, where $\tau_{\Ret,\Set}$ is a monomial in $q$ and $t$ with nonnegative exponents, and $\zeta_{\Ret,\Set}, \ov{\zeta}_{\Ret,\Set}$ are ratios of products of binomials of the form $1-q^xt^y$, with $x,y \geq 0$. These binomial factors are in fact $(q,t)$-hook-lengths of cells in $\la,\rho,\mu^{(\Ret)},$ and $\nu^{(\Set)}$.

\begin{prop}
\label{prop:alpha beta gamma}
Let $\mu = \mu^{(\Ret)}$ and $\nu = \nu^{(\Set)}$. The probabilities $p_{\Ret,\Set} = \mc{P}_{\la,\rho}(\mu \rightarrow \nu)$ and $\ov{p}_{\Ret,\Set} = \ov{\mc{P}}_{\la,\rho}(\mu \leftarrow \nu)$ are given by the formulas
\[
\mc{P}_{\la,\rho}(\mu \rightarrow \nu) = \tau_{\mu,\la,\rho,\nu} \dfrac{\alpha_{\la,\rho,\nu} \beta_{\mu,\la,\rho}}{\gamma_{\mu,\la,\rho,\nu}}, \qquad\quad
\ov{\mc{P}}_{\la,\rho}(\mu \leftarrow \nu) = \tau_{\mu,\la,\rho,\nu} \dfrac{\ov{\alpha}_{\la,\rho,\nu} \ov{\beta}_{\mu,\la,\rho}}{\gamma_{\mu,\la,\rho,\nu}},
\]
where $\tau_{\mu,\la,\rho,\nu} = \tau_{\Ret,\Set}$,
\[
\alpha_{\la,\rho,\nu} = \prod_{c \in \mc{R}_{\nu/\la \cup \rho} - \mc{C}_{\la \cup \rho/\rho}} \dfrac{\hl{\la}(c)}{\hl{\nu}(c)} \prod_{c \in \mc{C}_{\nu/\la \cup \rho} - \mc{R}_{\la \cup \rho/\la}} \dfrac{\ha{\rho}(c)}{\ha{\nu}(c)},
\]
\[
\beta_{\mu,\la,\rho} = \prod_{c \in \mc{R}_{\la \cap \rho/\mu} - \mc{C}_{\la/\la \cap \rho}} \dfrac{\hl{\rho}(c)}{\hl{\mu}(c)} \prod_{c \in \mc{C}_{\la \cap \rho/\mu} - \mc{R}_{\rho/\la \cap \rho}} \dfrac{\ha{\la}(c)}{\ha{\mu}(c)},
\]
\[
\gamma_{\mu,\la,\rho,\nu} = \dfrac{\ds \prod_{c \in \mc{R}_{\nu/\la \cup \rho} \cap \mc{C}_{\la \cap \rho/\mu}} \hl{\la}(c)\ha{\la}(c) \prod_{c \in \mc{C}_{\nu/\la \cup \rho} \cap \mc{R}_{\la \cap \rho/\mu}} \hl{\rho}(c)\ha{\rho}(c)}{\ds \prod_{c \in \mc{R}_{\nu/\la \cup \rho} \cap \mc{C}_{\nu/\la \cup \rho}} \hl{\nu}(c)\ha{\nu}(c) \prod_{c \in \mc{R}_{\la \cap \rho/\mu} \cap \mc{C}_{\la \cap \rho/\mu}} \hl{\mu}(c)\ha{\mu}(c)},
\]
and $\ov{\alpha}_{\la,\rho,\nu}, \ov{\beta}_{\mu,\la,\rho}$ are defined in the same way as $\alpha_{\la,\rho,\nu}, \beta_{\mu,\la,\rho}$, but with all the $\hl{\kappa}$'s and $\ha{\kappa}$'s interchanged.
\end{prop}

Recall the rational functions $\zeta_{\Ret,\Set}$ and $\ov{\zeta}_{\Ret,\Set}$ defined in \eqref{eq:zeta}.
Lemma~\ref{lem:alpha beta gamma pieces} implies
\[
\zeta_{\Ret,\Set} = \dfrac{\alpha_{\la,\rho,\nu}\beta_{\mu,\la,\rho}}{\gamma_{\mu,\la,\rho,\nu}},
\]
and Lemma~\ref{lem:alpha beta gamma pieces2} implies
\[
\ov{\zeta}_{\Ret,\Set} = \dfrac{\ov{\alpha}_{\la,\rho,\nu}\ov{\beta}_{\mu,\la,\rho}}{\gamma_{\mu,\la,\rho,\nu}}.
\]
Hence Proposition~\ref{prop:alpha beta gamma} is a direct consequence of these lemmas which we prove next.

\begin{lem}
\label{lem:alpha beta gamma pieces}
We have
\begin{equation}
\label{eq:prod 1}
\prod_{c \in \mc{R}_{\nu/\la \cup \rho} - \mc{C}_{\la \cup \rho/\rho}} \dfrac{\hl{\la}(c)}{\hl{\nu}(c)} = \prod_{i \in \Set} \dfrac{\ds \prod_{i < j \leq d} \left(1 - \dfrac{I_j}{S_i}\right)}{\ds \prod_{i < j \leq d \atop j \not \in \Set} \left(1 - \dfrac{O_j}{S_i}\right) \prod_{i \leq j \leq d \atop j \in \Set} \left(1 - t\dfrac{O_j}{S_i}\right)},
\end{equation}
\begin{equation}
\prod_{c \in \mc{C}_{\nu/\la \cup \rho} - \mc{R}_{\la \cup \rho/\la}} \dfrac{\ha{\rho}(c)}{\ha{\nu}(c)} = \prod_{j \in \Set} \dfrac{\ds \prod_{1 \leq i \leq j} \left(1 - \dfrac{S_j}{I_i}\right)}{\ds \prod_{0 \leq i < j \atop i \not \in \Set} \left(1 - \dfrac{S_j}{O_i}\right) \prod_{0 \leq i \leq j \atop i \in \Set} \left(1 - q\dfrac{S_j}{O_i}\right)},
\end{equation}
\begin{equation}
\prod_{c \in \mc{R}_{\la \cap \rho/\mu} - \mc{C}_{\la/\la \cap \rho}} \dfrac{\hl{\rho}(c)}{\hl{\mu}(c)} = \prod_{i \in \Ret} \dfrac{\ds \prod_{i \leq j \leq d} \left(1 - \dfrac{O_j}{R_i}\right)}{\ds \prod_{i < j \leq d \atop j \not \in \Ret} \left(1 - \dfrac{I_j}{R_i}\right) \prod_{i < j \leq d \atop j \in \Ret} \left(1 - t^{-1}\dfrac{I_j}{R_i}\right)},
\end{equation}
\begin{equation}
\label{eq:prod 4}
\prod_{c \in \mc{C}_{\la \cap \rho/\mu} - \mc{R}_{\rho/\la \cap \rho}} \dfrac{\ha{\la}(c)}{\ha{\mu}(c)} = \prod_{j \in \Ret} \dfrac{\ds \prod_{0 \leq i < j} \left(1 - \dfrac{R_j}{O_i}\right)}{\ds \prod_{1 \leq i < j \atop i \not \in \Ret} \left(1 - \dfrac{R_j}{I_i}\right) \prod_{1 \leq i < j \atop i \in \Ret} \left(1 - q^{-1}\dfrac{R_j}{I_i}\right)},
\end{equation}
\begin{equation}
\label{eq:prod 5}
\prod_{c \in \mc{R}_{\nu/\la \cup \rho} \cap \mc{C}_{\nu/\la \cup \rho}} \hl{\nu}(c)\ha{\nu}(c) = \prod_{i,j \in \Set \atop i \leq j} \left(1 - t\dfrac{O_j}{S_i}\right)\left(1 - q\dfrac{S_j}{O_i}\right),
\end{equation}
\begin{equation}
\prod_{c \in \mc{R}_{\la \cap \rho/\mu} \cap \mc{C}_{\la \cap \rho/\mu}} \hl{\mu}(c)\ha{\mu}(c) = \prod_{i,j \in \Ret \atop i < j} \left(1 - q^{-1}\dfrac{R_j}{I_i}\right)\left(1 - t^{-1}\dfrac{I_j}{R_i}\right),
\end{equation}
\begin{equation}
\prod_{c \in \mc{R}_{\nu/\la \cup \rho} \cap \mc{C}_{\la \cap \rho/\mu}} \hl{\la}(c)\ha{\la}(c) = \prod_{i \in \Set, j \in \Ret \atop i < j} \left(1 - \dfrac{I_j}{S_i}\right)\left(1 - \dfrac{R_j}{O_i}\right),
\end{equation}
\begin{equation}
\prod_{c \in \mc{C}_{\nu/\la \cup \rho} \cap \mc{R}_{\la \cap \rho/\mu}} \hl{\rho}(c)\ha{\rho}(c) = \prod_{i \in \Ret, j \in \Set \atop i \leq j} \left(1 - \dfrac{O_j}{R_i}\right)\left(1 - \dfrac{S_j}{I_i}\right).
\end{equation}
\end{lem}
\begin{proof}
We prove the equations \eqref{eq:prod 1} and \eqref{eq:prod 5}, the proofs of the other equations are analogue.

Let $i\in \Set$ and $i < j \leq d$. We consider the cells $c_{i,j}^{(1)},\ldots,c_{i,j}^{(m)}$ of $\nu$ in the same row as the $i$-th addable outer corner where the left corner of $c_{i,j}^{(1)}$ is below $I_j$ and the right corner of $c_{i,j}^{(m)}$ is below $O_j$, see Figure~\ref{fig:prod 1 cases}.
\begin{figure}
\begin{center}
\begin{tikzpicture}[scale=0.7]
	\draw[thick, blue, fill=myblue] (0,0) rectangle (1,1);
	\node at (.5,.5) {$+$};
	\draw[fill] (1,0) circle (2pt);
	\node at (1,-.5) {$S_i$};
	
	\draw[dotted] (1,0) rectangle (2.5,1);
	\draw[dotted] (3.5,0) rectangle (7,1);
	\draw[dotted] (2.5,2.5) -- (2.5,1);
	\draw[dotted] (7,2.5) -- (7,1);
	
	\draw[thick] (2.5,2.5) -- (7,2.5);
	\draw[fill] (2.5,2.5) circle (2pt);
	\node at (2.5,3) {$I_j$};
	\draw[fill] (7,2.5) circle (2pt);
	\node at (7,3) {$O_j$};
	
	\draw[thick] (2.5,0) rectangle (3.5,1);
	\node at (3,.5) {$c_{i,j}^{(1)}$};
	\draw[thick] (6,0) rectangle (7,1);
	\node at (6.5,.5) {$c_{i,j}^{(m)}$};

\begin{scope}[xshift=10cm]
	\draw[thick, blue, fill=myblue] (0,0) rectangle (1,1);
	\node at (.5,.5) {$+$};
	\draw[fill] (1,0) circle (2pt);
	\node at (1,-.5) {$S_i$};
	
	\draw[dotted] (1,0) rectangle (2.5,1);
	\draw[dotted] (3.5,0) rectangle (7,1);
	\draw[dotted] (2.5,2.5) -- (2.5,1);
	\draw[dotted] (7,2.5) -- (7,1);
	
	\draw[thick] (2.5,2.5) -- (7,2.5);
	\draw[thick, green!70!black, pattern=vertical lines, pattern color=green!70!black] (6,2.5) rectangle (7,4.5);
	\draw[fill] (2.5,2.5) circle (2pt);
	\node at (2.5,3) {$I_j$};
	\draw[fill] (7,4.5) circle (2pt);
	\node at (7,5) {$O_j$};
	
	\draw[thick] (2.5,0) rectangle (3.5,1);
	\node at (3,.5) {$c_{i,j}^{(1)}$};
	\draw[thick] (6,0) rectangle (7,1);
	\node at (6.5,.5) {$c_{i,j}^{(m)}$};
\end{scope}
\end{tikzpicture}
\end{center}
\caption{\label{fig:prod 1 cases} Two schematic situations for the cells $c_{i,j}^{(1)},\ldots,c_{i,j}^{(m)}$ where we use the color and shading code for different cells according to Table~\ref{tab:color code}.}
\end{figure}
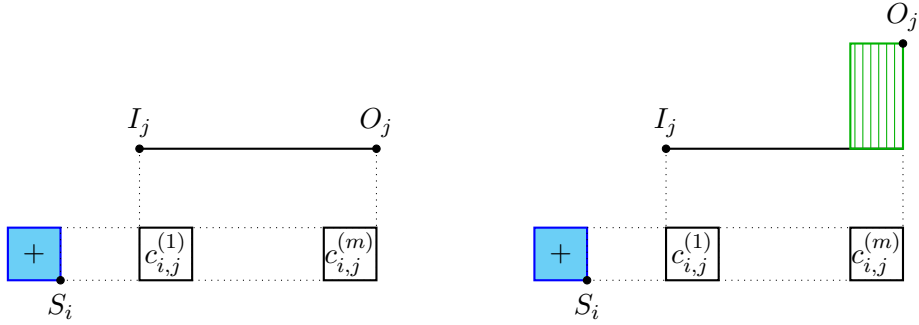
We have $\hl{\la}\left(c_{i,j}^{(k+1)}\right)= \hl{\nu}\left(c_{i,j}^{(k)}\right)$ for all $1 \leq k < m$ and
\[
\hl{\la}\left(c_{i,j}^{(1)}\right)= 1-\frac{I_j}{S_i}, \qquad 
\hl{\nu}\left(c_{i,j}^{(m)}\right) = \begin{cases}
1-\dfrac{O_j}{S_i} \qquad & j \notin \Set,\\[11pt]
1-t\dfrac{O_j}{S_j} & j \in \Set.
\end{cases}
\]
Hence the total contribution of the cells $c_{i,j}^{(1)},\ldots,c_{i,j}^{(m)}$ is
\[
\prod_{k=1}^m \frac{\hl{\la}\left(c_{i,j}^{(k)}\right)}{\hl{\nu}\left(c_{i,j}^{(k)}\right)} = \begin{cases}
\dfrac{1-\frac{I_j}{S_i}}{1-\frac{O_j}{S_i}} \qquad & j \notin \Set,\\[18pt]
\dfrac{1-\frac{I_j}{S_i}}{1-t\frac{O_j}{S_i}}  & j \in \Set.
\end{cases}
\]
For $i \in \Set$ let $c_i$ denote the $i$-th addable outer corner. By definition we have $\hl{\la}(c_i)=1$ and  $\hl{\nu}(c_i) = 1-t=1-t\frac{O_i}{S_i}$.
The left hand side of \eqref{eq:prod 1} is the product over the cells $c_{i,j}^{(1)},\ldots,c_{i,j}^{(m)}$ of the above form for all $i \in \Set$ and $i<j\leq d$ and all added outer corners $c_i$ with $i \in \Set$. Combining the above we therefore obtain the right hand side of \eqref{eq:prod 1}. \bigskip

Denote by $c_{i,j}$ the unique cell that lies in the same row as the $i$-th addable outer corner and in the same column as the $j$-th addable outer corner for $i \leq j$ as shown in Figure \ref{fig:prod 5}.
It is immediate that the $(q,t)$-hook lengths of $c_{i,j}$ with respect to $\nu$ are
\[
\hl{\nu}(c_{i,j}) = 1-t \dfrac{O_j}{S_i}, \qquad \qquad \ha{\nu}(c_{i,j}) = 1-q \dfrac{S_j}{O_i}.
\]
Since the cells contributing to the product on the left hand side of \eqref{eq:prod 5} are exactly all $c_{i,j}$ with $i,j \in \Set$, this implies the assertion.\qedhere
\begin{figure}[h]
\begin{center}
\begin{tikzpicture}[scale=0.65]
	\draw[dotted] (1,0) rectangle (3,1);
	\draw[dotted] (3,1) rectangle (4,2.5);

	\draw[thick, blue, fill=myblue] (0,0) rectangle (1,1);
	\node at (.5,.5) {$+$};
	\draw[fill] (1,0) circle (2pt);
	\node at (1,-.5) {$S_i=O_i$};
	
	\draw[thick, blue, fill=myblue] (3,2.5) rectangle (4,3.5);
	\node at (3.5,3) {$+$};
	\draw[fill] (4,2.5) circle (2pt);
	\node at (5.25,2.5) {$S_j=O_j$};
	
	\draw[thick] (3,0) rectangle (4,1);
	\node at (3.55,.4) {$c_{i,j}$};

\end{tikzpicture}
\end{center}
\caption{\label{fig:prod 5} The unique cell $c_{i,j}$ together with the positions of $S_i,S_j$.}
\end{figure}
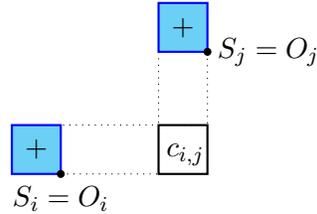
\end{proof}

\begin{lem}
\label{lem:alpha beta gamma pieces2}
Interchanging all $\hl{}$ and $\ha{}$ on the left hand side of the equations \eqref{eq:prod 1}-\eqref{eq:prod 4} has the effect of replacing $R_i$ by $\ov{R}_i$ and $S_j$ by $\ov{S}_j$ for all $1 \leq i \leq d$ and $0 \leq j \leq d$ on the right hand side of the equations.
\end{lem}
\begin{proof}
The proof is analogous to that of Lemma~\ref{lem:alpha beta gamma pieces}.
\end{proof}

\subsection{Proof of Equation \eqref{eq:weights formula}}
\label{sec:weights formula}

In the following we use the symbols $\la,\rho,-,+$ as shorthand for the skew shapes $\la / (\la \cap \rho)$, $\rho/(\la \cap \rho)$, $(\la \cap \rho)/\mu$, $\nu/(\la \cup \rho)$, respectively, as shown in Table \ref{tab:color code}. It will be clear from context whether $\la, \rho$ are partitions or symbols representing a certain skew shape.

\begin{table}[h]
\begin{center}
$
\begin{array}{c |cc}
\text{skew shape} &  \text{symbol} & \text{color and shading code} \\
\hline
\lambda/ (\lambda \cap \rho) = (\lambda \cup \rho)/\rho  & \la &  \lacell \\
\rho/ (\lambda \cap \rho)=(\lambda \cup \rho) / \lambda & \rho & \rhocell \\
(\lambda \cap \rho)/\mu & - & \mucell \\
\nu/ (\lambda \cup \rho) & + & \nucell
\end{array}
$
\caption{\label{tab:color code} The color and shading code for marking the cells of the skew shapes $\lambda / (\lambda \cap \rho), \rho / (\lambda \cap \rho), (\lambda \cap \rho)/ \mu,$ and $\nu / (\lambda \cup \rho)$.}
\end{center}
\end{table}

Let $c$ be a cell of $\nu$, and define $\omega_{\la,\rho}^{\mu,\nu}(c)$ to be the contribution of $c$ to the right-hand side of \eqref{eq:weights def}. Let $R_c$ (resp., $C_c$) be the subset of skew shapes $\kappa \in \{\la, \mu, -, +\}$ such that $c$ lies in the same row (resp., column) as a cell of $\kappa$. We call $R_c$ (resp., $C_c$) the \deff{row type} (resp., \deff{column type}) of $c$, and define the \deff{type} of $c$ to be the pair $(R_c,C_c)$. Since we assume that $\la / \mu, \nu / \rho$ are horizontal strips and $\rho/\mu, \nu /\la$ are vertical strips, the possible row types are either $R_c = \{ \rho\}$ or $R_c \subseteq \{ \la, +,-\}$. Similarly, the possible column types are either $C_c = \{\la \}$ or $C_c \subseteq \{ \rho, -, +\}$. It is immediate that $\omega_{\la,\rho}^{\mu,\nu}(c)$ is defined piecewise depending on the row and column type of $c$. It follows from \eqref{eq:weights def} that  $\omega_{\la,\rho}^{\mu,\nu}(c) = 1$ if $R_c = \{ \rho \}$ or $C_c = \{\la \}$, hence we ignore these two cases for now. Table \ref{tab:cell weights} contains formulas for the other cases where we abbreviate $b_\kappa(c)$ by $\kappa$ for reasons of readability.

\begin{table}[h]
\begin{center}
$
\begin{array}{c|cccc|cccc}
\text{\diagbox{$\; R_c$}{$C_c$}} & \emptyset & \{-\} & \{+\} & \{-,+\} & \{\rho\} & \{\rho,-\} & \{\rho,+\} & \{\rho,-,+\} \\\hline &&&&&\\[-8pt]
\emptyset
	& 1 & \frac{\la}{\mu} & \frac{\rho}{\nu} & \frac{\la \, \rho}{\mu \,\nu} & 1 & \frac{\la}{\mu} & \frac{\rho}{\nu} & \frac{\la\,\rho}{\mu\, \nu}	
\\[10pt]

\{-\}
	& \frac{\mu}{\rho} & \frac{\la}{\rho} & \frac{\mu}{\nu} & \frac{\la}{\nu} & \frac{\mu}{\rho} & \frac{\la}{\rho} & \frac{\mu}{\nu} & \frac{\la}{\nu}
\\[10pt]

\{+\}
	& \frac{\nu}{\la} & \frac{\nu}{\mu} & \frac{\rho}{\la} & \frac{\rho}{\mu} & \frac{\nu}{\la} & \frac{\nu}{\mu} & \frac{\rho}{\la} & \frac{\rho}{\mu}
\\[10pt]

\{-,+\}
	& \frac{\mu\, \nu}{\la\,\rho} & \frac{\nu}{\rho} & \frac{\mu}{\la} & 1 & \frac{\mu\, \nu}{\la\, \rho} & \frac{\nu}{\rho} & \frac{\mu}{\la} & 1
\\[8pt]
\hline &&&&\\[-8pt]

\{\la\}
	& 1 & \frac{\la}{\mu} & \frac{\rho}{\nu} & \frac{\la\,\rho}{\mu\,\nu} & 1 & \frac{\la}{\mu} & \frac{\rho}{\nu} & \frac{\la\,\rho}{\mu\,\nu}
\\[10pt]

\{\la,-\}
	& \frac{\mu}{\rho} & \frac{\la}{\rho} & \frac{\mu}{\nu} & \frac{\la}{\nu} & \frac{\mu}{\rho} & \frac{\la}{\rho} & \frac{\mu}{\nu} & \frac{\la}{\nu}
\\[10pt]

\{\la,+\}
	& \frac{\nu}{\la} & \frac{\nu}{\mu} & \frac{\rho}{\la} & \frac{\rho}{\mu} & \frac{\nu}{\la} & \frac{\nu}{\mu} & \frac{\rho}{\la} & \frac{\rho}{\mu}
\\[10pt]

\{\la,-,+\}
	& \frac{\mu\,\nu}{\la\,\rho} & \frac{\nu}{\rho} & \frac{\mu}{\la} & 1 & \frac{\mu\,\nu}{\la\,\rho} & \frac{\nu}{\rho} & \frac{\mu}{\la} & 1
\end{array}
$
\end{center}
\caption{\label{tab:cell weights} Explicit formulas for $\omega_{\la,\rho}^{\mu,\nu}(c)$ depending on the type of $c$.}
\end{table}

We prove the formula for the second entry in the first row; the other formulas are obtained in a similar manner. Let $c$ be a cell with type $(\emptyset, \{-\})$. This implies that $c$ is an element of $\mc{C}_{\la/\mu}$, $\mc{C}_{\rho/\mu}$ but not of $\mc{R}_{\la/\mu}$, $\mc{R}_{\rho/\mu}$, $\mc{R}_{\nu/\la}$, $\mc{R}_{\nu/\rho}$, $\mc{C}_{\nu/\la}$, $\mc{C}_{\nu/\rho}$. Hence we have
\begin{equation*}
\omega_{\la,\rho}^{\mu,\nu}(c)= \frac{b_{\rho}(c)}{b_{\mu}(c)}.
\end{equation*}
The type $(\emptyset, \{-\})$ of $c$ implies $a_{\rho}(c) = a_{\la}(c)$ and $\ell_{\rho}(c) = \ell_{\la}(c)$ and hence $b_\rho(c)=b_\la(c)$, which proves the formula for the second entry in the first row. \bigskip

For $R \subseteq \{ \la, -, +\}$ and $C \subseteq \{ \rho, -,+\}$, denote by $[R,C](c)$ the expression for $\omega_{\la,\rho}^{\mu,\nu}(c)$ given by row $R$ and column $C$ of Table \ref{tab:cell weights}. Observe that the four quadrants of Table \ref{tab:cell weights} are identical, i.e., for all $R,C \subseteq \{-,+\}$ one has  
\begin{equation}
[R,C](c) = [R \cup \{\la\},C](c) = [R ,C\cup \{\rho\}](c) = [R  \cup \{\la\},C\cup \{\rho\}](c).
\end{equation}
Furthermore, for $R,C \subseteq \{-,+\}$ one has  
\begin{equation}
[R,C](c) = \prod_{\kappa \in R} [\{\kappa\},\emptyset](c)
\prod_{\kappa \in C} [\emptyset,\{\kappa\}](c).
\end{equation}
Using these observations, together with the fact that $\omega_{\la,\rho}^{\mu,\nu}(c) = 1$ if $R_c = \{\rho\}$ or $C_c = \{\lambda\}$, we obtain the formula
\begin{multline*}
\prod_{c \in \nu} \omega_{\la,\rho}^{\mu,\nu}(c) =
\prod_{c: \, \substack{R_c \neq \{\rho\} \\ C_c \neq \{\la\}} }  \omega_{\la,\rho}^{\mu,\nu}(c) =
\prod_{c: \, \substack{R_c \neq \{\rho\} \\ C_c \neq \{\la\}} } [R_c \cap\{-,+\},C_c\cap \{-,+\}](c)
\\
=\prod_{ c: \, \substack{\{-\} \subseteq R_c \\ C_c \neq \{\la\}} } [\{-\},\emptyset](c)
\prod_{ c: \, \substack{\{-\} \subseteq C_c \\ R_c \neq \{\rho\}} } [\emptyset,\{-\}](c) \\
\times
\prod_{ c: \, \substack{\{+\} \subseteq R_c \\ C_c \neq \{\la\}} } [\{+\},\emptyset](c)
\prod_{ c: \, \substack{\{+\} \subseteq C_c \\ R_c \neq \{\rho\}} } [\emptyset,\{+\}](c).
\end{multline*}
By the definition of the type of a cell, the condition for $c$ in the first product above is equivalent to $c \in \mc{R}_{\la \cap \rho / \mu} - \mc{C}_{\la / \la \cap \rho}$. Since $[\{-\},\emptyset](c) = \frac{b_\mu(c)}{b_\rho(c)}$, the first product in the above expression is equal to the first product of \eqref{eq:weights formula}. By analogous arguments for the three other products, we see that the above is equal to \eqref{eq:weights formula}.

\section{Degenerations}
\label{sec:degenerations}
\subsection{$q$-Whittaker and Hall--Littlewood degenerations}
\label{sec:qWhit and HL}

In this section we study the $t = 0$ ($q$-Whittaker dual row insertion) and $q = 0$ (Hall--Littlewood  dual row insertion) degenerations of the probabilities \linebreak $\mc{P}_{\la,\rho}(\mu^{(\Ret)},\nu^{(\Set)})[q,t] = p_{\Ret,\Set}[q,t]$. By Lemma~\ref{lem:inverting q and t} these two specializations can be translated to the $t \rightarrow \infty$ ($q$-Whittaker  dual column insertion) and $q\rightarrow \infty$ (Hall--Littlewood  dual column insertion) degenerations. The key to understanding these degenerations is the combinatorial description of the exponents of $q$ and $t$ in the monomial $\tau_{\Ret,\Set}$ presented in \S \ref{sec:recovering dual RSK}. 
The notation $\shift(\Ret,\Set) := \{i: |\Ret\cap [i+1,d]| >|\Set \cap [i+1,d]|\}$ turns out useful for the next two lemmas.

\begin{lem}[$q$-Whittaker specialization]
\label{lem:q Whittaker spez}
Let $\la,\rho$ be two partitions, $d$ the number of removable inner corners of $\la \cap \mu$ and $\mu=\mu^{(\Ret)}$. Denote by $\alpha \subseteq [0,d]$ the set of integers $i$ with $(i+1)\in \Ret$ such that the $i$-th addable outer corner is in the same row as the $(i+1)$-st removable inner corner, i.e., $q^x S_i=R_{i+1}$ for some $x \in \mathbb{N}$. Then $p_{\Ret,\Set}[q,0]=0$ unless $\shift(\Ret,\Set) \subseteq \alpha\cup \{0\}$, in which case we have
\begin{equation}
\label{eq:prob in q-whittaker}
p_{\Ret,\Set}[q,0]=
\dfrac{\prod\limits_{i \in \Set\setminus\left(\Ret\cup\{0\}\right)} \left(1- \dfrac{S_i}{I_i}\right)
\prod\limits_{i \in \alpha\setminus \Set}\left(1-\dfrac{R_{i+1}}{O_i}\right)}
{\prod\limits_{i \in \alpha \setminus \left(\Ret\cup\{0\}\right)}\left(1-\dfrac{R_{i+1}}{I_i}\right)}
\prod_{i \in \shift(\Ret,\Set)} \frac{R_{i+1}}{S_{i}}.
\end{equation}	
\end{lem}

\begin{proof}
We use the description $p_{\Ret,\Set}=\tau_{\Ret,\Set} \zeta_{\Ret,\Set}$ in \eqref{eq:prob rewritten} to calculate the probability for $t=0$. Analogously to \eqref{eq:zeta at 0}, it is immediate that $\lim_{t\rightarrow 0}\zeta_{\Ret,\Set} \neq 0$. Therefore $\tau_{\Ret,\Set}$ determines if the probability is equal to $0$ or not.
Analogously to Lemma~\ref{lem:exponents} and the discussion thereafter, we calculate $\tau_{\Ret,\Set}$ by constructing two paths and pairing all up steps and down steps of each path with the exception of one down step.
Define $p:[d] \rightarrow [0,d]$ as the injection such that $U_i$ is paired with $D_{p(i)}$ for $i \in [d]$ and by $c_i$ the contribution of the pair $(U_i,D_{p(i)})$ to $\tau_{\Ret,\Set}$.
 Then $c_i=\left(\frac{S_{p(i)}}{R_i}\right)^z$ if $i \in \Ret$ and $c_i=\left(\frac{O_{p(i)}}{I_i}\right)^z$ if $i \in [d]\setminus \Ret$, where $z$ is a certain integer satisfying
\[
z= \begin{cases}
\text{positive} \quad & U_i \text{ is strictly above the $x$-axis (this implies $p(i)\geq i$)}, \\
0 & \text{the starting point of $U_i$ is on the $x$-axis}, \\
\text{positive} \quad & U_i \text{ is weakly below the $x$-axis (this implies $p(i) < i$)}.
\end{cases}
\]
It is immediate that the exponent of $t$ in $c_i$ is positive for $i \in \Ret$ unless $U_i$ starts on the $x$-axis, or $p(i)=i-1$ and $\frac{R_{i}}{S_{i-1}}$ is a power of $q$, i.e., $(i-1) \in \alpha \cap \Set$. This implies inductively, starting with the largest $i\in \Ret$, that $\left.c_i\right|_{t=0} \neq 0$ for all $i \in \Ret$ if and only if $p(i) \in \{i,i-1\}$ for all $i \in \Ret$ and $p(i)=i$ if $(i-1) \notin \alpha$. As a consequence, the second path consists of alternating up and down steps where all up steps start at the $x$-axis, implying the first part of the assertion.

\begin{figure}[h]
\begin{center}
\begin{tikzpicture}[scale=0.75]

\draw (0,0) node{$\textcolor{blue}{D_0}$};
\draw (1.25,0) node{$\textcolor{blue}{D_1}$};
\draw (2.5,0) node{$\mathcircled{\textcolor{blue}{D_2}}$};
\draw (3.75,0) node{$\mathcircled{\textcolor{blue}{D_3}}$};
\draw (5,0) node{$\textcolor{blue}{D_4}$};
\draw (6.25,0) node{$\mathcircled{\textcolor{blue}{D_5}}$};
\draw (7.5,0) node{$\textcolor{blue}{D_6}$};
\draw (8.75,0) node{$\textcolor{blue}{D_7}$};
\draw (10,0) node{$\mathcircled{\textcolor{blue}{D_8}}$};

\begin{scope}[xshift=0cm, yshift=-1.75cm]
\draw (1.25,0) node{$\textcolor{red}{U_1}$};
\draw (2.5,0) node{$\mathcircled{\textcolor{red}{U_2}}$};
\draw (3.75,0) node{$\textcolor{red}{U_3}$};
\draw (5,0) node{$\mathcircled{\textcolor{red}{U_4}}$};
\draw (6.25,0) node{$\mathcircled{\textcolor{red}{U_5}}$};
\draw (7.5,0) node{$\textcolor{red}{U_6}$};
\draw (8.75,0) node{$\textcolor{red}{U_7}$};
\draw (10,0) node{$\mathcircled{\textcolor{red}{U_8}}$};
\end{scope}

\draw[->, thick] (10,-1.25) -- (10,-.55);
\draw[->, thick] (6.25,-1.25) -- (6.25,-.55);
\draw[->, thick] (2.5,-1.25) -- (2.5,-.55);

\draw[->, thick] (4.65,-1.25) -- (4.1,-.5);

\draw[dotted, thick] (2.15,-1.25) -- (1.6,-.55);
\draw[dotted, thick] (5,-1.2) -- (5,-.55);
\draw[dotted, thick] (5.9,-1.25) -- (5.35,-.55);

\begin{scope}[yshift=-4.5cm, xshift=-3.25cm, scale=.9]
\draw[dashed,thin,gray] (0,0) -- (8,0);
\draw (0,0) -- (1,1) -- (2,0) -- (3,-1) -- (4,0) -- (5,1) -- (6,0) -- (7,1) -- (8,0);
\filldraw (0,0) circle (2pt);
\filldraw (1,1) circle (2pt);
\filldraw (2,0) circle (2pt);
\filldraw (3,-1) circle (2pt);
\filldraw (4,0) circle (2pt);
\filldraw (5,1) circle (2pt);
\filldraw (6,0) circle (2pt);
\filldraw (7,1) circle (2pt);
\filldraw (8,0) circle (2pt);
\node[red] at (.8,.3) {$U_2$};
\node[red] at (3.8,-.75) {$U_4$};
\node[red] at (4.8,.3) {$U_5$};
\node[red] at (6.8,.3) {$U_8$};
\node[blue] at (1.75,.75) {$D_2$};
\node[blue] at (2.75,-.3) {$D_3$};
\node[blue] at (5.75,.75) {$D_5$};
\node[blue] at (7.75,.75) {$D_8$};

\begin{scope}[xshift=9cm, scale=.9]
\draw[dashed,thin,gray] (1,0) -- (10,0);
\draw (1,1) -- (2,0) -- (3,1) -- (4,0) -- (5,1) -- (6,0) --(7,1) -- (8,0) -- (9,1) -- (10,0);
\filldraw (1,1) circle (2pt);
\filldraw (2,0) circle (2pt);
\filldraw (3,1) circle (2pt);
\filldraw (4,0) circle (2pt);
\filldraw (5,1) circle (2pt);
\filldraw (6,0) circle (2pt);
\filldraw (7,1) circle (2pt);
\filldraw (8,0) circle (2pt);
\filldraw (9,1) circle (2pt);
\filldraw (10,0) circle (2pt);
\node[blue] at (1.8,.75) {$D_0$};
\node[blue] at (3.8,.75) {$D_1$};
\node[blue] at (5.8,.75) {$D_4$};
\node[blue] at (7.8,.75) {$D_6$};
\node[blue] at (9.8,.75) {$D_7$};
\node[red] at (2.8,.3) {$U_1$};
\node[red] at (4.8,.3) {$U_3$};
\node[red] at (6.8,.3) {$U_6$};
\node[red] at (8.8,.3) {$U_7$};
\end{scope}
\end{scope}
\end{tikzpicture}
\caption{\label{fig:whittaker example} An example of the map $p$ restricted to $\Ret=\{2,4,5,8\}$ where $\alpha=\{1,3,4\}$ together with the corresponding paths. The dotted lines denote the alternative choices for $p(2),p(4)$ and $p(5)$ which preserve $\left.\tau_{\Ret,\Set}\right|_{t=0}\neq 0$.}
\end{center}
\end{figure}
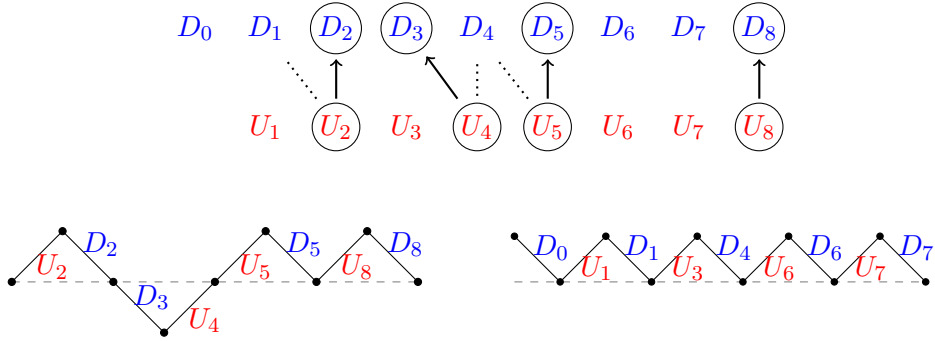

Assume $\Set$ is chosen such that $\left.\tau_{\Ret,\Set}\right|_{t=0} \neq 0$. By the above considerations it is immediate that $\left.c_i\right|_{t=0}= \frac{R_i}{S_{i-1}}$ if $p(i)=i-1$, i.e., $(i-1) \in \shift(\Ret,\Set)$, and $\left.c_i\right|_{t=0}=1$ otherwise. Hence we obtain
\[
\left.\tau_{\Ret,\Set}\right|_{t=0} = \prod_{i \in \shift(\Ret,\Set)} \frac{R_{i+1}}{S_{i}}.
\]
Finally, observe that all fractions in the products of \eqref{eq:zeta} have positive powers of $t$ except the term $\frac{S_i}{I_i}$ for $i \in \Set\setminus \left(\Ret\cup\{0\}\right)$, the fraction $\frac{R_{i+1}}{O_i}$ for $i \in \alpha\setminus \Set$ and $\frac{R_{i+1}}{I_i}$ for $i \in \alpha\setminus \left(\Ret\cup\{0\}\right)$, which proves the assertion.
\end{proof}

By the above lemma and Definition~\ref{def:qtRSK insertion}, we can describe the \deff{$q$-Whittaker dual row insertion} as follows.
Let $T$ be an SSYT and $i_1< \cdots <i_n$ be positive integers which we want to insert simultaneously into $T$.
For the $k$-th step of the insertion algorithm let $\mu=\mu^{(\Ret)}=T^{(k-1)}, \la=T^{(k)}$ and denote by $\rho$ the shape of the SSYT obtained in the previous step restricted to the entries at most $k-1$. Denote by $a_0,\ldots,a_d$ the addable outer cells of $\la \cup \rho$ and define $\alpha$ as the set of indices $j$ such that $(j+1)\in \Ret$ and $a_j$ is in the same row as $c_{j+1}$. For simplicity, we use the notation
\[
\wh{\Ret} = \begin{cases}
\Ret \cup \{0\} \qquad &\text{if } k \in \{i_1,\ldots,i_n\},\\
\Ret & \text{otherwise}.
\end{cases}
\]
We (re-)insert an entry $k$ into $T$ for each $j \in \wh{\Ret}$, starting with the smallest $j$, according to the following rule.
\begin{itemize}
\item If $(j-1) \notin \alpha$ or if there is already an entry $k$ in the box $a_{j-1}$, (re-)insert $k$ into the box $a_j$ with probability $1$.
\item If $(j-1) \in \alpha$ and there is no entry $k$ in the box $a_{j-1}$, reinsert $k$ into the box $a_{j-1}$ with probability $\frac{R_{j}}{S_{j-1}}$ if $(j-1)\in \Ret$ and with probability $\frac{R_{j}}{S_{j-1}}\left(1-\frac{S_{j-1}}{I_{j-1}} \right)$ otherwise; and reinsert $k$ into $a_j$ with probability $\left(1-\frac{R_j}{O_{j-1}}\right)$.
\item If $k$ is placed into a cell $a_j$ with an entry $z$ in $T$, then $z$ is bumped and reinserted in the $z$-th insertion step.
\end{itemize}
Finally, we have to multiply all individual probabilities from above and divide the result by $\prod_{j \in \alpha\setminus \left(\Ret\cup\{0\}\right)}\left(1-\frac{R_{j+1}}{I_j}\right)$.

\begin{rem}
The $q$-Whittaker dual row insertion was first described by Matveev and Petrov \cite[\S 5.1]{MatveevPetrov17} for Gelfand Tsetlin patterns which are known to be in bijection to semistandard Young tableaux, see for example
\cite[p. 313]{EC2}.
\end{rem}

Next we describe our probabilities for $q=0$. For readability, we write $\Ret^c= [d] \setminus \Ret$ and $\Set^c=[0,d] \setminus \Set$. 

\begin{lem}[Hall--Littlewood specialization]
\label{lem:Hall-Littlewood spez}
Let $\la,\rho$ be two partitions, $d$ the number of removable inner corners of $\la \cap \rho$ and $\mu=\mu^{(\Ret)}$. We define $\beta \subseteq [0,d]$ as the set of integers $i$ such that the $i$-th addable outer corner is in the column next to the $(i+1)$-st removable inner corner, i.e., $t^x O_i=I_{i+1}$ for a suitable $x \in \mathbb{N}$ and $(i+1) \in \Ret^c$.
Then $p_{\Ret,\Set}[0,t]=0$ unless $\shift(\Ret^c,\Set^c) \subseteq \beta$ in which case we have 
\begin{equation}
\label{eq:prob in HL}
p_{\Ret,\Set}[0,t]= \dfrac{\prod\limits_{i \in \Set\cap \beta}\left(1-\dfrac{I_{i+1}}{S_i}\right)
\prod\limits_{i \in \Ret \setminus \Set} \left(1-\dfrac{O_i}{R_i} \right)}
{\prod\limits_{i \in \Ret \cap \beta}\left(1-\dfrac{I_{i+1}}{R_i}\right)}
\prod_{i \in \shift(\Ret^c,\Set^c)}\frac{I_{i+1}}{O_i}.
\end{equation}
\end{lem}

\begin{proof}
It is immediate, that $\lim_{q\rightarrow 0} \zeta_{\Ret,\Set} \neq 0$, hence $\left.\tau_{\Ret,\Set}\right|_{q=0}$ determines if $p_{\Ret,\Set}[0,t]$ is nonzero.
Analogously to the proof of Lemma~\ref{lem:q Whittaker spez} we denote by $p:[d] \rightarrow [0,d]$ the pairing of up and down steps and by $c_i$ the contribution of the pair $(U_i,D_{p(i)})$ to $\tau_{\Ret,\Set}$. For $i \in \Ret^c$ it follows that the exponent of $q$ in $c_i$ is positive unless $U_i$ starts on the $x$-axis or $p(i)=i-1$ and $O_{i-1} t^x=I_{i}$, i.e., $(i-1)\in \beta$. This implies inductively, starting with the largest $i \in \Ret^c$, that $\left.c_i\right|_{q=0}\neq 0$ for all $i \in \Ret^c$ if and only if $p(i) \in \{i,i-1\}$ for all $i \in \Ret^c$ and $p(i)=i$ if $(i-1) \notin \beta$. See Figure~\ref{fig:HL example} for an example. As a consequence the first path consists in this case alternately of up and down steps and we obtain
\[
\left.\tau_{\Ret,\Set}\right|_{q=0}=\prod_{i \in \shift\left(\Ret^c,\Set^c\right)}\frac{I_{i+1}}{O_i}.
\]

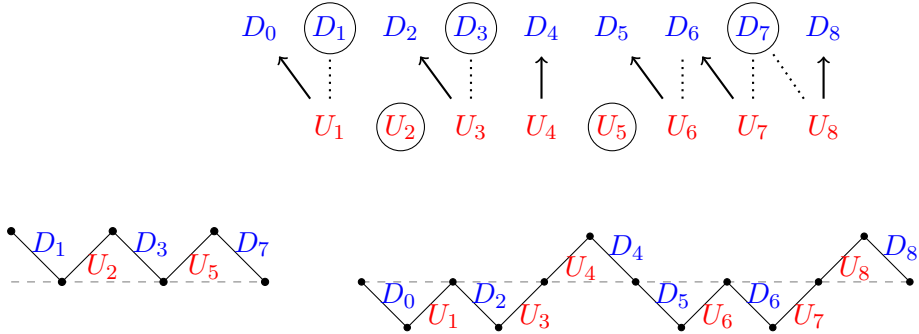
\begin{figure}[h]
\begin{center}
\begin{tikzpicture}[scale=0.75]

\draw (0,0) node{$\textcolor{blue}{D_0}$};
\draw (1.25,0) node{$\mathcircled{\textcolor{blue}{D_1}}$};
\draw (2.5,0) node{$\textcolor{blue}{D_2}$};
\draw (3.75,0) node{$\mathcircled{\textcolor{blue}{D_3}}$};
\draw (5,0) node{$\textcolor{blue}{D_4}$};
\draw (6.25,0) node{$\textcolor{blue}{D_5}$};
\draw (7.5,0) node{$\textcolor{blue}{D_6}$};
\draw (8.75,0) node{$\mathcircled{\textcolor{blue}{D_7}}$};
\draw (10,0) node{$\textcolor{blue}{D_8}$};

\begin{scope}[xshift=0cm, yshift=-1.75cm]
\draw (1.25,0) node{$\textcolor{red}{U_1}$};
\draw (2.5,0) node{$\mathcircled{\textcolor{red}{U_2}}$};
\draw (3.75,0) node{$\textcolor{red}{U_3}$};
\draw (5,0) node{$\textcolor{red}{U_4}$};
\draw (6.25,0) node{$\mathcircled{\textcolor{red}{U_5}}$};
\draw (7.5,0) node{$\textcolor{red}{U_6}$};
\draw (8.75,0) node{$\textcolor{red}{U_7}$};
\draw (10,0) node{$\textcolor{red}{U_8}$};
\end{scope}

\draw[dotted, thick] (1.25,-1.25) -- (1.25,-.55);
\draw[dotted, thick] (3.75,-1.25) -- (3.75,-.55);
\draw[->, thick] (5,-1.25) -- (5,-.5);
\draw[dotted, thick] (7.5,-1.25) -- (7.5,-.55);
\draw[dotted, thick] (8.75,-1.25) -- (8.75,-.55);
\draw[->, thick] (10,-1.25) -- (10,-.5);

\draw[->, thick] (.9,-1.25) -- (.35,-.5);
\draw[->, thick] (3.4,-1.25) -- (2.85,-.5);
\draw[->, thick] (7.15,-1.25) -- (6.6,-.5);
\draw[->, thick] (8.4,-1.25) -- (7.85,-.5);
\draw[dotted, thick] (9.65,-1.25) -- (9.1,-.45);

\begin{scope}[yshift=-4.5cm, xshift=-3.5cm, scale=.9]
\draw[dashed,thin,gray] (-1,0) -- (4,0);
\draw (-1,1) -- (0,0) -- (1,1) -- (2,0) -- (3,1) -- (4,0);
\filldraw (-1,1) circle (2pt);
\filldraw (0,0) circle (2pt);
\filldraw (1,1) circle (2pt);
\filldraw (2,0) circle (2pt);
\filldraw (3,1) circle (2pt);
\filldraw (4,0) circle (2pt);
\node[red] at (.8,.3) {$U_2$};
\node[red] at (2.8,.3) {$U_5$};
\node[blue] at (-.25,.75) {$D_1$};
\node[blue] at (1.75,.75) {$D_3$};
\node[blue] at (3.75,.75) {$D_7$};

\begin{scope}[xshift=5cm, scale=.9]
\draw[dashed,thin,gray] (1,0) -- (13,0);
\draw (1,0) -- (2,-1) -- (3,0) -- (4,-1) -- (5,0) -- (6,1) --(7,0) -- (8,-1) -- (9,0) -- (10,-1) -- (11,0) -- (12,1) -- (13,0);
\filldraw (1,0) circle (2pt);
\filldraw (2,-1) circle (2pt);
\filldraw (3,0) circle (2pt);
\filldraw (4,-1) circle (2pt);
\filldraw (5,0) circle (2pt);
\filldraw (6,1) circle (2pt);
\filldraw (7,0) circle (2pt);
\filldraw (8,-1) circle (2pt);
\filldraw (9,0) circle (2pt);
\filldraw (10,-1) circle (2pt);
\filldraw (11,0) circle (2pt);
\filldraw (12,1) circle (2pt);
\filldraw (13,0) circle (2pt);
\node[blue] at (1.8,-.3) {$D_0$};
\node[blue] at (3.8,-.3) {$D_2$};
\node[blue] at (6.8,.75) {$D_4$};
\node[blue] at (7.8,-.3) {$D_5$};
\node[blue] at (9.8,-.3) {$D_6$};
\node[blue] at (12.8,.75) {$D_8$};
\node[red] at (2.8,-.75) {$U_1$};
\node[red] at (4.8,-.75) {$U_3$};
\node[red] at (5.8,.3) {$U_4$};
\node[red] at (8.8,-.75) {$U_6$};
\node[red] at (10.8,-.75) {$U_7$};
\node[red] at (11.8,.3) {$U_8$};
\end{scope}

\end{scope}
\end{tikzpicture}
\caption{\label{fig:HL example} An example of $p$ restricted to $\Ret^c=\{1,3,4,6,7,8\}$ where $\beta=\{0,2,5,6,7\}$ together with the corresponding paths. The dotted lines denote the alternative choices for $p(1),p(3),p(6),p(7)$ and $p(8)$ which preserve that $\left.\tau_{\Ret,\Set}\right|_{q=0}$ is nonzero.}
\end{center}
\end{figure}

We observe that all fractions in the products of \eqref{eq:zeta} have positive powers in $q$ except for 
$\frac{I_{i+1}}{S_{i}}$ for $i\in \Set \cap \beta$, the fraction $\frac{O_i}{R_i}$ for $i \in \Ret \setminus \Set$ and $\frac{I_{i+1}}{R_{i}}$ for $i \in \Ret\cap \beta$ which proves the claim.
\end{proof}

We translate the dual growth rules for Hall--Littlewood polynomials into a \deff{Hall--Littlewood dual row insertion} by Definition~\ref{def:qtRSK insertion}. Let $T$ be an SSYT and $i_1< \cdots<i_n$ be positive integers which we want to insert simultaneously into $T$. For the $k$-th step of the insertion algorithm let $\mu=\mu^{(\Ret)}=T^{(k-1)}, \la=T^{(k)}$ and denote by $\rho$ be the shape of the SSYT obtained in the previous step restricted to entries at most $k-1$. Let $a_0, \ldots, a_d$ the addable outer corners of $\la \cup \rho$ and define $\wh{\Ret}$ as before. We (re-)insert $k$ for each $j\in \wh{R}$, starting with the smallest $j$, according to the following rules.
\begin{itemize}
\item For $j \notin \beta$ or $j+1 \in \Ret$ we (re-)insert $k$ into the cell $a_j$ with probability $1$.
\item If $j \in \beta$ we (re-)insert $k$ into a box $a_i$, where $j \leq i$ satisfies $\{j,j+1,\ldots,i-1\} \subseteq \beta$, with probability
\[
\left(1- \dfrac{I_{i+1}}{S_i}\right)^{[i \in \beta]}\left(1-\dfrac{O_j}{R_j}\right)^{[i \neq j] }
\left(1-\frac{I_{j+1}}{R_j}\right)^{-1}
\prod\limits_{l=j}^{i-1}\dfrac{I_{l+1}}{O_l},
\]
where $[\text{expression}]$ is the Iverson bracket which is $1$ if `$\text{expression}$' is true and $0$ otherwise, and using the convention $\frac{O_0}{R_0}=\frac{I_1}{R_0}=0$.
\item If $k$ is placed into a cell $a_i$ with an entry $z$ in $T$, then $z$ is bumped and reinserted in the $z$-th insertion step.
\end{itemize}



\begin{rem}
By Lemma~\ref{lem:inverting q and t} and \eqref{eq:sym of dual RSK}, one can define the \deff{$q$-Whittaker dual column insertion} as the process obtained by first conjugating all objects, applying the Hall--Littlewood dual-row insertion and finally replacing $t$ by $q^{-1}$. The $q$-Whittaker dual column insertion was first introduced by Matveev and Petrov in \cite[\S 5.4]{MatveevPetrov17} for Gelfand-Tsetlin patterns. A Hall--Littlewood dual column insertion can be defined analogously.
\end{rem}

\subsection{Jack limit}
\label{sec:Jack}
The Jack polynomials $P_\la(x;\alpha)$ are a one-parameter family of symmetric functions introduced by Henry Jack \cite{Jack70}. They are obtained from the Macdonald polynomials $P_\la(x;q,t)$ by setting $q = t^{\alpha}$, and taking the limit $t \rightarrow 1$. 
The next theorem states an intriguing property for the $\dRSK$ correspondence when restricting to the Jack polynomial setting. It can be seen easily that it does not hold in the general Macdonald setting.

\begin{thm}
\label{thm:Jack}
Let $A$ be an $m \times n$ $\{0,1\}$-matrix such that at most one entry in each column is nonzero. In the Jack limit of $\dRSK$,  interchanging adjacent columns of $A$ does not affect the distribution of the $P$-tableaux.
\end{thm}

As shown in the next section, restricting to $\{0,1\}$-matrix with at most one nonzero entry per column yields a $(q,t)$-variation of RS for words. The above theorem and Lemma~\ref{lem:inverting q and t} imply that interchanging two adjacent rows, when restricted to $\{0,1\}$-matrices with at most one nonzero entry per row, does not affect the distribution of the $Q$-tableaux.
Before proving the above theorem we show that it can not be extended to all $\{0,1\}$-matrices. Regard the matrix 
 \[
A=\begin{pmatrix}
0 & 0 & 1 \\
1 & 1 & 0 \\
0 & 0 & 1
\end{pmatrix}.
\]
One can check easily that the tableau
\[
T= \young(12,23)\;, 
\]
can not be obtained as the $P$-tableau of any dual growth associated to $A$.
On the other hand, by interchanging the second and third column, there are two possibilities to obtain $T$ as a $P$-tableau as shown next.

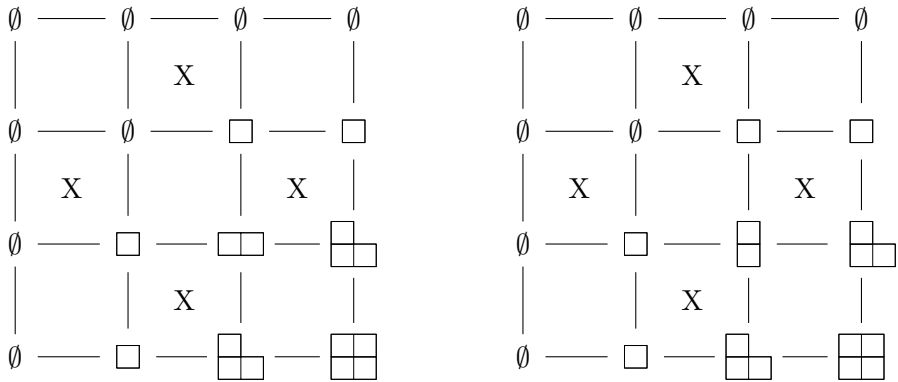
\begin{center}
\begin{tikzpicture}[scale=1.5]
\Yboxdim{.2cm}
\node at (0,3) {$\emptyset$};
\node at (0,2) {$\emptyset$};
\node at (1,2) {$\emptyset$};
\node at (0,1) {$\emptyset$};
\node at (0,0) {$\emptyset$};
\node at (1,3) {$\emptyset$};
\node at (2,3) {$\emptyset$};
\node at (3,3) {$\emptyset$};
\node at (1.5,2.5) {X};
\node at (.5,1.5) {X};
\node at (2.5,1.5) {X};
\node at (1.5,.5) {X};
\tyng(1.9cm,1.9cm,1);
\tyng(2.9cm,1.9cm,1);
\tyng(.9cm,.9cm,1);
\tyng(1.8cm,.9cm,2);
\tyng(2.8cm,.8cm,2,1);
\tyng(.9cm,-.1cm,1);
\tyng(1.8cm,-.2cm,2,1);
\tyng(2.8cm,-.2cm,2,2);
\draw (0.2,3) -- (.8,3);
\draw (1.2,3) -- (1.8,3);
\draw (2.2,3) -- (2.8,3);
\draw (0.2,2) -- (.8,2);
\draw (1.2,2) -- (1.75,2);
\draw (2.25,2) -- (2.75,2);
\draw (0.2,1) -- (.75,1);
\draw (1.25,1) -- (1.7,1);
\draw (2.3,1) -- (2.7,1);
\draw (0.2,0) -- (.75,0);
\draw (1.25,0) -- (1.7,0);
\draw (2.3,0) -- (2.7,0);
\draw (0,.2) -- (0,.8);
\draw (0,1.2) -- (0,1.8);
\draw (0,2.2) -- (0,2.8);
\draw (1,.25) -- (1,.75);
\draw (1,1.25) -- (1,1.8);
\draw (1,2.2) -- (1,2.8);
\draw (2,.3) -- (2,.75);
\draw (2,1.25) -- (2,1.8);
\draw (2,2.2) -- (2,2.8);
\draw (3,.3) -- (3,.7);
\draw (3,1.3) -- (3,1.75);
\draw (3,2.25) -- (3,2.8);

\begin{scope}[xshift=4.5cm]
\node at (0,3) {$\emptyset$};
\node at (0,2) {$\emptyset$};
\node at (1,2) {$\emptyset$};
\node at (0,1) {$\emptyset$};
\node at (0,0) {$\emptyset$};
\node at (1,3) {$\emptyset$};
\node at (2,3) {$\emptyset$};
\node at (3,3) {$\emptyset$};
\node at (1.5,2.5) {X};
\node at (.5,1.5) {X};
\node at (2.5,1.5) {X};
\node at (1.5,.5) {X};
\tyng(1.9cm,1.9cm,1);
\tyng(2.9cm,1.9cm,1);
\tyng(.9cm,.9cm,1);
\tyng(1.9cm,.8cm,1,1);
\tyng(2.9cm,.8cm,2,1);
\tyng(.9cm,-.1cm,1);
\tyng(1.8cm,-.2cm,2,1);
\tyng(2.8cm,-.2cm,2,2);
\draw (0.2,3) -- (.8,3);
\draw (1.2,3) -- (1.8,3);
\draw (2.2,3) -- (2.8,3);
\draw (0.2,2) -- (.8,2);
\draw (1.2,2) -- (1.75,2);
\draw (2.25,2) -- (2.75,2);
\draw (0.2,1) -- (.75,1);
\draw (1.25,1) -- (1.75,1);
\draw (2.25,1) -- (2.7,1);
\draw (0.2,0) -- (.75,0);
\draw (1.25,0) -- (1.7,0);
\draw (2.3,0) -- (2.7,0);
\draw (0,.2) -- (0,.8);
\draw (0,1.2) -- (0,1.8);
\draw (0,2.2) -- (0,2.8);
\draw (1,.25) -- (1,.75);
\draw (1,1.25) -- (1,1.8);
\draw (1,2.2) -- (1,2.8);
\draw (2,.3) -- (2,.7);
\draw (2,1.3) -- (2,1.8);
\draw (2,2.2) -- (2,2.8);
\draw (3,.3) -- (3,.7);
\draw (3,1.3) -- (3,1.75);
\draw (3,2.25) -- (3,2.8);
\end{scope}
\end{tikzpicture}
\end{center}

The probability of the first dual growth is
\begin{multline*}
\mc{P}_{(1),(1)}\bigl(\emptyset \rightarrow (2)\bigr) \cdot 
\mc{P}_{(2),(1)}\bigl((1) \rightarrow (2,1)\bigr) \cdot 
\mc{P}_{(2,1),(2,1)}\bigl((2) \rightarrow (2,2)\bigr) \\
= \frac{q(1-t)}{1-qt} \cdot \frac{qt(1-q)}{1-q^2t} \cdot \frac{q(1-t)(1-q^2t)}{(1+q)(1-qt)^2}
= \frac{q^3t (1-q)(1-t)^2}{(1+q)(1-qt)^3},
\end{multline*}
and for the second growth
\begin{multline*}
\mc{P}_{(1),(1)}\bigl(\emptyset \rightarrow (1,1)\bigr) \cdot 
\mc{P}_{(1),(1,1)}\bigl((1) \rightarrow (2,1)\bigr) \cdot 
\mc{P}_{(2,1),(2,1)}\bigl((1,1) \rightarrow (2,2)\bigr) \\
= \frac{1-q}{1-qt} \cdot
\frac{1-t}{1-qt^2} \cdot
\frac{(1-q)(1-qt^2)}{(1+t)(1-qt)^2}
= \frac{(1-q)^2(1-t)}{(1+t)(1-qt)^3}.
\end{multline*}
By summing both probabilities and taking the Jack limit, we obtain
\[
\frac{\alpha}{2\alpha(1+\alpha)^2},
\]
which is not $0$ and hence different to the probability of obtaining $T$ as the $P$-tableau for $A$.

\begin{proof}[Proof of Theorem~\ref{thm:Jack}]
Let $A^\prime$ be the matrix obtained from $A$ by interchanging the $k$-th and $(k+1)$-st columns. The assertion is immediate if at most one of these columns has a nonzero entry. Denote by $i_0<j_0$ the rows of the $1$ entry of the $k$-th or $(k+1)$-st column respectively.
Define $M=M_{k}(A)$ as the number of $1$ entries strictly below and to the left of the entry $(j_0,k)$. We prove the assertion by induction on $M$. We postpone the proof of the base case and assume that we proved the statement for $M-1$. Let $l<k$ be maximal such that an entry in column $l$ and a row below $j_0$ is equal to $1$. By the induction hypothesis we can consecutively interchange the columns $(l,l+1),(l+1,l+2),\ldots,(k,k+1)$ which has the effect of putting the $l$-th column to the right of former $(k+1)$-st column and reducing $M_k(A)$ by 1. We can now interchange the former $k$-th and $(k+1)$-st column by the induction hypothesis and then consecutively interchange columns to move the former $l$-th column back to its original position. Hence we realized interchanging the columns $k,k+1$ by a series of interchanges where the corresponding $M_i$ is at most $M-1$. We can therefore restrict ourselves to the base case. \medskip

In the following we regard dual growths associated  with $A$ and $A^\prime$ restricted to the $k$-th and $(k+1)$-st columns. We denote by I the case in which the 1 entries are in positions $(i_0,k)$ and $(j_0,k+1)$ and by II the case with 1 entries in positions $(i_0,k+1)$ and $(j_0,k)$, compare to Figure~\ref{fig:Jack matrix}.

\begin{figure}[h]
\begin{center}
\begin{tikzpicture}
\draw[dotted, thick] (-0.2,.6) -- (-0.2,1);
\draw[dotted, thick] (1.2,.6) -- (1.2,1);
\draw[dotted, thick] (2.6,.6) -- (2.6,1);
\node at (-0.2,0.2) {$\mu_0$};
\node at (1.2,0.2) {$\mu_0$};
\node at (2.6,0.2) {$\mu_0$};
\node at (-0.2,-1.2) {$\mu_0$};
\node at (1.2,-1.2) {$\rho_0$};
\node at (2.6,-1.2) {$\rho_0$};
\node at (.5,-.5) {X};
\draw (-0.2,-.1) -- (-0.2,-.9);
\draw (1.2,-.1) -- (1.2,-.9);
\draw (2.6,-.1) -- (2.6,-.9);
\draw (0.1,0.2) --(.9,0.2);
\draw (1.5,0.2) --(2.3,0.2);
\draw (0.1,-1.2) --(.9,-1.2);
\draw (1.5,-1.2) --(2.3,-1.2);
\node at (-0.2,-2.6) {$\mu_1$};
\node at (1.2,-2.6) {$\rho_1$};
\node at (2.6,-2.6) {$\rho_1$};
\draw (-0.2,-1.5) -- (-0.2,-2.3);
\draw (1.2,-1.5) -- (1.2,-2.3);
\draw (2.6,-1.5) -- (2.6,-2.3);
\draw (0.1,-2.6) --(.9,-2.6);
\draw (1.5,-2.6) --(2.3,-2.6);
\draw[dotted, thick] (-0.2,-3) -- (-0.2,-3.6);
\draw[dotted, thick] (1.2,-3) -- (1.2,-3.6);
\draw[dotted, thick] (2.6,-3) -- (2.6,-3.6);

\begin{scope}[yshift=-4.2cm]
\node at (-0.2,0.2) {$\mu$};
\node at (1.2,0.2) {$\rho$};
\node at (2.6,0.2) {$\rho$};
\node at (-0.2,-1.15) {$\la$};
\node at (1.2,-1.2) {$*$};
\node at (2.6,-1.2) {$\nu$};
\node at (1.9,-.5) {X};
\draw (-0.2,-.1) -- (-0.2,-.9);
\draw (1.2,-.1) -- (1.2,-.9);
\draw (2.6,-.1) -- (2.6,-.9);
\draw (0.1,0.2) --(.9,0.2);
\draw (1.5,0.2) --(2.3,0.2);
\draw (0.1,-1.2) --(.9,-1.2);
\draw (1.5,-1.2) --(2.3,-1.2);
\draw[dotted, thick] (-0.2,-1.6) -- (-0.2,-2.2);
\draw[dotted, thick] (1.2,-1.6) -- (1.2,-2.2);
\draw[dotted, thick] (2.6,-1.6) -- (2.6,-2.2);
\node at (-0.2,-2.55) {$\la$};
\node at (1.2,-2.6) {$*$};
\node at (2.6,-2.6) {$\nu$};
\draw (0.1,-2.6) --(.9,-2.6);
\draw (1.5,-2.6) --(2.3,-2.6);
\node at (1.2,-3.4) {I};
\end{scope}

\node at (-2,-.5) {row $i_0$};
\node at (-2,-4.7) {row $j_0$};

\begin{scope}[xshift=6cm]
\draw[dotted, thick] (-0.2,.6) -- (-0.2,1);
\draw[dotted, thick] (1.2,.6) -- (1.2,1);
\draw[dotted, thick] (2.6,.6) -- (2.6,1);
\node at (-0.2,0.2) {$\mu_0$};
\node at (1.2,0.2) {$\mu_0$};
\node at (2.6,0.2) {$\mu_0$};
\node at (-0.2,-1.2) {$\mu_0$};
\node at (1.2,-1.2) {$\mu_0$};
\node at (2.6,-1.2) {$\rho_0$};
\node at (1.9,-.5) {X};
\draw (-0.2,-.1) -- (-0.2,-.9);
\draw (1.2,-.1) -- (1.2,-.9);
\draw (2.6,-.1) -- (2.6,-.9);
\draw (0.1,0.2) --(.9,0.2);
\draw (1.5,0.2) --(2.3,0.2);
\draw (0.1,-1.2) --(.9,-1.2);
\draw (1.5,-1.2) --(2.3,-1.2);
\node at (-0.2,-2.6) {$\mu_1$};
\node at (1.2,-2.6) {$\mu_1$};
\node at (2.6,-2.6) {$\rho_1$};
\draw (-0.2,-1.5) -- (-0.2,-2.3);
\draw (1.2,-1.5) -- (1.2,-2.3);
\draw (2.6,-1.5) -- (2.6,-2.3);
\draw (0.1,-2.6) --(.9,-2.6);
\draw (1.5,-2.6) --(2.3,-2.6);
\draw[dotted, thick] (-0.2,-3) -- (-0.2,-3.6);
\draw[dotted, thick] (1.2,-3) -- (1.2,-3.6);
\draw[dotted, thick] (2.6,-3) -- (2.6,-3.6);

\begin{scope}[yshift=-4.2cm]
\node at (-0.2,0.2) {$\mu$};
\node at (1.2,0.2) {$\mu$};
\node at (2.6,0.2) {$\rho$};
\node at (-0.2,-1.15) {$\la$};
\node at (1.2,-1.2) {$*$};
\node at (2.6,-1.2) {$\nu$};
\node at (.5,-.5) {X};
\draw (-0.2,-.1) -- (-0.2,-.9);
\draw (1.2,-.1) -- (1.2,-.9);
\draw (2.6,-.1) -- (2.6,-.9);
\draw (0.1,0.2) --(.9,0.2);
\draw (1.5,0.2) --(2.3,0.2);
\draw (0.1,-1.2) --(.9,-1.2);
\draw (1.5,-1.2) --(2.3,-1.2);
\draw[dotted, thick] (-0.2,-1.6) -- (-0.2,-2.2);
\draw[dotted, thick] (1.2,-1.6) -- (1.2,-2.2);
\draw[dotted, thick] (2.6,-1.6) -- (2.6,-2.2);
\node at (-0.2,-2.55) {$\la$};
\node at (1.2,-2.6) {$*$};
\node at (2.6,-2.6) {$\nu$};
\draw (0.1,-2.6) --(.9,-2.6);
\draw (1.5,-2.6) --(2.3,-2.6);
\node at (1.2,-3.4) {II};
\end{scope}
\end{scope}
\end{tikzpicture}
\end{center}
\caption{\label{fig:Jack matrix} A schematic of dual growth diagrams associated with $A$ and $A^\prime$ restricted to the $k$-th and $(k+1)$-st columns.}
\end{figure}
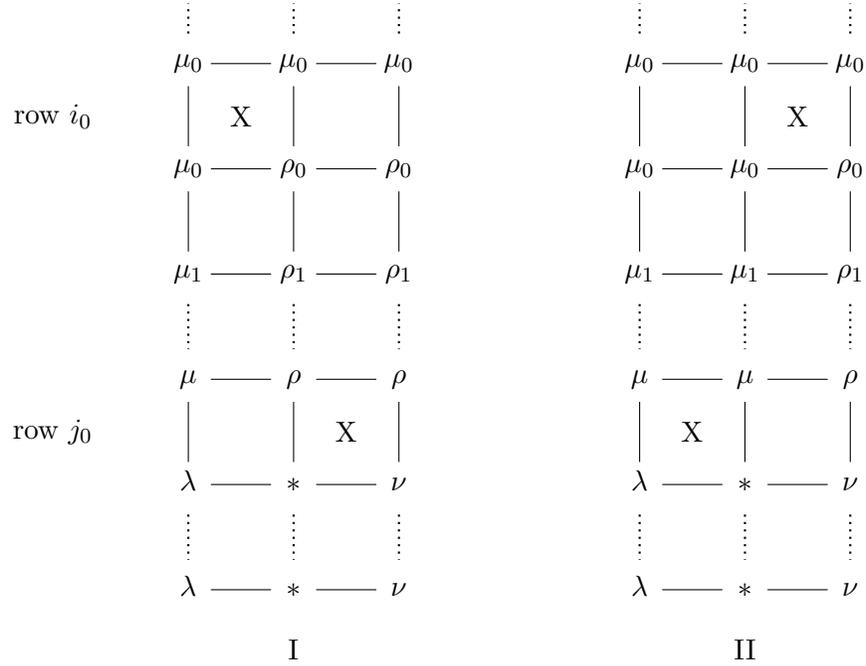

We claim that the probability to obtain the partitions $\mu,\rho,\rho$ at the top of row $j_0$ for a growth diagram in case I is equal to the probability of obtaining the partitions $\mu,\mu,\rho$ in case II. It is immediate that each square above row $i_0$ has probability $1$ in both cases. For row $i_0$ the probabilities of the squares are $\mc{P}_{\mu_0,\mu_0}(\mu_0 \rightarrow \rho_0)$ and $1$ in situation I and $1$ and $\mc{P}_{\mu_0,\mu_0}(\mu_0 \rightarrow \rho_0)$ in situation II. Similarly for the rows between $i_0$ and $j_0$, the probabilities for the squares in case I and II are both $\mc{P}_{\mu_{i+1},\rho_i}(\mu_i \rightarrow \rho_{i+1})$ and $1$, however in different order, which proves the claim.

Finally, since we can assume $A_{i,j}=0$ for all $i>j_0$ and $j<k$, the sequences of partitions for all rows below row $j_0$ are for both case I and II of the form $\la,*,\nu$; the probabilities of the corresponding squares are therefore always $1$.
Hence it suffices to show that for all partitions $\mu \prec \la$ and $\mu \lessdot \rho$
the distributions of $\nu$, with $|\nu / \la|=2$, in both situations I and II are equal, which is shown in Lemma~\ref{lem:Jack Case 1} and Lemma~\ref{lem:Jack Case 2}.
\end{proof}

\section{$qRSt$ for words}
\label{sec:words}
In the final section of this paper we show that restricting $\dRSK$ to $\{0,1\}$-matrices which have at most one $1$ entry in each column yields a generalization of the RS correspondence for words to the Macdonald setting.\bigskip

Let $\mu \subseteq \la$ be two partitions and set\footnote{Due to the different definition of $\mc{C}_{\la/\mu}$ in \cite{AignerFrieden22} the definition of $\vp_{\la/\mu}$ looks different, however it is easy to see that they coincide.}
\[
\vp_{\la/\mu} (q,t) = \prod_{c \in \mc{C}_{\la/\mu}} \frac{b_\la(c)}{b_\mu(c)}.
\]
For a semistandard Young tableau $T$ define
\[
\vp_T(q,t) = \prod_{i \geq 1} \vp_{T^{(i)}/T^{(i-1)}} (q,t).
\]
It was shown by Macdonald \cite[Ch. VI (7.13)]{Mac}, that the Macdonald polynomials $Q_\la(\x;q,t)$ have the following monomial expansions
\begin{equation}
\label{eq:Q via SSYT}
Q_\la(\x;q,t) = \sum_{T \in \SSYT_\la}\vp_T(q,t)\x^T.
\end{equation}
Again, we take the perspective of viewing the above monomial expansion as the definition of the Macdonald polynomials $Q_\la$. The aim of $\qrst$ for words is a combinatorial proof of the squarefree part with respect to $\y$ of the Cauchy identity \eqref{eq_Cauchy_Mac} for Macdonald polynomials. As in the previous sections, we first derive a combinatorial proof for a certain commutation relation which is then used to define the $\qrst$ correspondence for words.

The \deff{restricted $(q,t)$-down operator} $D_{q,t}$ is defined as
\[
D_{q,t} \la = \sum_{\substack{\mu \prec \la \\ |\la/\mu|=1}}\vp_{\la/\mu}(q,t) \; \mu.
\]

\begin{thm}
\label{thm:RS com}
The $(q,t)$-up and restricted $(q,t)$-down operators satisfy the commutation relation
\begin{equation}
\label{eq:RS com}
D_{q,t} U_x(q,t) = U_x(q,t)  D_{q,t} + \frac{1-t}{1-q}x U_x(q,t).
\end{equation}
\end{thm}
\begin{proof}
Define the sets
\begin{align*}
\mc{D}^{(k,l)}(\la,\rho)&=\{\mu \prec \la,\rho: |\la/\mu|=k, |\rho/\mu|=l\}, \\  
\mc{U}^{(k,l)}(\la,\rho)&=\{\nu \succ \la,\rho: |\nu/\rho|=k, |\nu/\la|=l\},
\end{align*}
together with the weights
\[
\omega^{\RS}_{\la,\rho}(\mu)= \psi_{\la / \mu}(q,t) \vp_{\rho /\mu}(q,t) x^{|\la/\mu|}, \qquad
\ov{\omega}^{\RS}_{\la,\rho}(\nu)= \psi_{\nu / \rho}(q,t) \vp_{\nu /\la}(q,t) x^{|\nu/\rho|}.
\]
Then equation \eqref{eq:RS com} is equivalent to the system of equations
\begin{equation}
\label{eq:RS via weighted sets}
\sum_{\nu \in \mc{U}^{(k+1,1)}(\la,\rho)} \ov{\omega}^{\RS}_{\la,\rho}(\nu)
= \sum_{\mu \in \mc{D}^{(k+1,1)}(\la,\rho)} \omega^{\RS}_{\la,\rho}(\mu)
+\frac{1-t}{1-q}x \sum_{\mc{D}^{(k,0)}(\la,\rho)} \omega^{\RS}_{\la,\rho}(\mu),
\end{equation}
for all partitions $\la,\rho$ and nonnegative integers $k$.
Comparing the definitions of $\vp$ and $\vp^*$ we observe for two partitions $\mu \subseteq \rho$ with $|\rho/\mu|\leq 1$
\[
\vp_{\rho/\mu} = \begin{cases}
\vp^*_{\rho/\mu} \quad &  \mu=\rho,\\
\frac{1-t}{1-q}\vp^*_{\rho/\mu} \quad & \mu \lessdot \rho,
\end{cases}
\]
where $\mu \lessdot \rho$ is defined as $\mu \prec \rho$ and $|\rho/\mu|=1$. This implies
\begin{align*}
\omega^{\RS}_{\la,\rho}(\mu) &=\omega_{\la,\rho}(\mu) & \text{if } \mu \in \mc{D}^{(k,0)}(\la,\rho),\\
\omega^{\RS}_{\la,\rho}(\mu) &= \frac{1-t}{1-q} \omega_{\la,\rho}(\mu)
 & \text{if } \mu \in \mc{D}^{(k+1,1)}(\la,\rho),\\
\ov{\omega}^{\RS}_{\la,\rho}(\nu) &= \frac{1-t}{1-q} \ov{\omega}_{\la,\rho}(\nu)
& \text{if }  \nu \in \mc{U}^{(k+1,1)}(\la,\rho),
\end{align*}
By dividing \eqref{eq:RS via weighted sets} through $\frac{1-t}{1-q}$ we obtain a special case of  \eqref{eq: sum of weights} which proves the claim.
\end{proof}

We define $\qrst$ for words by restricting $\dRSK$ to $\{0,1\}$-matrices with at most one $1$ entry in each column. Thereby we obtain the following probabilistic local dual growth rules, where the corresponding probabilities are shown below each case.
\begin{center}
\begin{tikzpicture}[baseline=-3.5ex]
\draw (-0.2,0) -- (-0.2,-1);
\draw (1.2,0) -- (1.2,-1);
\draw (0,0.2) -- (1,0.2);
\draw (0,-1.2) -- (1,-1.2);
\node at (-0.2,0.2) {$\mu$};
\node at (1.2,0.2) {$\mu$};
\node at (-0.2,-1.2) {$\la$};
\node at (1.25,-1.25) {$\la$};
\node at (.5,-.5) {$0$};
\node at (.5,-2) {$1$};


\begin{scope}[xshift=3.5cm]
\node at (.5,-2.5) {where $\mu \lessdot \rho$};
\draw (-0.2,0) -- (-0.2,-1);
\draw (1.2,0) -- (1.2,-1);
\draw (0,0.2) -- (1,0.2);
\draw (0,-1.2) -- (1,-1.2);
\node at (-.2,0.2) {$\mu$};
\node at (1.2,0.2) {$\rho$};
\node at (-0.2,-1.2) {$\la$};
\node at (1.25,-1.25) {$\nu$};
\node at (.5,-.5) {$0$};
\node at (.5,-2) {$\mc{P}_{\la,\rho}(\mu \rightarrow\nu)$};
\end{scope}

\begin{scope}[xshift=7cm]
\draw (-0.2,0) -- (-0.2,-1);
\draw (1.2,0) -- (1.2,-1);
\draw (0,0.2) -- (1,0.2);
\draw (0,-1.2) -- (1,-1.2);
\node at (-0.2,0.2) {$\mu$};
\node at (1.2,0.2) {$\mu$};
\node at (-0.2,-1.2) {$\la$};
\node at (1.25,-1.25) {$\nu$};
\node at (.5,-.5) {$1$};
\node at (.5,-2) {$\mc{P}_{\la,\mu}(\mu \rightarrow \nu)$};
\end{scope}
\end{tikzpicture}
\end{center}

\begin{cor}
\label{cor:qrst is prob bij}
$\qrst$ for words defines a probabilistic bijection between the weighted set of $m \times n$ $\{0,1\}$-matrices with at most one entry equal to $1$ in each column and weight $\omega$ and $\ds \bigsqcup_{\la \subseteq (m^n)} \SSYT(\la) \times \SYT^{\leq n}(\la)$ and weight $\ov{\omega}^{\RS}$, where
\[
\omega(A) = \prod_{\substack{1 \leq i \leq m \\1 \leq j \leq n}}(x_iy_j)^{A_{i,j}} \quad\quad \text{ and } \quad\quad \ov{\omega}^{\RS}(P,Q) = \psi_P(q,t)\vp_Q(q,t)\mb{x}^P\mb{y}^Q.
\]
Moreover, $\qrst$ for words extends $\qrst$ as defined in \cite{AignerFrieden22} and specializes to the row insertion version of RS for words for $q,t\rightarrow 0$ and the column insertion version of RS for words  for $q,t\rightarrow \infty$.
\end{cor}
\begin{proof}
It follows from the proof of Theorem~\ref{thm:RS com}, that $\qrst$ for words is the claimed probabilistic bijection.
Comparing with the growth rules in \cite[\S 4.6]{AignerFrieden22} and the description of the probabilities in \cite[Proposition 4.5.1]{AignerFrieden22} it is immediate that $\qrst$ for words is an extension of $\qrst$ as defined in \cite{AignerFrieden22}.
Finally the claim regarding the limits follows from Corollary~\ref{cor: limits} and the fact that $\RSK^*$ restricted to $\{0,1\}$-matrices with at most one entry equal to $1$ in each row yields RS for words.
\end{proof}

It is immediate, that we obtain the following insertion rules for $\qrst$ for words as a special case of Definition~\ref{def:qtRSK insertion}.

\begin{defn}
\label{def:qrst insertion}
Let $T$ be a semistandard Young tableau, and $i$ a positive integer. The $\qrst$-insertion of $i$ into $T$, denoted by
\[
\widehat{T}=T \xleftarrow{\dRSK} i,
\]
is the probability distribution computed as follows.
\begin{itemize}
\item For each $\nu \in \U^{(k+1,1)}\left(T^{(i)},T^{(i-1)}\right)$, where $k=|T^{(i)}/T^{(i-1)}|$, place $i$ in the cell $\nu/T^{(i)}$ with probability $\mc{P}_{T^{(i)},T^{(i-1)}}\left(T^{(i-1)} \rightarrow \nu\right)$.
\item If an entry $z$ was bumped during the process, reinsert it in the cell $\nu /T^{(z)}$ with probability 
$\mc{P}_{T^{(z)},\widehat{T}^{(z-1)}}\left(T^{(z-1)} \rightarrow \nu\right)$, where $\nu \in \U^{(k,1)}\left(T^{(z)},\widehat{T}^{(z-1)}\right)$ and $k=|T^{(z)}/T^{(z-1)}|$.
\end{itemize}
\end{defn}

\begin{ex}
Inserting $2$ into the semistandard Young tableau \young(1123,2) yields
\begin{align*}
\Yboxdim{0.4cm}
{\small{\young(1123,22)}} \text{ with prob. } &  \mc{P}_{(3,1),(2)}\big((2) \rightarrow (3,2)\big) \\
& =\ds qt\frac{1-q}{1-q^2t},
\\[18pt]
\Yboxdim{0.4cm}
{\small{\young(11223,2)}} \text{ with prob. } &\mc{P}_{(3,1),(2)}\big(((2) \rightarrow (4,1)\big)\cdot\mc{P}_{(4,1),(4,1)}\big((3,1) \rightarrow (5,1)\big) \\
&=\ds q \frac{(1-q^3t^2)^2}{(1+q t)(1-q^3t)(1-q^4t^2)},
\\[18pt]
\Yboxdim{0.4cm}
{\small{\young(1122,23)}} \text{ with prob. } & \mc{P}_{(3,1),(2)}\big((2) \rightarrow (4,1)\big)\cdot\mc{P}_{(4,1),(4,1)}\big((3,1) \rightarrow (4,2)\big) \\
& = \ds \frac{(1-q)(1-t)(1-q^3t^2)}{(1-q^2t)(1-q^2t^2)(1-q^3t)},
 \\[18pt]
\Yboxdim{0.4cm}
{\small{\young(1122,2,3)}} \text{ with prob. } &\mc{P}_{(3,1),(2)}\big((2) \rightarrow (4,1)\big)\cdot\mc{P}_{(4,1),(4,1)}\big((3,1) \rightarrow (4,1,1)\big) \\
& = \ds t \frac{(1-q)^2}{(1-q^2t^2)(1-q^4t^2)}.
\end{align*}

\end{ex}

\begin{appendix}
\section{Two technical lemmas for the Jack limit}
\label{sec:appendix}

\begin{lem}
\label{lem:Jack Case 1}
For partitions $\mu \lessdot \rho \prec \nu$ and $\mu \prec \la \subseteq \nu$ satisfying $\la \cap \rho=\mu$ and $|\nu / \la|=2$ the probability to obtain $\nu$ in the partial growth diagrams I and II are equal. 
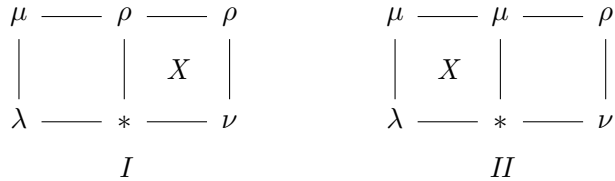
\begin{center}
\begin{tikzpicture}
\node at (-0.2,0.2) {$\mu$};
\node at (1.2,0.2) {$\rho$};
\node at (2.6,0.2) {$\rho$};
\node at (-0.2,-1.15) {$\la$};
\node at (1.2,-1.2) {$*$};
\node at (2.6,-1.2) {$\nu$};
\node at (1.9,-.5) {X};
\draw (-0.2,-.1) -- (-0.2,-.9);
\draw (1.2,-.1) -- (1.2,-.9);
\draw (2.6,-.1) -- (2.6,-.9);
\draw (0.1,0.2) --(.9,0.2);
\draw (1.5,0.2) --(2.3,0.2);
\draw (0.1,-1.2) --(.9,-1.2);
\draw (1.5,-1.2) --(2.3,-1.2);
\node at (1.2,-1.8) {I};

\begin{scope}[xshift=5cm]
\node at (-0.2,0.2) {$\mu$};
\node at (1.2,0.2) {$\mu$};
\node at (2.6,0.2) {$\rho$};
\node at (-0.2,-1.15) {$\la$};
\node at (1.2,-1.2) {$*$};
\node at (2.6,-1.2) {$\nu$};
\node at (.5,-.5) {X};
\draw (-0.2,-.1) -- (-0.2,-.9);
\draw (1.2,-.1) -- (1.2,-.9);
\draw (2.6,-.1) -- (2.6,-.9);
\draw (0.1,0.2) --(.9,0.2);
\draw (1.5,0.2) --(2.3,0.2);
\draw (0.1,-1.2) --(.9,-1.2);
\draw (1.5,-1.2) --(2.3,-1.2);
\node at (1.2,-1.8) {II};
\end{scope}
\end{tikzpicture}
\end{center}
\end{lem}

\begin{proof}
Let $d$ be the number of removable inner corners of  $\mu$ with respect to $(\la,\mu)$ and $R_i,I_i,S_j,O_j$ for $i \in [d]$ and $j \in [0,d]$ be the points defined for $\la,\mu$. Since we regard removable and addable corners for various partitions at the same time, we identify in the following each outer and inner corner with their ``index'' and can therefore omit referring to the shapes they are defined on. The removable inner corners of $\mu$ with respect to $(\la,\mu)$ have indices in $[d]$ and the addable outer corners of $\la$ have indices in $[0,d]$.
Further we use the notation
\[
p\bigl(\Ret,\Ret^c,\Set,\Set^c\bigr) := 
\prod_{s \in \Set} \frac{\prod\limits_{i \in \Ret^c}(S_s -I_i)}{\prod\limits_{j \in \Set^c}(S_s-O_j)}
\prod_{r \in \Ret} \frac{\prod\limits_{j \in \Set^c}(R_r -O_j)}{\prod\limits_{i \in \Ret^c}(R_r-I_i)}.
\]
We regard the following two cases.\medskip

\textbf{Case 1: $|\{\pi: \la \lessdot \pi \lessdot \nu\}|=2$.}\\
There exist two addable outer corners $s_1,s_2 \in [0,d]$ such that $\rho$ is obtained by adding $s_1$ to $\mu$ and $\nu$ is obtained by adding both $s_1$ and $s_2$ to $\la$, see Figure~\ref{fig:Case1 quebecois} for a sketch. Indeed, the only possible partition for $*$ in situation I is the partition $\la \cup \rho$ which implies that $(\la \cup \rho)/\la = \rho/\mu$ is an addable cell $s_1 \in [0,d]$. The condition $|\{\pi: \la \lessdot \pi \lessdot \nu\}|=2$ implies the existence of a partition $\la_2\neq \la \cup \rho$ which satisfies $\la \lessdot \pi \lessdot \nu$. By definition $\la_2$ is therefore obtained by adding an outer corner $s_2\neq s_1$ to $\la$. The conditions of the right square in the growth diagram II imply $\nu=\la_2\cup \rho$ which yields the above claim.

The possible local configurations for I and II are shown next.
\begin{center}
\begin{tikzpicture}
\node at (-0.2,0.2) {$\mu$};
\node at (1.2,0.2) {$\rho$};
\node at (2.6,0.2) {$\rho$};
\node at (-0.2,-1.15) {$\la$};
\node at (1.2,-1.2) {$\la \cup \rho$};
\node at (2.6,-1.2) {$\nu$};
\node at (1.9,-.5) {X};
\draw (-0.2,-.1) -- (-0.2,-.9);
\draw (1.2,-.1) -- (1.2,-.9);
\draw (2.6,-.1) -- (2.6,-.9);
\draw (0.1,0.2) --(.9,0.2);
\draw (1.5,0.2) --(2.3,0.2);
\draw (0.1,-1.2) --(.7,-1.2);
\draw (1.7,-1.2) --(2.3,-1.2);
\node at (1.2,-1.8) {I};

\begin{scope}[xshift=5cm]
\node at (-0.2,0.2) {$\mu$};
\node at (1.2,0.2) {$\mu$};
\node at (2.6,0.2) {$\rho$};
\node at (-0.2,-1.15) {$\la$};
\node at (1.2,-1.2) {$\la \cup \rho$};
\node at (2.6,-1.2) {$\nu$};
\node at (.5,-.5) {X};
\draw (-0.2,-.1) -- (-0.2,-.9);
\draw (1.2,-.1) -- (1.2,-.9);
\draw (2.6,-.1) -- (2.6,-.9);
\draw (0.1,0.2) --(.9,0.2);
\draw (1.5,0.2) --(2.3,0.2);
\draw (0.1,-1.2) --(.7,-1.2);
\draw (1.7,-1.2) --(2.3,-1.2);
\node at (1.2,-1.8) {IIa};
\end{scope}
\begin{scope}[xshift=9cm]
\node at (-0.2,0.2) {$\mu$};
\node at (1.2,0.2) {$\mu$};
\node at (2.6,0.2) {$\rho$};
\node at (-0.2,-1.15) {$\la$};
\node at (1.2,-1.2) {$\la_2$};
\node at (2.6,-1.2) {$\nu$};
\node at (.5,-.5) {X};
\draw (-0.2,-.1) -- (-0.2,-.9);
\draw (1.2,-.1) -- (1.2,-.9);
\draw (2.6,-.1) -- (2.6,-.9);
\draw (0.1,0.2) --(.9,0.2);
\draw (1.5,0.2) --(2.3,0.2);
\draw (0.1,-1.2) --(.9,-1.2);
\draw (1.5,-1.2) --(2.3,-1.2);
\node at (1.2,-1.8) {IIb};
\end{scope}
\end{tikzpicture}
\end{center}

In the following we assume that the cells with index $x^N$ and $x^W$ as shown in Figure~\ref{fig:Case1 quebecois} are both addable outer corners of $(\la \cup \rho)$ with respect to $(\la \cup \rho,\rho)$ and call this ``the generic case''. We show in Remark~\ref{rem:nongeneric case} that we can restrict ourselves to this generic case. The partition $\rho$ has $d+1$ removable corners with respect to $(\la \cup \rho,\rho)$, namely the removable corners with index $i \in [d]$ and the cell $\rho/\mu$ which is denoted with the index $x$. Note that the cell $\rho/\mu$ has two indices: the index $s_1$ when regarded  as an outer corner and the index $x$ when regarded as an inner corner. The addable outer corners of $(\la \cup \rho)$ with respect to $(\la \cup \rho,\rho)$ are the outer corners with index in $[0,d]\setminus \{s_1\}$ and the cells directly above and to the left of $x$ which we denote by $x^N$ and $x^W$.
\begin{figure}[h]
\begin{center}
\begin{tikzpicture}[scale=1.5, xscale=-1]
\tyng(0,0,9,6,6,3,2)
\draw[thick, magenta, pattern=horizontal lines, pattern color=magenta] (5.5,.5) rectangle (4.5,0);
\draw[thick, magenta, pattern=horizontal lines, pattern color=magenta] (1.5,1.5) rectangle (2,2);
\draw[thick, dotted] (4,.5) -- (4,1) -- (3.5,1) -- (3.5,1.5) -- (3,1.5);
\draw[thick, blue, fill=myblue] (3.5,.5) rectangle (3,1);
\draw[thick, blue, fill=myblue] (2.5,1.5) rectangle (2,2);
\draw[line width=1.75pt] (0,0) -- (4.5,0) -- (4.5,.5) -- (3,.5) -- (3,1.5) -- (1.5,1.5) -- (1.5,2) -- (1,2) -- (1,2.5) -- (0,2.5) -- (0,0) -- (0,1);
\node at (3.25,.75) {$s_1$};
\node at (2.25,1.75) {$s_2$};
\node at (3.75,.75) {$x^W$};
\node at (3.25,1.25) {$x^N$};
\draw[fill, blue] (3.5,.5) circle (1.5pt);
\draw[fill, blue] (3,1) circle (1.5pt);
\draw[fill, red] (3,.5) circle (1.5pt);
\node[blue] at (3.75,.25)  {$S_{x^W}$};
\node[blue] at (2.75,1.25)  {$S_{x^N}$};
\node[red] at (2.75,.25)  {$R_{x}$};
\end{tikzpicture}
\end{center}
\caption{\label{fig:Case1 quebecois} A sketch for Lemma~\ref{lem:Jack Case 1}.} 
\end{figure}
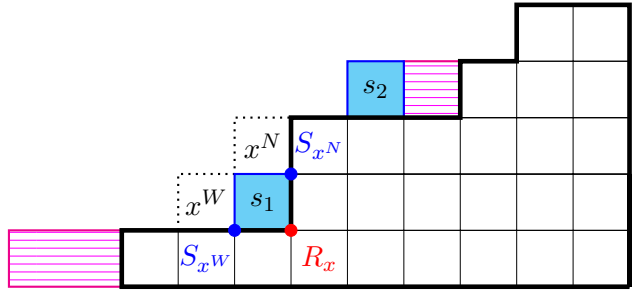
The points corresponding to the newly defined corners satisfy
\[
R_x=S_{s_1}, \quad I_x = \frac{t}{q} S_{s_1}, \quad S_{x^N}=O_{x^N} = t S_{s_1}, \quad S_{x^W}=O_{x^W}=q^{-1}S_{s_1}.
\]
Hence we obtain for the probability of $\nu$ in situation I the expression
\begin{multline}
\label{eq:Case 1a I}
\mc{P}_{\la \cup \rho,\rho}(\rho \rightarrow \nu) = p\Bigl(\emptyset,[d] \cup\{x\},\{s_2\},\{x^N,x^W\} \cup[0,d]\setminus\{s_1,s_2\}\Bigr) \\
= \frac{(S_{s_2}-\frac{t}{q}S_{s_1}) \prod\limits_{i \in [d]} (S_{s_2}-I_i)}
{(S_{s_2}-tS_{s_1})(S_{s_2}-q^{-1}S_{s_1})
\prod\limits_{j \in [0,d]\setminus\{s_1,s_2\}}(S_{s_2}-O_j)}.
\end{multline}

On the other hand, the probability to obtain $\nu$ in situation IIa or IIb is
\begin{multline}
\label{eq:Case 1a II}
\mc{P}_{\la,\mu}(\mu \rightarrow \la \cup \rho)\mc{P}_{\la\cup \rho,\rho}(\mu \rightarrow \nu) 
+ \mc{P}_{\la,\mu}(\mu \rightarrow \la_2)
 \\
= p \Bigl(\emptyset,[d],\{s_1\},[0,d]\setminus\{s_1\} \Bigr)
p \Bigl(\{x\},[d],\{s_2\},\{x^N,x^W\}\cup[0,d]\setminus\{s_1,s_2\} \Bigr) \\
+ p \Bigl(\emptyset,[d],\{s_2\},[0,d]\setminus\{s_2\} \Bigr)
 \\
= \frac{\prod\limits_{i \in [d]}(S_{s_1}-I_i)}{\prod\limits_{j \in [0,d]\setminus\{s_1\}}(S_{s_1}-O_j)} \\
\times \frac{(S_{s_1}-tS_{s_1})(S_{s_1}-q^{-1}S_{s_1})\prod\limits_{i \in [d]}(S_{s_2}-I_i)\prod\limits_{j\in [0,d]\setminus\{s_1,s_2\}}(S_{s_1}-O_j)}{(S_{s_2}-tS_{s_1})(S_{s_2}-q^{-1}S_{s_1})\prod\limits_{j\in [0,d]\setminus\{s_1,s_2\}}(S_{s_2}-O_j)\prod\limits_{i \in [d]}(S_{s_1}-I_i)} \\
+ \frac{\prod\limits_{i \in [d]}(S_{s_2}-I_i)}{\prod\limits_{j \in [0,d]\setminus\{s_2\}}(S_{s_2}-O_j)}
 \\
= \frac{\prod\limits_{i \in [d]}(S_{s_2}-I_i)}{\prod\limits_{j\in [0,d]\setminus\{s_2\}}(S_{s_2}-O_j)}
\left(
1-\frac{(S_{s_1}-tS_{s_1})(S_{s_1}-q^{-1}S_{s_1})}{(S_{s_2}-tS_{s_1})(S_{s_2}-q^{-1}S_{s_1})}
\right).
\end{multline}

By taking the fraction of \eqref{eq:Case 1a I} and \eqref{eq:Case 1a II} we obtain
\begin{equation}
\label{eq:Case 1a frac}
\frac{t S_{s_1}-q S_{s_2}}{S_{s_1}(1-q+q t)-q S_{s_2}},
\end{equation}
which yields $1$ in the Jack limit and hence proves the assertion for this case.\medskip

\textbf{Case 2: $|\{\pi: \la \lessdot \pi \lessdot \nu\}|=1$.}\\
The fact that $\la\cup \rho$ is the only partition $\pi$ satisfying $\la \lessdot \pi \lessdot \nu$ implies that $\nu$ is obtained by adding either the outer corner $x^N$ or $x^W$ to $\la \cup \rho$. In both cases the local configurations for I and II are shown next.
\begin{center}
\begin{tikzpicture}
\node at (-0.2,0.2) {$\mu$};
\node at (1.2,0.2) {$\rho$};
\node at (2.6,0.2) {$\rho$};
\node at (-0.2,-1.15) {$\la$};
\node at (1.2,-1.2) {$\la \cup \rho$};
\node at (2.6,-1.2) {$\nu$};
\node at (1.9,-.5) {X};
\draw (-0.2,-.1) -- (-0.2,-.9);
\draw (1.2,-.1) -- (1.2,-.9);
\draw (2.6,-.1) -- (2.6,-.9);
\draw (0.1,0.2) --(.9,0.2);
\draw (1.5,0.2) --(2.3,0.2);
\draw (0.1,-1.2) --(.7,-1.2);
\draw (1.7,-1.2) --(2.3,-1.2);
\node at (1.2,-1.8) {I};

\begin{scope}[xshift=5cm]
\node at (-0.2,0.2) {$\mu$};
\node at (1.2,0.2) {$\mu$};
\node at (2.6,0.2) {$\rho$};
\node at (-0.2,-1.15) {$\la$};
\node at (1.2,-1.2) {$\la \cup \rho$};
\node at (2.6,-1.2) {$\nu$};
\node at (.5,-.5) {X};
\draw (-0.2,-.1) -- (-0.2,-.9);
\draw (1.2,-.1) -- (1.2,-.9);
\draw (2.6,-.1) -- (2.6,-.9);
\draw (0.1,0.2) --(.9,0.2);
\draw (1.5,0.2) --(2.3,0.2);
\draw (0.1,-1.2) --(.7,-1.2);
\draw (1.7,-1.2) --(2.3,-1.2);
\node at (1.2,-1.8) {II};
\end{scope}
\end{tikzpicture}
\end{center}
If $\nu$ is obtained by adding the cell $x^N$ to $\la \cup \rho$, the probability to obtain $\nu$ in situation I is
\begin{multline}
\label{eq:Case 1b I}
\mc{P}_{\la \cup \rho,\rho}(\rho \rightarrow \nu) 
= p\Bigl(\emptyset,[d] \cup\{x\},\{x^N\},\{x^W\} \cup[0,d]\setminus\{s_1\}\Bigr) 
\\
= \frac{(S_{x^N} - \frac{t}{q}S_1)\prod\limits_{i \in [d]}(S_{x^N}-I_i)}{(S_{x^N}-S_{x^W}) \prod\limits_{j \in [0,d] \setminus \{s_1\}} (S_{x^N}-O_j)},
\end{multline}
and in situation II
\begin{multline}
\label{eq:Case 1b II}
\mc{P}_{\la,\mu}(\mu \rightarrow \la \cup \rho)  \mc{P}_{\la \cup \rho,\rho}(\mu \rightarrow \nu)
\\
= p\Bigl(\emptyset,[d],\{s_1\},[0,d]\setminus\{s_1\}\Bigr) 
p\Bigl(\{x\},[d],\{x^N\},\{x^W\} \cup[0,d]\setminus\{s_1\}\Bigr) 
\\
=
\frac{\prod\limits_{i \in [d]}(S_{s_1}-I_i)}{\prod\limits_{j \in [0,d]\setminus \{s_1\}}(S_{s_1}-O_j)}
\cdot
\frac{(S_{s_1}-O_{x^W})\prod\limits_{i \in [d]}(S_{x^N}-I_i)\prod\limits_{j \in [0,d]\setminus \{s_1\}}(S_{s_1}-O_j)}
{(S_{x^N}-O_{x^W})\prod\limits_{j \in [0,d]\setminus \{s_1\}}(S_{x^N}-O_j)\prod\limits_{i \in [d]}(S_{s_1}-I_i)}
\\
= \frac{(S_{s_1}-O_{x^W})\prod\limits_{i \in [d]}(S_{x^N}-I_i)}
{(S_{x^N}-O_{x^W})\prod\limits_{j \in [0,d]\setminus \{s_1\}}(S_{x^N}-O_j)}.
\end{multline}
If $\nu$ is however obtained by adding $x^W$ to $\la \cup \rho$, we obtain the corresponding probabilities by interchanging $x^N$ and $x^W$ in \eqref{eq:Case 1b I} and \eqref{eq:Case 1b II}.

Dividing \eqref{eq:Case 1b I} by \eqref{eq:Case 1b II}, we obtain after cancellation $t$ or, when replacing first $x^N$ and $x^W$, $q^{-1}$. Taking the Jack limit in both cases yields $1$ and hence finishes the proof.
\end{proof}

\begin{rem}[The non-generic case]
\label{rem:nongeneric case}
Assume in the setting of Case 1 of the above proof that $x^N$ is not an addable cell of $(\la \cup \rho)$, i.e., we are not in the generic case.  This implies that there exists an $i_0$ with $S_{x^N}=I_{i_0}$ such that $i_0$ is not a removable cell of $\rho$.
In this case the probability of $\nu$ in I is given by
\begin{multline}
\label{eq:Case1 I degenerate}
\mc{P}_{\la \cup \rho,\rho}(\rho \rightarrow \nu) = p\Bigl(\emptyset,\{x\} \cup [d]\setminus\{i_0\},\{s_2\},\{x^W\} \cup[0,d]\setminus\{s_1,s_2\}\Bigr) \\
= \frac{(S_{s_2}-\frac{t}{q}S_{s_1}) \prod\limits_{i \in [d]\setminus \{i_0\}} (S_{s_2}-I_i)}
{(S_{s_2}-q^{-1}S_{s_1})
\prod\limits_{j \in [0,d]\setminus\{s_1,s_2\}}(S_{s_2}-O_j)}.
\end{multline}
The right hand sides of \eqref{eq:Case 1a I} and \eqref{eq:Case1 I degenerate} differ by the factor
\[
\frac{S_{s_2}-I_{i_0}}{S_{s_2}-t S_{s_1}}=1,
\]
where the equality follows from $S_{x^N}=t S_{s_1}$.
This observation can be generalized as follows. For two sets $A=\{a_1,\ldots,a_n\} \subseteq \Ret^c$ and $B=\{b_1,\ldots,b_n\} \subseteq \Set^c$ holds
\begin{equation}
\label{eq:comparing to generic case}
\frac{p\bigl(\Ret,\Ret^c\setminus A,\Set,\Set^c \setminus B\bigr)}{p\bigl(\Ret,\Ret^c,\Set,\Set^c\bigr)}
= 
\prod_{\substack{s \in \Set \\ 1 \leq i \leq n}}
\frac{S_s-O_{b_i}}{S_s-I_{a_i}}
\prod_{\substack{r \in \Ret \\ 1 \leq i \leq n}}
\frac{R_r-I_{a_i}}{R_r-O_{b_i}}.
\end{equation}
Comparing to the generic case, a cell which is an outer corner in the generic case with index $b_i$ becomes not addable if it lies directly north and west of a cell which is an inner corner in the generic case with index $a_i$, compare to Figure~\ref{fig:non generic case}. This implies the identity $O_{b_i}=I_{a_i}$.
By choosing $A$ to be the set of inner cells of the generic case, which are not removable and $B$ as the set of outer corners of the generic case which are not addable, the expression in \eqref{eq:comparing to generic case} is equal to $1$. This implies that we may assume the generic case for all calculations.

\begin{figure}
\begin{center}
\begin{tikzpicture}[scale=1.5, xscale=-1]
\tyng(0,0,4,1)
\draw[thick, orange, fill=myorange] (1.5,.5) rectangle (2,0);
\node at (1.75,.25) {$a$};
\draw[thick, blue, fill=myblue] (.5,.5) rectangle (1,1);
\node at (.75,.75) {$b$};

\node[blue] at (.25,.75)  {$O_{b}$};
\node[red] at (2.25,.25)  {$I_{a}$};

\draw[line width=1.75pt] (2,-.2) -- (2,.5) -- (.5,.5) -- (.5,1.2);
\draw[line width=1.75pt, dotted] (2,0) -- (2,-.5);
\draw[line width=1.75pt, dotted] (.5,1) -- (.5,1.5);
\draw[fill, red] (2,.5) circle (1.5pt);
\draw[fill, blue] (.5,.5) circle (1.5pt);

\begin{scope}[xshift=-4cm]
\tyng(0,0,4,4)
\draw[thick, orange, fill=myorange] (1.5,.5) rectangle (2,0);
\node at (1.75,.25) {$a$};
\draw[thick, blue, fill=myblue] (2.5,.5) rectangle (2,1);
\node at (2.25,.75) {$b$};

\node[blue] at (1.75,.75)  {$O_{b}$};
\node[red] at (2.25,.25)  {$I_{a}$};

\draw[line width=1.75pt] (2,-.2) -- (2,1.2);
\draw[line width=1.75pt, dotted] (2,0) -- (2,-.5);
\draw[line width=1.75pt, dotted] (2,1) -- (2,1.5);
\draw[fill, red] (2,.5) circle (1.5pt);
\end{scope}
\end{tikzpicture}

\caption{\label{fig:non generic case} For  the ``generic case'' on the left the cell $b$ is addable and the cell $a$ is removable. The cell $b$ stops being addable and the cell $a$ stops being removable once the row lengths of both rows are equal as shown on the right. In this case  we obtain the identity $O_b=I_a$.}
\end{center}
\end{figure}
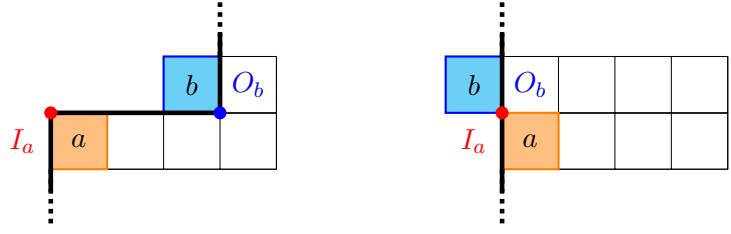
\end{rem}

\begin{lem}
\label{lem:Jack Case 2}
For partitions $\mu \lessdot \rho \prec \la \subseteq \nu$ satisfying $|\nu / \la|=2$ the probability to obtain $\nu$ in the partial growth diagrams I and II are equal. 
\begin{center}
\begin{tikzpicture}
\node at (-0.2,0.2) {$\mu$};
\node at (1.2,0.2) {$\rho$};
\node at (2.6,0.2) {$\rho$};
\node at (-0.2,-1.15) {$\la$};
\node at (1.2,-1.2) {$*$};
\node at (2.6,-1.2) {$\nu$};
\node at (1.9,-.5) {X};
\draw (-0.2,-.1) -- (-0.2,-.9);
\draw (1.2,-.1) -- (1.2,-.9);
\draw (2.6,-.1) -- (2.6,-.9);
\draw (0.1,0.2) --(.9,0.2);
\draw (1.5,0.2) --(2.3,0.2);
\draw (0.1,-1.2) --(.9,-1.2);
\draw (1.5,-1.2) --(2.3,-1.2);
\node at (1.2,-1.8) {I};

\begin{scope}[xshift=5cm]
\node at (-0.2,0.2) {$\mu$};
\node at (1.2,0.2) {$\mu$};
\node at (2.6,0.2) {$\rho$};
\node at (-0.2,-1.15) {$\la$};
\node at (1.2,-1.2) {$*$};
\node at (2.6,-1.2) {$\nu$};
\node at (.5,-.5) {X};
\draw (-0.2,-.1) -- (-0.2,-.9);
\draw (1.2,-.1) -- (1.2,-.9);
\draw (2.6,-.1) -- (2.6,-.9);
\draw (0.1,0.2) --(.9,0.2);
\draw (1.5,0.2) --(2.3,0.2);
\draw (0.1,-1.2) --(.9,-1.2);
\draw (1.5,-1.2) --(2.3,-1.2);
\node at (1.2,-1.8) {II};
\end{scope}
\end{tikzpicture}
\end{center}
\end{lem}

\begin{proof}
As in the proof of Lemma~\ref{lem:Jack Case 1} let $d$ be the number of removable inner corners of $\mu$ with respect to $(\la,\mu)$ and define the points $R_i,I_i,S_j$ and $O_j$ for $\la,\mu$. We regard three different cases where we assume $|\{\pi: \la \lessdot \pi \lessdot \nu\}|=2$ for Case 1 and 2, and $|\{\pi: \la \lessdot \pi \lessdot \nu\}|=1$ in Case 3.\medskip

\textbf{Case 1.}\\
In this case we assume $|\{\pi: \la \lessdot \pi \lessdot \nu\}|=2$ and that $\nu$ is obtained by adding two outer corners with indices $s_1,s_2 \in[0,d]$ to $\la$. We denote by  $\la_i$ the partition obtained by adding $s_i$ to $\la$ for $i =1,2$. See Figure~\ref{fig:Case2 quebecois} for a sketch.
\begin{figure}[h]
\begin{center}
\begin{tikzpicture}[scale=1.5, xscale=-1]
\tyng(0,0,10,8,5,5)
\draw[thick, magenta, pattern=horizontal lines, pattern color=magenta] (6,.5) rectangle (5,0);
\draw[thick, magenta, pattern=horizontal lines, pattern color=magenta] (3,1.5) rectangle (3.5,1);
\draw[thick, magenta, pattern=horizontal lines, pattern color=magenta] (1,2.5) rectangle (0,2);
\draw[thick, dotted] (5,.5) -- (5,1) -- (4.5,1);
\draw[thick, dotted] (3,1.5) -- (3,2) -- (2.5,2);
\draw[thick, dotted] (2,2) -- (2,2.5) -- (1.5,2.5);
\draw[thick, orange, fill=myorange] (2.5,1.5) rectangle (3,1);
\draw[thick, blue, fill=myblue] (4.5,.5) rectangle (4,1);
\draw[thick, blue, fill=myblue] (1.5,2.5) rectangle (1,2);
\node at (4.25,.75) {$s_1$};
\node at (1.25,2.25) {$s_2$};
\node at (2.75,1.25) {$x$};
\node at (2.75,1.75) {$x^N$};
\node at (4.75,.75) {$s_1^W$};
\node at (1.75,2.25) {$s_2^W$};;
\draw[line width=1.75pt] (0,0) -- (5,0) -- (5,.5) -- (4,.5) -- (4,1) -- (2.5,1) -- (2.5,2) -- (0,2) -- (0,0) -- (0,1);
\draw[fill, blue] (5,.5) circle (1.5pt);
\draw[fill, blue] (2.5,1.5) circle (1.5pt);
\draw[fill, blue] (1.5,2) circle (1.5pt);
\draw[fill, red] (2.5,1) circle (1.5pt);
\node[blue] at (4.75,.25)  {$S_{s_1^W}$};
\node[blue] at (2.25,1.25)  {$S_{x^N}$};
\node[blue] at (1.25,1.75)  {$S_{s_2^W}$};
\node[red] at (2.25,.75)  {$R_{x}$};
\end{tikzpicture}
\end{center}
\caption{\label{fig:Case2 quebecois} A sketch for Lemma~\ref{lem:Jack Case 1}.}  
\end{figure}
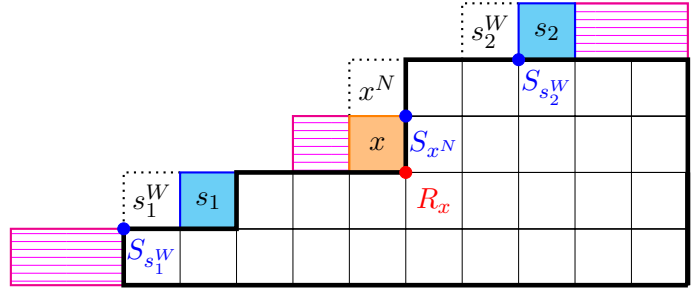
The possible local configurations for I and II are shown next.
\begin{center}
\begin{tikzpicture}[scale=0.8]
\node at (-0.2,0.2) {$\mu$};
\node at (1.2,0.2) {$\rho$};
\node at (2.6,0.2) {$\rho$};
\node at (-0.2,-1.15) {$\la$};
\node at (1.2,-1.2) {$\la_1$};
\node at (2.6,-1.2) {$\nu$};
\node at (1.9,-.5) {X};
\draw (-0.2,-.1) -- (-0.2,-.9);
\draw (1.2,-.1) -- (1.2,-.9);
\draw (2.6,-.1) -- (2.6,-.9);
\draw (0.1,0.2) --(.9,0.2);
\draw (1.5,0.2) --(2.3,0.2);
\draw (0.1,-1.2) --(.9,-1.2);
\draw (1.5,-1.2) --(2.3,-1.2);
\node at (1.2,-1.8) {Ia};

\begin{scope}[xshift=4cm]
\node at (-0.2,0.2) {$\mu$};
\node at (1.2,0.2) {$\rho$};
\node at (2.6,0.2) {$\rho$};
\node at (-0.2,-1.15) {$\la$};
\node at (1.2,-1.2) {$\la_2$};
\node at (2.6,-1.2) {$\nu$};
\node at (1.9,-.5) {X};
\draw (-0.2,-.1) -- (-0.2,-.9);
\draw (1.2,-.1) -- (1.2,-.9);
\draw (2.6,-.1) -- (2.6,-.9);
\draw (0.1,0.2) --(.9,0.2);
\draw (1.5,0.2) --(2.3,0.2);
\draw (0.1,-1.2) --(.9,-1.2);
\draw (1.5,-1.2) --(2.3,-1.2);
\node at (1.2,-1.8) {Ib};
\end{scope}
\begin{scope}[xshift=8cm]
\node at (-0.2,0.2) {$\mu$};
\node at (1.2,0.2) {$\mu$};
\node at (2.6,0.2) {$\rho$};
\node at (-0.2,-1.15) {$\la$};
\node at (1.2,-1.2) {$\la_1$};
\node at (2.6,-1.2) {$\nu$};
\node at (.5,-.5) {X};
\draw (-0.2,-.1) -- (-0.2,-.9);
\draw (1.2,-.1) -- (1.2,-.9);
\draw (2.6,-.1) -- (2.6,-.9);
\draw (0.1,0.2) --(.9,0.2);
\draw (1.5,0.2) --(2.3,0.2);
\draw (0.1,-1.2) --(.9,-1.2);
\draw (1.5,-1.2) --(2.3,-1.2);
\node at (1.2,-1.8) {IIa};
\end{scope}
\begin{scope}[xshift=12cm]
\node at (-0.2,0.2) {$\mu$};
\node at (1.2,0.2) {$\mu$};
\node at (2.6,0.2) {$\rho$};
\node at (-0.2,-1.15) {$\la$};
\node at (1.2,-1.2) {$\la_2$};
\node at (2.6,-1.2) {$\nu$};
\node at (.5,-.5) {X};
\draw (-0.2,-.1) -- (-0.2,-.9);
\draw (1.2,-.1) -- (1.2,-.9);
\draw (2.6,-.1) -- (2.6,-.9);
\draw (0.1,0.2) --(.9,0.2);
\draw (1.5,0.2) --(2.3,0.2);
\draw (0.1,-1.2) --(.9,-1.2);
\draw (1.5,-1.2) --(2.3,-1.2);
\node at (1.2,-1.8) {IIb};
\end{scope}
\end{tikzpicture}
\end{center}
Denote by $x$ the inner cell $\rho/\mu$, by $x^N$ the cell directly above $x$ and by $s_i^W$ the cell on the left of the cell $s_i$ for $i=1,2$.
The according points for these cells satisfy
\[
I_x = \frac{t}{q}R_x, \qquad S_{x^N}=O_{x^N}=tR_x, \qquad S_{s_i^W}=O_{s_i^W}=q^{-1}S_{s_i}, 
\]
for $i=1,2$. Similar to Remark~\ref{rem:nongeneric case} we can assume that $x^N$ and $s_i^W$ are addable outer corners of $\la_i$ with respect to $(\la_i,\rho)$.

The probability to obtain $\nu$ in Ia is
\begin{multline}
\label{eq:Case2 Ia}
\mc{P}_{\la,\rho}(\mu \rightarrow \la_1) \mc{P}_{\la_1,\rho}(\rho \rightarrow\nu)\\
= p\Bigl( \{x\},[d],\{s_1\},\{ x^N \} \cup [0,d]\setminus\{s_1\} \Bigr) \\
\times
p\Bigl( \emptyset,[d] \cup\{x\},\{s_2\},\{ x^N s_1^W\} \cup [0,d]\setminus\{s_1,s_2\} \Bigr)\\
=\frac{(R_x - t R_x) \prod\limits_{i \in [d]}(S_{s_1}-I_i)\prod\limits_{j \in [0,d]\setminus\{ s_1\}}(R_x-O_j)}{(S_{s_1} - t R_x) \prod\limits_{j \in [0,d]\setminus\{ s_1\}}(S_{s_1}-O_j)\prod\limits_{i \in [d]}(R_x-I_i)} \\
\times
\frac{(S_{s_2}-\frac{t}{q}R_x)\prod\limits_{i \in [d]}(S_{s_2}-I_i)}{(S_{s_2}-tR_x)(S_{s_2}-q^{-1}S_{s_1})\prod\limits_{j \in [0,d]\setminus\{ s_1,s_2\}}(S_{s_2}-O_j)} \\
= F \frac{(R_x-O_{s_2})(S_{s_2}-\frac{t}{q}R_x)}{(S_{s_1}-O_{s_2})(S_{s_2}-q^{-1}S_{s_1})},
\end{multline}
where
\begin{multline*}
F= \frac{(R_x-t R_x)}{(S_{s_1}-t R_x)(S_{s_2}-t R_x)} \prod_{i \in[d]}\frac{(S_{s_1}-I_i)(S_{s_2}-I_i)}{(R_{x}-I_i)} \\ \times
\prod_{j \in [0,d]\setminus\{s_1,s_2\}} \frac{(R_x-O_j)}{(S_{s_1}-O_j)(S_{s_2}-O_j)}.
\end{multline*}
We obtain the corresponding probability of $\nu$ in configuration Ib by interchanging $s_1$ and $s_2$ in \eqref{eq:Case2 Ia}.

The probability of $\nu$ in configuration IIa is given by
\begin{multline}
\label{eq:Case2 IIa}
\mc{P}_{\la,\mu}(\mu \rightarrow \la_1)\mc{P}_{\la_1,\rho}(\mu \rightarrow\nu) \\
= p\Bigl(\emptyset , [d] , \{s_1\} , [0,d]\setminus \{s_1\} \Bigr)
p\Bigl( \{x\} , [d] , \{s_2\} , \{ x^N, s_1^W \} \cup [0,d] \setminus \{s_1,s_2\} \Bigr)\\
= \frac{\prod\limits_{i \in [d]}(S_{s_1}-I_i)}{\prod\limits_{j \in [0,d] \setminus \{s_1\}}(S_{s_1}-O_j)} 
\\ \times 
\frac{(R_x-tR_x)(R_x-q^{-1}S_{s_1})\prod\limits_{i \in [d]}(S_{s_2}-I_i)\prod\limits_{j \in [0,d] \setminus \{s_1,s_2\}}(R_x-O_j)}{(S_{s_2}-tR_x)(S_{s_2}-q^{-1}S_{s_1})\prod\limits_{j \in [0,d] \setminus \{s_1,s_2\}}(S_{s_2}-O_j)\prod\limits_{i \in [d]}(R_x-I_i)} \\
=F \cdot \frac{(R_x-q^{-1}S_{s_1})(S_{s_1}-t R_x)}{(S_{s_2}-q^{-1}S_{s_1})(S_{s_1}-S_{s_2})}.
\end{multline}
In order to obtain the probability for $\nu$ in IIb, we interchange $s_1$ and $s_2$ in \eqref{eq:Case2 IIa}.

Dividing the probability of obtaining $\nu$ in I by the probability of obtaining $\nu$ in II yields
\[
\frac{1}{q} + \frac{(1-q)(q S_{s_1}-S_{s_2})(q S_{s_2}-S_{s_1})}
{q\left( (1+q)(R_x^2 qt -S_{s_1}S_{s_2}) + q(S_{s_1}+S_{s_2}) (S_{s_1}+S_{s_2}-(q+t)R_x) \right)}.
\]
It is not difficult to check, that the second fraction in the above expression has Jack limit $0$, which proves this case.\medskip

\textbf{Case 2.}\\
In this case we assume again $|\{\pi: \la \lessdot \pi \lessdot \nu\}|=2$, but $\nu$ is now obtained by adding an outer corner with index $s_1 \in[0,d]$ and $x^N$ to $\la$. Denote by $\la_1$ (resp., $\la^N$) the partition obtained by adding $s_1$ (resp., $x^N$) to $\la$. The possible local configurations are shown next.

\begin{center}
\begin{tikzpicture}
\node at (-0.2,0.2) {$\mu$};
\node at (1.2,0.2) {$\rho$};
\node at (2.6,0.2) {$\rho$};
\node at (-0.2,-1.15) {$\la$};
\node at (1.2,-1.2) {$\la_1$};
\node at (2.6,-1.2) {$\nu$};
\node at (1.9,-.5) {X};
\draw (-0.2,-.1) -- (-0.2,-.9);
\draw (1.2,-.1) -- (1.2,-.9);
\draw (2.6,-.1) -- (2.6,-.9);
\draw (0.1,0.2) --(.9,0.2);
\draw (1.5,0.2) --(2.3,0.2);
\draw (0.1,-1.2) --(.9,-1.2);
\draw (1.5,-1.2) --(2.3,-1.2);
\node at (1.2,-1.8) {Ia};

\begin{scope}[xshift=4cm]
\node at (-0.2,0.2) {$\mu$};
\node at (1.2,0.2) {$\rho$};
\node at (2.6,0.2) {$\rho$};
\node at (-0.2,-1.15) {$\la$};
\node at (1.2,-1.2) {$\la^N$};
\node at (2.6,-1.2) {$\nu$};
\node at (1.9,-.5) {X};
\draw (-0.2,-.1) -- (-0.2,-.9);
\draw (1.2,-.1) -- (1.2,-.9);
\draw (2.6,-.1) -- (2.6,-.9);
\draw (0.1,0.2) --(.9,0.2);
\draw (1.5,0.2) --(2.3,0.2);
\draw (0.1,-1.2) --(.9,-1.2);
\draw (1.5,-1.2) --(2.3,-1.2);
\node at (1.2,-1.8) {Ib};
\end{scope}
\begin{scope}[xshift=9cm]
\node at (-0.2,0.2) {$\mu$};
\node at (1.2,0.2) {$\mu$};
\node at (2.6,0.2) {$\rho$};
\node at (-0.2,-1.15) {$\la$};
\node at (1.2,-1.2) {$\la_1$};
\node at (2.6,-1.2) {$\nu$};
\node at (.5,-.5) {X};
\draw (-0.2,-.1) -- (-0.2,-.9);
\draw (1.2,-.1) -- (1.2,-.9);
\draw (2.6,-.1) -- (2.6,-.9);
\draw (0.1,0.2) --(.9,0.2);
\draw (1.5,0.2) --(2.3,0.2);
\draw (0.1,-1.2) --(.9,-1.2);
\draw (1.5,-1.2) --(2.3,-1.2);
\node at (1.2,-1.8) {II};
\end{scope}
\end{tikzpicture}
\end{center}

The probability to obtain $\nu$ is in situation Ia 
\begin{multline}
\label{eq:Case2b Ia}
\mc{P}_{\la,\rho}(\mu \rightarrow \la_1) \mc{P}_{\la_1,\rho}(\rho \rightarrow\nu)\\
= p\Bigl( \{x\},[d],\{s_1\},\{ x^N \} \cup [0,d]\setminus\{s_1\} \Bigr)
p\Bigl( \emptyset,[d] \cup\{x\},\{x^N\},\{s_1^W\} \cup [0,d]\setminus\{s_1\} \Bigr)\\
=
\frac{(R_x-t R_x)\prod\limits_{i \in [d]}(S_{s_1}-I_i)\prod\limits_{j \in [0,d]\setminus\{s_1\}}(R_x-O_j)}
{(S_{s_1}-tR_x)\prod\limits_{j \in [0,d]\setminus\{s_1\}}(S_{s_1}-O_j)\prod\limits_{i \in [d]}(R_x-I_i)}
\\ \times
\frac{(tR_x-\frac{t}{q} R_x)\prod\limits_{i \in [d]}(tR_x-I_i)}
{(tR_x-q^{-1}S_{s_1})\prod\limits_{j \in [0,d]\setminus\{s_1\}}(t R_x-O_j)},
\end{multline}
in situation Ib
\begin{multline}
\label{eq:Case2b Ib}
\mc{P}_{\la,\rho}(\mu \rightarrow \la^N) \mc{P}_{\la^N,\rho}(\rho \rightarrow\nu)\\
= p\Bigl( \{x\},[d],\{x^N\},[0,d]\Bigr)
p\Bigl( \emptyset,[d],\{s_1\},[0,d]\setminus\{s_1\} \Bigr)\\
= \frac{\prod\limits_{i \in [d]}(tR_x-I_i)\prod\limits_{j \in [0,d]}(R_x-O_j)}
{\prod\limits_{j \in [0,d]}(tR_x-O_j)\prod\limits_{i \in [d]}(R_x-I_i)}
\cdot
\frac{\prod\limits_{i \in [d]}(S_{s_1}-I_i)}{\prod\limits_{j \in [0,d]\setminus\{s_1\}}(S_{s_1}-O_j)},
\end{multline}
and in situation II
\begin{multline}
\label{eq:Case2b II}
\mc{P}_{\la,\mu}(\mu \rightarrow \la_1) \mc{P}_{\la_1,\rho}(\mu \rightarrow\nu) \\
= p\Bigl( \emptyset,[d],\{s_1\}, [0,d]\setminus\{s_1\} \Bigr)
p\Bigl( \{x\},[d],\{x^N\},\{s_1^W\} \cup [0,d]\setminus\{s_1\} \Bigr)\\
=\frac{\prod\limits_{i \in [d]}(S_{s_1}-I_i)}{\prod\limits_{j \in [0,d]\setminus \{s_1\}}(S_{s_1}-O_j)}
\cdot
\frac{(R_x-q^{-1}S_{s_1})\prod\limits_{i \in [d]}(t R_x-I_i)\prod\limits_{j \in [0,d]\setminus \{s_1\}}(R_x-O_j)}
{(t R_x-q^{-1}S_{s_1})\prod\limits_{j \in [0,d]\setminus \{s_1\}}(tR_x-O_j)\prod\limits_{i \in [d]}(R_x-I_i)}.
\end{multline}
By adding the probabilities of $\nu$ in Ia and Ib and dividing by the probability of $\nu$ in II, we obtain
\[
1+\frac{(1-q)(1-t)}{q(1-\frac{S_{s_1}}{q R_x})},
\]
which yields $1$ in the Jack limit and hence proves the claim in this case. \medskip

\textbf{Case 3.}\\
In the last case we assume $|\{\pi: \la \lessdot \pi \lessdot \nu\}|=1$. This implies that there exists an outer corner $s_1\in [0,d]$ such that $\nu$ is obtained by adding $s_1$ and $s_1^W$ to $\la$. Denote again by $\la_1$ the partition obtained by adding $s_1$ to $\la$.
The two possible local configurations are shown next.
\begin{center}
\begin{tikzpicture}
\node at (-0.2,0.2) {$\mu$};
\node at (1.2,0.2) {$\rho$};
\node at (2.6,0.2) {$\rho$};
\node at (-0.2,-1.15) {$\la$};
\node at (1.2,-1.2) {$\la_1$};
\node at (2.6,-1.2) {$\nu$};
\node at (1.9,-.5) {X};
\draw (-0.2,-.1) -- (-0.2,-.9);
\draw (1.2,-.1) -- (1.2,-.9);
\draw (2.6,-.1) -- (2.6,-.9);
\draw (0.1,0.2) --(.9,0.2);
\draw (1.5,0.2) --(2.3,0.2);
\draw (0.1,-1.2) --(.9,-1.2);
\draw (1.5,-1.2) --(2.3,-1.2);
\node at (1.2,-1.8) {I};

\begin{scope}[xshift=5cm]
\node at (-0.2,0.2) {$\mu$};
\node at (1.2,0.2) {$\mu$};
\node at (2.6,0.2) {$\rho$};
\node at (-0.2,-1.15) {$\la$};
\node at (1.2,-1.2) {$\la_1$};
\node at (2.6,-1.2) {$\nu$};
\node at (.5,-.5) {X};
\draw (-0.2,-.1) -- (-0.2,-.9);
\draw (1.2,-.1) -- (1.2,-.9);
\draw (2.6,-.1) -- (2.6,-.9);
\draw (0.1,0.2) --(.9,0.2);
\draw (1.5,0.2) --(2.3,0.2);
\draw (0.1,-1.2) --(.9,-1.2);
\draw (1.5,-1.2) --(2.3,-1.2);
\node at (1.2,-1.8) {II};
\end{scope}
\end{tikzpicture}
\end{center}
The probability of $\nu$ in situation I is 
\begin{multline}
\label{eq:Case3 I}
\mc{P}_{\la,\rho}(\mu \rightarrow \la_1) \mc{P}_{\la_1,\rho}(\rho \rightarrow\nu)
\\=
p\Bigl( \{x\},[d],\{s_1\},\{x^N\} \cup [0,d]\setminus\{s_1\} \Bigr)
p\Bigl( \emptyset,[d]\cup \{x\},\{s_1^W\},\{x^N\} \cup  [0,d]\setminus\{s_1\} \Bigr)
\\=
\frac{(R_x-tR_x)\prod\limits_{i \in [d]}(S_{s_1}-I_i)\prod\limits_{j \in [0,d]\setminus\{s_1\}}(R_x-O_j)}
{(S_{s_1}-t R_x)\prod\limits_{j \in [0,d]\setminus\{s_1\}}(S_{s_1}-O_j)\prod\limits_{i \in [d]}(R_x-I_i)}
\\ \times
\frac{(q^{-1}S_{s_1}-\frac{t}{q}R_x)\prod\limits_{i \in [d]}(q^{-1}S_{s_1}-I_i)}
{(q^{-1}S_{s_1}-t R_x)\prod\limits_{j \in [0,d]\setminus\{s_1\}}(q^{-1}S_{s_1}-O_j)},
\end{multline}
and in situation II
\begin{multline}
\label{eq:Case3 II}
\mc{P}_{\la,\mu}(\mu \rightarrow \la_1) \mc{P}_{\la_1,\rho}(\mu \rightarrow\nu)
\\=
p\Bigl( \emptyset,[d],\{s_1\}, [0,d]\setminus\{s_1\} \Bigr)
p\Bigl( \{x\},[d],\{s_1^W\},\{x^N\} \cup  [0,d]\setminus\{s_1\} \Bigr)
\\=
\frac{\prod\limits_{i \in [d]}(S_{s_1}-I_i)}{\prod\limits_{j \in [0,d]\setminus\{s_1\}}(S_{s_1}-O_j)} 
\\ \times
\frac{(R_x-t R_x)\prod\limits_{i \in [d]}(q^{-1}S_{s_1}-I_i)\prod\limits_{j \in [0,d]\setminus\{s_1\}}(R_x-O_j)}
{(q^{-1}S_{s_1}-t R_x)\prod\limits_{j \in [0,d]\setminus\{s_1\}}(q^{-1}S_{s_1}-O_j)\prod\limits_{i \in [d]}(R_x-I_i)}.
\end{multline}
By dividing \eqref{eq:Case3 I} through \eqref{eq:Case3 II} we obtain $\frac{1}{q}$ which yields $1$ in the Jack limit.
\end{proof}

\end{appendix}

\bibliographystyle{abbrvurl}
\bibliography{dual_qtRSK}

\end{document}